\documentclass[a4paper, reqno, twoside, 12pt, dvipsnames]{amsart}
\usepackage{etex}
\usepackage{amsfonts,amssymb,graphics,amsthm,amsmath}
\usepackage{mathtools,thmtools}
\usepackage{latexsym}
\usepackage[2emode]{psfrag}
\usepackage{mathrsfs}
\usepackage[all]{xy}
\usepackage{enumerate}
\usepackage[inline]{enumitem}
\usepackage[latin1]{inputenc}
\usepackage{lscape}
\usepackage{a4wide}
\usepackage[backref=page]{hyperref}  %
\usepackage{version}
\usepackage{etoolbox}

\usepackage{lmodern}
\usepackage{microtype}
\usepackage{xcolor}

\usepackage{txfonts}

\usepackage[all]{pstricks}
\usepackage{pst-text,pst-node,pst-coil}
\usepackage{tikz}

\usepackage{cancel}
\usepackage[normalem]{ulem}
\usetikzlibrary{arrows,shapes,backgrounds}

\usepgflibrary{arrows}
\usetikzlibrary{topaths}
\usetikzlibrary{er}

\makeatletter
\patchcmd{\@setaddresses}{\indent}{\noindent}{}{}
\patchcmd{\@setaddresses}{\indent}{\noindent}{}{}
\patchcmd{\@setaddresses}{\indent}{\noindent}{}{}
\patchcmd{\@setaddresses}{\indent}{\noindent}{}{}
\makeatother

\makeatletter																																													
\DeclareMathSizes{12}{12}{7}{5}																																				
\makeatother																																													

\newtheorem{theorem}{Theorem}[section]
\newtheorem{proposition}[theorem]{Proposition}

\newtheorem{lemma}[theorem]{Lemma}
\newtheorem{corollary}[theorem]{Corollary}
\newtheorem*{theorem*}{Theorem}
\newtheorem*{proposition*}{Proposition}
\newtheorem{alphthm}{Theorem}

\newtheorem{alphprop}[alphthm]{Proposition}

\theoremstyle{definition}
\newtheorem{definition}[theorem]{Definition}

\newtheorem{example}[theorem]{Example}

\theoremstyle{remark}
\declaretheorem[name=Remark,qed={$\blacktriangle$},sibling=theorem]{remark}

\newtheoremstyle{questions}{}{}{\color{red}}{}{\color{blue}\bfseries}{}{ }{}
\theoremstyle{questions}

\newcommand{\tensor}[1]{\otimes_{#1}} 

\newcommand{\env}[1]{#1\tensor{}#1^{o}}
\newcommand{\Sf}[1]{\mathsf{#1}}

\newcommand{\rmod}[1]{\Sf{Mod}_{#1}}
\newcommand{\lmod}[1]{{}_{#1}\Sf{Mod}}
\newcommand{\bimod}[1]{{}_{#1}{\Sf{ Mod}}_{#1}}
\newcommand{\Bimod}[2]{{}_{#1}{\Sf{Mod}}_{#2}}

\newcommand{\bara}[1]{\overline{#1}}

\newcommand{\op}[1]{#1^{\mathrm{o}}} 

\newcommand{\End}[2]{\operatorname{End}_{#1}\left(#2\right)}

\newcommand{\lcomod}[1]{ {}^{#1}\mathsf{Comod}}
\newcommand{\rcomod}[1]{ \mathsf{Comod}{}^{#1}}

\newcommand{\Ae}{A^{\mathrm{e}}}
\newcommand{\Ao}{A^{\mathrm{o}}}
\newcommand{\LR}[1]{\left\{\underset{}{} #1 \right\}}

\newcommand{\B}[1]{\boldsymbol{#1}}
 
 \newcommand{\id}{\operatorname{Id}}

\newcommand{\Gg}{\mathscr{G}}
\newcommand{\Hh}{\mathscr{H}}

\newcommand{\Nn}{\mathscr{N}}
\newcommand{\Oo}{\mathscr{O}}

\newcommand{\cA}{{\mathcal A}}
\newcommand{\cB}{{\mathcal B}}
\newcommand{\cC}{{\mathcal C}}
\newcommand{\cD}{{\mathcal D}}

\newcommand{\cH}{{\mathcal H}}

\newcommand{\cK}{{\mathcal K}}
\newcommand{\cL}{{\mathcal L}}

\newcommand{\cP}{{\mathcal P}}

\newcommand{\cR}{{\mathcal R}}
\newcommand{\cS}{{\mathcal S}}
\newcommand{\cT}{{\mathcal T}}
\newcommand{\cU}{{\mathcal U}}
\newcommand{\cV}{{\mathcal V}}

\newcommand{\Sscript}[1]{#1} 
\newcommand{\due}[3]{{}_{#2} {#1} {}_{#3}\!\mathop{}}
\newcommand{\rcaction}[1]{\delta_{#1}} 

\newcommand{\ls}[1]{{}_{\Sscript{s}} {#1}}
\newcommand{\lt}[1]{{}_{\Sscript{t}} {#1}}
\newcommand{\rs}[1]{{#1}{}_{\Sscript{s}} }
\newcommand{\rt}[1]{{#1}{}_{\Sscript{t}}}
\newcommand{\tM}[1]{\lt{#1}}
\newcommand{\Mt}[1]{\rt{#1}}
\newcommand{\Mto}[1]{{#1}{}_{\Sscript{\op{t}}}}
\newcommand{\sM}[1]{\ls{{#1}}}
\newcommand{\sMt}[1]{{}_{\Sscript{s}} {#1}{}_{\Sscript{t}}}
\newcommand{\tMs}[1]{{}_{\Sscript{t}} {#1}{}_{\Sscript{s}}}
\newcommand{\sMs}[1]{{}_{\Sscript{s}} {#1}{}_{\Sscript{s}}}
\newcommand{\tMt}[1]{{}_{\Sscript{t}} {#1}{}_{\Sscript{t}}}
\newcommand{\sMto}[1]{{}_{\Sscript{s}} {#1}{}_{\Sscript{\op{t}}}}
\newcommand{\soMt}[1]{{}_{\Sscript{\op{s}}} {#1}{}_{\Sscript{t}}}
\newcommand{\tMso}[1]{{}_{\Sscript{t}} {#1}{}_{\Sscript{\op{s}}}}
\newcommand{\toMs}[1]{{}_{\Sscript{\op{t}}} {#1}{}_{\Sscript{s}}}
\newcommand{\NN}{\mathbb{N}}
\newcommand{\ZZ}{\mathbb{Z}}

\newcommand{\K}{\Bbbk}

\newcommand{\mf}[1]{\mathfrak{#1}} 

\newcommand{\coinv}[2]{{#1}^{{\sf co}{#2}}} 

\newcommand{\ot}{\otimes}
\newcommand{\tak}[1]{\times_{\Sscript{#1}}}
\newcommand{\wla}{\triangleright} 
\newcommand{\wra}{\triangleleft} 
\newcommand{\bla}{~{\raisebox{+1pt}{$\scriptstyle \blacktriangleright$}}~} 
\newcommand{\bra}{~{\raisebox{+1pt}{$\scriptstyle \blacktriangleleft$}}~} 

\newcommand{\cl}[1]{\overline{#1}}
\newcommand{\vectk}{\mathsf{Vect}_{\K}}

\newcommand{\LHopf}[2]{\prescript{#1}{#2}{\rm HopfMod}}

\newcommand{\crho}{\mathfrak{r}}
\newcommand{\clambda}{\mathfrak{l}}
\newcommand{\calpha}{\mathfrak{a}}

\newcommand{\img}{\operatorname{Im}}
\newcommand{\ad}{\operatorname{Ad}}
\newcommand{\calg}{\mathrm{CAlg}}
\newcommand{\aff}{\mathrm{Aff}}
\newcommand{\set}{\mathrm{Set}}
\newcommand{\iso}[1]{{#1}_{(i)}}
\newcommand{\lcoinv}[2]{{{}^{#2}{#1}}}
\newcommand{\rcoinv}[2]{{{#1}{}^{#2}}}

\newdir{ >}{{}*!/-10pt/@{>}}

\usepackage[textsize=scriptsize]{todonotes}

\begin{document}
\allowdisplaybreaks

\title[]{Correspondence theorems for Hopf algebroids with applications to affine groupoids.
}

\author{Laiachi El Kaoutit}
\address{Universidad de Granada, Departamento de \'{A}lgebra and IMAG. Facultad de Ciencias. Fuente Nueva s/n. E810071 Granada, Spain}
\email{kaoutit@ugr.es}
\urladdr{\url{http://www.ugr.es/~kaoutit/}}

\author{Aryan Ghobadi}
\address{School of Mathematics, Queen Mary University of London, Mile End Road, E1 4NS, London, UK}
\email{a.ghobadi@qmul.ac.uk}
\urladdr{\url{https://sites.google.com/view/aghobadimath}}

\author{Paolo Saracco}
\address{D\'epartement de Math\'ematique, Universit\'e Libre de Bruxelles, Boulevard du Triomphe, B-1050 Brussels, Belgium.}
\email{paolo.saracco@ulb.be}
\urladdr{\url{https://sites.google.com/view/paolo-saracco}}
\urladdr{\url{https://paolo.saracco.web.ulb.be}}

\author{Joost Vercruysse}
\address{D\'epartement de Math\'ematique, Universit\'e Libre de Bruxelles, Boulevard du Triomphe, B-1050 Brussels, Belgium.}
\email{joost.vercruysse@ulb.be}
\urladdr{\url{http://joost.vercruysse.web.ulb.be}}

\date{\today}
\subjclass[2020]{Primary  ; Secondary }
\thanks{}

\begin{abstract}
We provide a correspondence between one-sided coideal subrings and one-sided ideal two-sided coideals in an arbitrary bialgebroid. We prove that, under some expected additional conditions, this correspondence becomes bijective for Hopf algebroids. As an application, we investigate normal Hopf ideals in commutative Hopf algebroids (affine groupoid schemes) in connection with the study of normal affine subgroupoids.
\end{abstract}

\maketitle

\begin{small}
\tableofcontents
\end{small}


\pagestyle{headings}


\section*{Introduction}

\subsection*{Motivations and overviews}
One of the fundamental results in the theory of affine groups states that the category of affine abelian groups is an abelian category, where epimorphisms are faithfully flat morphisms and monomorphism are closed immersions \cite[III, \S3, n$^\circ$7, 7.4 Corollaire, page 355]{DemGab:GATIGAGGC}.  A purely algebraic proof of this theorem was given by M.~Takeuchi in \cite[Corollary 4.16]{Takeuchi:1972}, which asserts, in algebraic terms, that the category of commutative and cocommutative Hopf algebras over a field $\Bbbk$ is an abelian category, hereby generalizing a well-known result from Grothendieck, who proved the same result under the additional condition of finite dimensionality. The main ingredient in Takeuchi's proof is a Galois-type one-to-one correspondence between all sub-Hopf algebras and all normal Hopf ideals of a given commutative Hopf algebra. This bijection associates any sub-Hopf algebra with the ideal generated by the kernel of its counit. It is noteworthy that injectivity follows from the fact that any Hopf algebra is faithfully flat over its arbitrary sub-Hopf algebras \cite[Theorem 3.1]{Takeuchi:1972}. The same correspondence can be found in \cite[\S 4, page 201]{Abe}, although with a slightly different proof. This correspondence does not rely on commutativity or cocommutativity of the Hopf algebras, and allowed for example to an extension of Takeuchi's result showing that the category of cocommutative (but not necessarily commutative) Hopf algebras is semi-abelian \cite{GSV}

The aim of this research is to extend the aforementioned correspondence to the ``multi-object'' setting, that is, from (affine) groups to (affine) groupoids. In Hopf algebraic terms, groupoids can be described as weak Hopf algebras or Hopf algebroids.
We will therefore firstly provide a Galois correspondence between (certain classes of) Hopf ideals and sub Hopf algebroids of general Hopf algebroids (in fact, even bialgebroids) and specializing to the commutative case, we obtain results for affine groupoids. Similar to the classical case explained above, the motivation of this work is that it could lead ultimately to a better understanding of the exactness properties of the category of (abelian) affine $\Bbbk$-groupoids. 
Furthermore, our results have applications on a wide verity of topics, since the interest in Hopf algebroids is widespread in different branches of mathematics: from algebraic topology, algebraic geometry and differential geometry (see \cite{Ravenel:1986, Deligne:1990, LaiachiPaolo-big}), to the study of linear differential equations (see \cite{LaiachiPaolo-Dlin, LaiachiGomez}), noncommutative differential calculus \cite{ghobadi2020} and in the study of the fundamental groupoid of quivers \cite{ghobadi2021isotopy}, to mention only a few.

\subsection*{Description of the main results}

Recall that a \emph{left Hopf algebroid} over $\Bbbk$ (in the sense of Schauenburg \cite{Schauenburg}) is a pair $(A, \cH)$ of $\Bbbk$-algebras together with a family of structure maps that make of $\cH$ an $\env{A}$-ring and an $A$-coring in a compatible way (see Definition \ref{def:bialgebroid}) and whose Hopf-Galois map (described in equation \eqref{Eq:betamap} below) is bijective. 
The so-called \emph{translation map} $\gamma:\cH\to \cH\times_{A^o}\cH$ is obtained from the inverse of the Hopf-Galois map and it is explicitly defined in equation \eqref{Eq:gammamap} below.  
For a ring $B$ over which a Hopf algebroid $(A, \cH)$ coacts, we will denote the subring of coinvariant elements by $\coinv{B}{\cH}$, and for any subset $\cK$ of $\cH$, the symbol $\cK^{+}$ stands for the intersection of the kernel of the counit of $\cH$ with $\cK$. Finally, recall that a functor $\cL \colon \cC \to \cD$ admitting a right adjoint $\cR$ is said to be comonadic if the comparison functor $\cK \colon \cC \to \cD^{\cL\cR}$ to the Eilenberg-Moore category of coalgebras for the comonad $\cL\cR$ is an equivalence.

With these notations at hand, our first main result in the general context of noncommutative Hopf algebroids, is the following.
The functor $\cH \tensor{B} -$ in the statement below denotes the extension of scalars functor ${_B{\sf Mod}}\to{_{\cH}{\sf Mod}}$.

\begin{alphthm}[{Theorem \ref{mainthm:Galois}}]\label{thmA}
Let $(A,\cH)$ be a left Hopf algebroid over $\Bbbk$ such that $\sM{\cH} = {}_{{A\tensor{}\op{1}}}\cH$ is $A$-flat. We have a well-defined inclusion-preserving bijective correspondence
\[
\begin{gathered}
\xymatrix @R=0pt{
{\left\{ \begin{array}{c} \text{left ideal } 2\text{-sided coideals } I \text{ in } \cH \\ \text{such that } \cH \tensor{B}- \text{ is comonadic,} \\ \text{where } B \coloneqq \coinv{\cH}{\frac{\cH}{I}}
\end{array} \right\}} \ar@{<->}[r] & {\left\{ \begin{array}{c} \text{right } \cH\text{-comodule } \op{A}\text{-subrings } B \text{ of } \cH \\   \text{via } t \text{ such that } \cH \tensor{B} - \text{ is comonadic} \\ \text{and } \gamma(B)\subseteq B \tensor{\op{A}}\cH \end{array} \right\}} \\
I \ar@{|->}[r] & \coinv{\cH}{\frac{\cH}{I}} \\
\cH B^+ & B \ar@{|->}[l]
}
\end{gathered}
\]
\end{alphthm}

When specialized to the commutative setting, it induces our second main theorem.

\begin{alphthm}[Theorem \ref{thm:forsecisiamo}]\label{thmB}
Let $(A,\cH)$ be a commutative Hopf algebroid such that $\sM{\cH}$ is $A$-flat. Then we have a well-defined inclusion-preserving bijective correspondence
\[
\begin{gathered}
\xymatrix @R=0pt{
{\left\{ \begin{array}{c}  \text{normal Hopf ideals } I \text{ in } \cH \text{ such that} \\ \cH \text{ is pure over } \coinv{\cH}{\frac{\cH}{I}} \end{array} \right\}} \ar@{<->}[r] & {\left\{ \begin{array}{c} \text{sub-Hopf algebroids } \cK \subseteq \cH \text{ such that } \\ \cH \text{ is pure over } \cK  \end{array} \right\}} \\
I \ar@{|->}[r] & \coinv{\cH}{\frac{\cH}{I}} \\
\cH \cK^+ & \cK \ar@{|->}[l]
}
\end{gathered}
\]
\end{alphthm}

Given a commutative $\Bbbk$-algebra $A$, we denote by $\Gg_A$ the associated $\Bbbk$-functor, that is, the functor from all commutative $\Bbbk$-algebras to sets that sends any algebra $C$ the set of all algebra maps from $A$ to $C$.  In this way, a commutative Hopf algebroid $(A,\cH)$ gives rise to the pair of $\Bbbk$-functors $(\Gg_A, \Gg_{\cH})$ which form a presheaf of groupoids (i.e., an affine $\Bbbk$-groupoid scheme). 
If, as a working terminology, we say that a subgroupoid $\left(\Gg_A,\Gg_{\frac{\cH}{I}}\right)$ of a given groupoid $\big(\Gg_A,\Gg_\cH\big)$ is pure whenever the extension $\coinv{\cH}{\frac{\cH}{I}} \subseteq \cH$ is pure, then Theorem \ref{thmB} can be rephrased as follows.

\begin{alphprop}\label{propC}
Let $(A,\cH)$ be a commutative Hopf algebroid such that $\sM{\cH}$ is $A$-flat. Then we have an inclusion-preserving bijective correspondence
\[
\begin{gathered}
\xymatrix @R=0pt{
{\left\{ \begin{array}{c}  \text{pure normal affine subgroupoids} \\  \left(\Gg_A,\Gg_{\frac{\cH}{I}}\right) \text{ of } \big(\Gg_A,\Gg_\cH\big) \end{array} \right\}} \ar@{<->}[r] & {\left\{ \begin{array}{c} \text{sub-Hopf algebroids } (A,\cK) \text{ of } (A,\cH) \\ \text{ such that } \cK \subseteq \cH \text{ is pure}  \end{array} \right\}} \\
\left(\Gg_A,\Gg_{\frac{\cH}{I}}\right) \ar@{|->}[r] & \coinv{\cH}{\frac{\cH}{I}} \\
\left(\Gg_A,\Gg_{\frac{\cH}{\cH\cK^+}}\right) & \cK \ar@{|->}[l]
}
\end{gathered}
\]
\end{alphprop}

A concrete application of the above to the Hopf algebroid of functions on a finite groupoid (i.e., with a finite set of arrows) is detailed in Example \ref{ex:finitegroupoidexample}, where we show that there is a bijective correspondence between normal subgroupoids of a finite groupoid and pure sub-Hopf algebroids of the associated Hopf algebroid of functions.

As a consequence of our main theorem, we have the following remarkable result.

\begin{alphprop}[Corollaries \ref{cor:Takeuchi3.12} and \ref{cor:serve?}]\label{propD}
Let $\K$ be an algebraically closed field and let $(A,\cH)$ be a commutative Hopf algebroid over $\Bbbk$ such that $\sM{\cH}$ is $A$-flat. 
\begin{enumerate}[label=(\alph*),ref=\emph{(\alph*)}]
\item If $\big(\Gg_A(\K),\Gg_{\cH/I}(\K)\big)$ is a normal subgroupoid of $\big(\Gg_A(\K),\Gg_\cH(\K)\big)$ such that $\cH$ is faithfully flat over $\coinv{\cH}{\frac{\cH}{I}}$, then 
\[\Gg_\cH(\K) / \Gg_{\cH/I}(\K) \ \cong \ \Gg_{\coinv{\cH}{({\cH}/{I})}}(\K).\]
\item\label{item:PropCb} If $(A,\cK)$ is a sub-Hopf algebroid of the commutative Hopf algebroid $(A,\cH)$ such that $\cH$ is faithfully flat over $\cK$, then 
\[\Gg_\cH(\K)/\Gg_{\cH/\cH\cK^+}(\K) \ \cong \ \Gg_{\cK}(\K).\]
\end{enumerate}
\end{alphprop}

This result generalizes corresponding known results for affine algebraic groups, compare for example with \cite[III, \S 3, n$^\circ$7, 7.2, Théorème, page 353 and 7.3 Corollaire, page 354]{DemGab:GATIGAGGC}.


\section{Preliminaries, notation and first results}\label{sec:prelim}

We work over a ground field $\K$. All vector spaces, algebras and coalgebras will be over $\K$. The unadorned tensor product $\tensor{}$ stands for $\tensor{\K}$. By a ring we always mean a ring with identity element and modules over rings or algebras are always assumed to be unital.  All over the paper, we assume a certain familiarity of the reader with the language of monoidal categories and of (co)monoids therein (see, for example, \cite[VII]{MacLane}).

We begin by collecting some facts about bimodules, (co)rings and bialgebroids that will be needed in the sequel. The aim is that of keeping the exposition self-contained. Many results and definitions we will present herein hold in a more general context and under less restrictive hypotheses, but we preferred to limit ourselves to the essentials.

Given a (preferably, non-commutative) $\K$-algebra $A$, the category of $A$-bimodules forms a non-strict monoidal category $\left(\Bimod{A}{A},\tensor{A},A,\calpha,\clambda,\crho\right)$. 
Nevertheless, all over the paper we will behave as if the structural natural isomorphisms
\begin{gather*}
\calpha_{M,N,P} \colon (M\tensor{A}N)\tensor{A}P \to M\tensor{A}(N\tensor{A}P), \quad (m\tensor{A}n)\tensor{A}p \to m\tensor{A}(n\tensor{A}p), \\
\clambda_M \colon A \tensor{A} M \to M, \quad a\tensor{A}m \mapsto a\cdot m, \qquad \text{and} \\  
\crho_M \colon M\tensor{A}A \to M, \quad m \tensor{A} a \mapsto m \cdot a,
\end{gather*}
were ``the identities'', that is, as if $\Bimod{A}{A}$ was a strict monoidal category.
If $A$ is a non-commutative algebra, then we denote by $\Ao$ its opposite algebra. In this case, an element $a \in A$ may be denoted by $\op{a}$ when it is helpful to stress that it is viewed as an element in $\op{A}$. We freely use the canonical isomorphism between the category of left $A$-module $\lmod{A}$ and that of right $\Ao$-modules $\rmod{\Ao}$.  


\subsection{The enveloping algebra and bimodules}
Let $A$ be an algebra, we denote by $\Ae \coloneqq A\tensor{}\Ao$ the enveloping algebra of $A$ and we identify the category $\lmod{\Ae}$ of left $\Ae$-modules with the category $\bimod{A}$ of $A$-bimodules. Giving a morphism of algebras $\Ae \to R$ is equivalent to providing two commuting algebra maps $s\colon  A \to R$ and $t\colon \Ao \to R$, called the source and the target, respectively. In particular, the identity of $\Ae$ gives rise to $s\colon  A \to \Ae, a \mapsto a \ot \op{1}$ and $t\colon \Ao \to \Ae, \op{a} \mapsto 1 \ot \op{a}$.

Given two $\Ae$-bimodules $M$ and $N$, there are several $A$-bimodule structures underlying $M$ and $N$. Namely,
\[(a \ot \op{b}) \cdot m \cdot (c \ot \op{d}) = s(a)t(\op{b}) m s(c)t(\op{d})\]
for all $m \in M$, $a,b,c,d\in A$. This leads to several ways of considering the tensor product over $A$ between these underlying $A$-bimodules. 
For the sake of clarity, we will adopt the following notations. Given an $\Ae$-bimodule $M$, the $A$-action by elements of the form $a\tensor{}\op{1}$ will be denoted by $\ls{M}$ or $\rs{M}$ and the $\op{A}$-action of the element of the form $1\tensor{}\op{a}$ by $\lt{M}$ or $\rt{M}$, depending on which side we are putting those elements. 
We still have other actions by the elements $\op{t(a)}$, or $\op{s(a)}$ to produce other $A$ or $\op{A}$ actions. In this case we use the notation $\op{t}$ and $\op{s}$ located in the corresponding side on which  we are declaring the action. Summing up, we denote
\begin{gather*}
\ls{M}  \coloneqq  {}_{{A\tensor{}\op{1}}}M, 
\quad 
\lt{M}  \coloneqq  {}_{{1\tensor{}\op{A}}}M, 
\quad 
\rt{M}  \coloneqq  {M}_{{1\tensor{}\op{A}}}, 
\quad  
\rs{M}  \coloneqq  {M}_{{A\tensor{}\op{1}}}, 
\\ 
\sMs{M} \coloneqq  {}_{{A\tensor{}\op{1}}}{M}_{{A\tensor{}\op{1}}}, 
\quad 
\sMt{M} \coloneqq  {}_{{A\tensor{}\op{1}}}{M}_{{1\tensor{}\op{A}}}, 
\quad 
\tMs{M} \coloneqq  {}_{{1\tensor{}\op{A}}}{M}_{{A\tensor{}\op{1}}}, 
\quad   
\tMt{M} \coloneqq  {}_{{1\tensor{}\op{A}}}{M}_{{1\tensor{}\op{A}}}, 
\end{gather*}
from which we can switch the given actions and obtain new two-sided actions involving $A$ and $\Ao$:
\[
\underset{A\text{-bimodule}}{\sMto{M} 			\coloneqq		{}_{{A\tensor{}\op{A}}}{M}}, 
\quad 
\underset{\op{A}\text{-bimodule}}{\tMso{M} 	\coloneqq 	{}_{{A\tensor{}\op{A}}}{M}}, 
\quad  
\underset{\op{A}\text{-bimodule}}{\soMt{M} 	\coloneqq		{M}_{{A\tensor{}\op{A}}}}, 
\quad
\underset{A\text{-bimodule}}{\toMs{M} 			\coloneqq		{M}_{{A\tensor{}\op{A}}}},
\]
and so on. For example, $\Mto{M} = \tM{M}$. For the sake of simplicity, we may also often resort to the following variation of the previous conventions: for $a,b \in A$, $m \in M$
\begin{equation}\label{eq:triangleactions}
a \wla m \wra b \coloneqq (a \otimes \op{b})\cdot m \qquad \text{and} \qquad a \bla m \bra b \coloneqq m\cdot (b \otimes \op{a}).
\end{equation}


\subsection{Pure extensions of rings}\label{ssec:purity}
Purity conditions will play a crucial role in establishing our main theorems. Therefore, we devote this subsection to collect a few results in details in this regard, for the convenience of the reader.

Given a ring $R$, we recall from \cite[chap.~I, \S 2, Ex.~24, p.~66]{Bou} that a morphism of left (resp. right) $R$-modules $f\colon M \to N$ is called \emph{pure} (or \emph{universally injective}) if and only if for every right (resp. left) $R$-module $P$, the morphism of abelian groups $P \tensor{R} f$ (resp. $f \tensor{R} P$) is injective. In particular, $f$ itself is injective, by taking $P = R$.

The subsequent Proposition \ref{prop:equivpure} (and its corollaries) should be well-known: it states that a ring extension $f \colon \cA \to \cB$ is pure as a morphism of right $\cA$-modules if and only if the extension-of-scalars functor $\cB \tensor{\cA} -$ is faithful. They will be of great help in what follows.

\begin{proposition}\label{prop:equivpure}
Let $R$ be a ring and let $f \colon \cA \to \cB$ be a morphism of $R$-rings. The following are equivalent.
\begin{enumerate}[leftmargin=0.7cm,ref={\it (\arabic*)}]
\item\label{item:equivpure1} $f \colon \cA \to \cB$ is a pure morphism of right $\cA$-modules;
\item\label{item:equivpure2} for every left $\cA$-module $M$, the morphism $\varrho_M\colon M \to \cB \tensor{\cA} M$, $m \mapsto 1_\cB \tensor{\cA} m$, is injective for all $m \in M$;
\item\label{item:equivpure2bis} for every morphism of left $\cA$-modules $g \colon M \to N$, $\cB \tensor{\cA} g = 0$ implies $g = 0$;
\item\label{item:equivpure4} if $M \xrightarrow{g} N \xrightarrow{h} P$ are morphisms of left $\cA$-modules such that
\[\cB \tensor{\cA} M \xrightarrow{\cB \tensor{\cA} g} \cB \tensor{\cA} N \xrightarrow{\cB \tensor{\cA} h} \cB \tensor{\cA} P\]
is an exact sequence of $\cB$-modules, then $M \xrightarrow{g} N \xrightarrow{h} P$ is an exact sequence of $\cA$-modules.
\item\label{item:equivpure3} if $g \colon M \to N$ is a morphism of left $\cA$-modules such that $\cB \tensor{\cA} g \colon \cB \tensor{\cA} M \to \cB \tensor{\cA} N$ is injective, then $g$ itself is injective.
\end{enumerate}
In addition, any of the foregoing entails that
\begin{enumerate}[leftmargin=0.7cm,ref={\it (\arabic*)},resume]
\item\label{item:equivpure6} for every left ideal $\mf{m}$ in $\cA$, we have $\mf{m} = f^{-1}(\cB\mf{m})$;
\end{enumerate}
and hence, in particular, that
\begin{enumerate}[leftmargin=0.7cm,ref={\it (\arabic*)},resume]
\item\label{item:equivpure7} for every maximal left ideal $\mf{m}$ in $\cA$, there exists a maximal left ideal $\mf{n}$ in $\cB$ such that $\mf{m} = f^{-1}(\mf{n})$.
\end{enumerate}
\end{proposition}

\begin{proof}
The equivalence between \ref{item:equivpure1} and \ref{item:equivpure2} is the definition of purity. The equivalence between \ref{item:equivpure2bis} and \ref{item:equivpure4} is the fact that an additive covariant functor between abelian categories is faithful if and only if it reflects exact sequences (see \cite[\S3.1, exercise 4]{Popescu}). 
The implication from \ref{item:equivpure2} to \ref{item:equivpure2bis} follows by commutativity of the diagram
\[
\xymatrix @R=15pt @C=30pt{
M \ar[r]^-{\varrho_M} \ar[d]_-{g} & \cB \tensor{\cA} M \ar[d]^-{\cB \tensor{\cA} g} \\
N \ar[r]_-{\varrho_N} & \cB \tensor{\cA} N.
}
\]
To prove that \ref{item:equivpure4} implies \ref{item:equivpure3}, consider the morphisms $0 \to M \xrightarrow{g} N$ and apply the functor $\cB \tensor{\cA} -$.
The implication from \ref{item:equivpure3} to \ref{item:equivpure2} follows by considering the morphism
\[\cB \tensor{\cA}\varrho_M \colon \cB \tensor{\cA} M \to \cB \tensor{\cA} \cB \tensor{\cA} M, \qquad b \tensor{\cA} m \mapsto b \tensor{\cA} 1_\cB \tensor{\cA} m,\]
and observing that it admits the retraction 
\[\cB \tensor{\cA} \cB \tensor{\cA} M \to \cB \tensor{\cA} M, \qquad b \tensor{\cA} b' \tensor{\cA} m \mapsto bb' \tensor{\cA} m,\]
whence it is injective.
Finally, let us show that \ref{item:equivpure2} $\Rightarrow$ \ref{item:equivpure6} $\Rightarrow$ \ref{item:equivpure7}. To prove that \ref{item:equivpure2} implies \ref{item:equivpure6} let $\mf{m}$ be a left ideal in $\cA$. Since $\cA/\mf{m}$ is a left $\cA$-module, \ref{item:equivpure2} entails that
\[\cA/\mf{m} \xrightarrow{f \tensor{\cA} \cA/\mf{m}} \cB \tensor{\cA} \cA/\mf{m} \cong \cB /\cB\mf{m}\] 
is injective and therefore $\mf{m} = f^{-1}(\cB\mf{m})$. Now, to show that \ref{item:equivpure6} implies \ref{item:equivpure7} observe that, by taking $\mf{m} = 0$ in \ref{item:equivpure6}, we know that $f$ is injective and hence we may assume that $\cA \subseteq \cB$. Suppose that $\mf{m}$ is a maximal ideal in $\cA$ and let $\mf{n} \subset \cB$ be a maximal ideal containing $\cB \mf{m}$. Then $\mf{m} = f^{-1}(\cB\mf{m}) \subseteq f^{-1}(\mf{n})$ and so, since $1 \notin \mf{n}$, $\mf{m} = f^{-1}(\mf{n})$ by maximality of $\mf{m}$. 
\end{proof}

Proposition \ref{prop:equivpure} is the analogue of \cite[Chapter I, \S3, n$^\circ$5, Proposition 9]{Bou} and \cite[\S 13.1, Theorem]{Waterhouse} for pure morphisms.

\begin{corollary}[of Proposition \ref{prop:equivpure}]\label{cor:eqpure}
Let $R$ be a ring and let $f \colon \cA \to \cB$ be a morphism of $R$-rings. If $f$ is pure as a morphism of right $\cA$-modules, then $(M,\varrho_M)$ is the equalizer of the pair
\[
\xymatrix @C=50pt{
\cB \tensor{\cA} M \ar@<+0.5ex>[r]^-{\lambda_\cB \tensor{\cA} M} \ar@<-0.5ex>[r]_-{\varrho_\cB  \tensor{\cA} M} & \cB \tensor{\cA} \cB \tensor{\cA} M
}
\]
in the category of left $\cA$-modules, for every left $\cA$-module $M$. In particular, we have the equalizer 
\[
\xymatrix @C=30pt{
\cA \ar[r]^-{f} & \cB \ar@<+0.5ex>[r]^-{\lambda_\cB} \ar@<-0.5ex>[r]_-{\varrho_\cB} & \cB \tensor{\cA} \cB.
}
\]
\end{corollary}

\begin{proof}
Consider the parallel arrows
\[
\xymatrix @C=60pt{
\cB \tensor{\cA} \cB \tensor{\cA} M \ar@<+0.5ex>[r]^-{\cB \tensor{\cA} \lambda_\cB \tensor{\cA} M} \ar@<-0.5ex>[r]_-{\cB \tensor{\cA} \varrho_\cB \tensor{\cA} M} & \cB \tensor{\cA} \cB \tensor{\cA} \cB \tensor{\cA} M.
}
\]
It is clear that the function
\[\cB \tensor{\cA} \varrho_M \colon \cB \tensor{\cA} M \to \cB \tensor{\cA} \cB \tensor{\cA} M, \qquad b \tensor{\cA} m \mapsto b \tensor{\cA} 1_\cB \tensor{\cA} m,\]
lands in the equalizer of $\cB \tensor{\cA} \lambda_\cB \tensor{\cA} M$ and $\cB \tensor{\cA} \varrho_\cB \tensor{\cA} M$ and it is injective, since it admits 
\[\cB \tensor{\cA} \cB \tensor{\cA} M \to \cB \tensor{\cA} M, \qquad b \tensor{\cA} b' \tensor{\cA} m \mapsto bb' \tensor{\cA} m,\]
as a retraction. On the other hand, any $\sum_i b_i \tensor{\cA} b_i' \tensor{\cA} m_i$ in the equalizer satisfies $\sum_i b_i \tensor{\cA} b_i' \tensor{\cA} m_i = \sum_i b_i b_i' \tensor{\cA} 1_\cB \tensor{\cA} m_i$ and hence it is in the image of $\cB \tensor{\cA} \varrho_M$. Therefore, $(\cB \tensor{\cA} M,\cB \tensor{\cA} \varrho_M)$ is the equalizer of $\cB \tensor{\cA} \lambda_\cB \tensor{\cA} M$ and $\cB \tensor{\cA} \varrho_\cB \tensor{\cA} M$, whence $(M,\varrho_M)$ is the equalizer of $\lambda_\cB \tensor{\cA} M$ and $\varrho_\cB \tensor{\cA} M$ by purity (statement \ref{item:equivpure4} in Proposition \ref{prop:equivpure}).
\end{proof}

\begin{corollary}[of Proposition \ref{prop:equivpure}]\label{cor:frompuretoflat}
Let $R$ be a ring and let $f \colon \cA \to \cB$ be a morphism of $R$-rings. If $f$ is a pure morphism of right $R$-modules and $\cB_R$ is flat, then $\cA_R$ is flat.
\end{corollary}

\begin{proof}
Take a monomorphism of left $R$-modules $M \xrightarrow{g} N$. By flatness of $\cB$ on $R$ we have that $\cB \tensor{R} M \xrightarrow{\cB \tensor{R} g} \cB \tensor{R} N$ is injective and so
\[\cB \tensor{\cA} \cA \tensor{R} M \xrightarrow{\cB \tensor{\cA} \cA \tensor{R} g} \cB \tensor{\cA} \cA \tensor{R} N\]
is injective, too. By purity (statement \ref{item:equivpure3} in Proposition \ref{prop:equivpure}), $\cA \tensor{R} M \xrightarrow{\cA \tensor{R} g} \cA \tensor{R} N$ is injective.
\end{proof}

\begin{corollary}[of Proposition \ref{prop:equivpure}]\label{cor:pure}
Let $R$ be any ring and let $\cA \xrightarrow{g} \cT \xrightarrow{h} \cB$ be morphisms of $R$-rings. Set $f \coloneqq h \circ g$. If $f \colon \cA \to \cB$ is a pure morphism of right $\cA$-modules, then the morphism of right $R$-modules $g$ is pure as well. In particular, $f$ is pure as a morphism of right $R$-modules.
\end{corollary}

\begin{proof}
Let $P$ be any left $R$-module. Consider the morphism
\[\cB \xrightarrow{\cong} \cB \tensor{\cA} \cA \xrightarrow{\cB \tensor{\cA} f} \cB \tensor{\cA} \cB.\]
Since it is split injective, with retraction induced by the multiplication of $\cB$, the morphism
\[\cB \tensor{R} P \xrightarrow{\cong} \cB \tensor{\cA} \cA \tensor{R} P \xrightarrow{\cB \tensor{\cA} f \tensor{R} P} \cB \tensor{\cA} \cB \tensor{R} P\]
is (split) injective as well and hence 
\[\cB \tensor{\cA} \cA \tensor{R} P \xrightarrow{\cB \tensor{\cA} f \tensor{R} P} \cB \tensor{\cA} \cB \tensor{R} P\]
is injective. Since 
\[
\cB \tensor{\cA} f \tensor{R} P = \left(\cB \tensor{\cA} h \tensor{R} P\right) \circ \left(\cB \tensor{\cA} g \tensor{R} P\right),
\]
then $\cB \tensor{\cA} g \tensor{R} P$ is injective, too. Being $\cB_\cA$ pure, the morphism $g \tensor{R} P$ is injective by statement \ref{item:equivpure3} in Proposition \ref{prop:equivpure}. The last statement follows by taking $f$ itself as $g$ and $\id_\cB$ as $h$.
\end{proof}

\begin{lemma}\label{lem:ffpure}
Let $R$ be any ring and let $\cA \xrightarrow{g} \cT \xrightarrow{h} \cB$ be morphisms of $R$-rings. If $\cB_\cA$ is faithfully flat, then the morphism of right $R$-modules $g$ is pure. In particular, $f \coloneqq h \circ g$ is pure.
\end{lemma}

\begin{proof}
Since a faithfully flat extension is pure, the statement follows from Corollary \ref{cor:pure}.
\end{proof}


\subsection{The Takeuchi-Sweedler crossed product}
Let ${}_{\Ae}M{}_{\Ae}$ and ${}_{\Ae}N{}_{\Ae}$ be two $\Ae$-bimodules. We first define the $A$-bimodule 
\begin{equation}\label{eq:tensA}
M \tensor{A} N \coloneqq  \lt{M} \tensor{A}  \ls{N} = \frac{M \ot N}{\Big\langle t(\op{a})m \ot n - m \ot s(a)n ~\big\vert~ m\in M, n\in N, a\in A \Big\rangle}.
\end{equation}
This will usually be the tensor product over $A$ that we are going to consider more often, unless specified otherwise.
Then, inside $M \tensor{A} N$ we consider the subspace
\begin{equation}\label{eq:takA}
M \tak{A} N \coloneqq \left\{ \left. \sum_i m_i \tensor{A} n_i \in M\tensor{A}N \ \right| \ \sum_i m_it(\op{a}) \tensor{A} n_i = \sum_i m_i \tensor{A} n_is(a)\right\}.
\end{equation}
It is easy to see that $M \tak{A} N$ is an $A$-subbimodule (left $\Ae$-submodule) with respect to the actions
\begin{equation}\label{eq:actionstak}
a\cdot \left(\sum_i m_i \tensor{A} n_i\right) = \sum_i s(a)m_i \tensor{A} n_i \qquad \text{and} \qquad \left(\sum_i m_i \tensor{A} n_i\right)\cdot a = \sum_i m_i \tensor{A} t(\op{a})n_i
\end{equation}
for all $\sum_i m_i \tensor{A} n_i\in M\tak{A}N$ and $a\in A$, but it is also an $A$-subbimodule (right $\Ae$-submodule) with respect to the actions
\begin{equation}\label{eq:actionstak2}
a\bullet \left(\sum_i m_i \tensor{A} n_i\right) = \sum_i m_i \tensor{A} n_it(\op{a}) \qquad \text{and} \qquad \left(\sum_i m_i \tensor{A} n_i\right)\bullet a = \sum_i m_is(a) \tensor{A} n_i.
\end{equation}
In particular, it is an $\Ae$-bimodule itself.
There exist categorical ways to describe this Takeuchi-Sweedler product in terms of ends and coends (see, e.g., \cite{Sweedler-groups,Takeuchi-groups}, where the integral notation dating back to \cite{Yoneda} was used) or in terms of monoidal products (see, e.g., \cite{Bohm:handbook,Takeuchi-Morita}), but, for the convenience of the unaccustomed reader and for the sake of simplicity, we decided to opt for the more elementary description above. In particular, in $M \tak{A} N$ the following relations hold for all $\sum_i m_i \tensor{A} n_i \in M\tak{A} N$ and all $a\in A$:
\[
\sum_i t(\op{a})m_i \tensor{A} n_i \stackrel{\eqref{eq:tensA}}{=} \sum_i m_i \tensor{A} s(a)n_i \qquad \text{and} \qquad \sum_i m_it(\op{a}) \tensor{A} n_i \stackrel{\eqref{eq:takA}}{=} \sum_i m_i \tensor{A} n_is(a).
\]
In this way if $\cU$ and $\cV$ are two $\Ae$-rings, then $\cU \tak{A} \cV$ is also an $\Ae$-ring, with multiplication 
\[\left(\sum_iu_i \tensor{A} v_i\right)\cdot \left(\sum_ju'_j \tensor{A} v'_j\right) = \sum_{i,j}u_iu'_j\tensor{A}v_iv'_j\]
for all $u_i,u'_j\in \cU$, $v_i,v'_j\in \cV$ and $\K$-algebra morphism $\env{A} \to \cU \tak{A} \cV, a \ot \op{b} \mapsto s_\cU(a) \tensor{A} t_\cV\big(\op{b}\big)$. The $\Ae$-actions \eqref{eq:actionstak} and \eqref{eq:actionstak2} are induced by this $\Ae$-ring structure.


\subsection{Left bialgebroids}\label{ssec:n1}

Next, we recall the definition of a left bialgebroid. It can be considered as a revised version of the notion of a $\times_A$-bialgebra as it appears in \cite[Definition 4.3]{Schauenburg-Hbim}. However, we prefer to mimic \cite{Lu} as presented in \cite[Definition 2.2]{BrMi}.

\begin{definition}\label{def:bialgebroid}
A \emph{left bialgebroid} is the datum of
\begin{enumerate}[label=(B\arabic*),ref=(B\arabic*)]
\item a pair $\left(A,\cH\right)$ of $\K$-algebras;
\item a $\K$-algebra map $\etaup\colon  \Ae \to \cH$, inducing a source $s\colon A\to \cH$ and a target $t\colon \op{A}\to\cH$ and making of $\cH$ an $\Ae$-bimodule;
\item an $A$-coring structure $\left(\cH,\Delta,\varepsilon\right)$ on the $A$-bimodule ${}_{\Ae}\cH=\sMto{\cH}$, that is to say,
\[
\Delta \colon  \sMto{\cH} \to \sMto{\cH}\tensor{A}\sMto{\cH} \qquad \text{and} \qquad \varepsilon\colon \sMto{\cH} \to A
\]
$A$-bilinear maps such that
\begin{equation}\label{eq:coasscoun}
\left(\Delta \tensor{A} \cH\right)\circ \Delta = \left(\cH \tensor{A}\Delta\right) \circ \Delta \qquad \text{and} \qquad \left(\varepsilon\tensor{A}\cH\right)\circ \Delta = \id_\cH = \left(\cH\tensor{A}\varepsilon\right)\circ \Delta;
\end{equation}
\end{enumerate}
subject to the following compatibility conditions
\begin{enumerate}[resume*]
\item\label{item:B4} $\Delta$ takes values into $\sMto{\cH} \tak{A} \sMto{\cH}$ and $\Delta\colon  \sMto{\cH} \to \sMto{\cH}\tak{A}\sMto{\cH}$ is a morphism of $\K$-algebras;
\item\label{item:B5} $\varepsilon\Big(xs\big(\varepsilon\left(y\right)\big)\Big) = \varepsilon\left(xy\right) = \varepsilon\Big(xt\big(\op{\varepsilon\left(y\right)}\big)\Big)$ for all $x,y\in\cH$;
\item\label{item:B6} $\varepsilon(1_\cH)=1_A$.
\end{enumerate}
A $\K$-linear map $\varepsilon\colon  \cH \to A$ which is left $\Ae$-linear and satisfies \ref{item:B5} and \ref{item:B6} is called a \emph{left character} on the $\Ae$-ring $\cH$ (see \cite[Lemma 2.5 and following]{Bohm:handbook}).

A \emph{morphism} of bialgebroids $\B{\phi}\colon  (A,\cH) \to (B,\cK)$ is a pair of algebra maps $\left( \phi _{{0}} \colon A \to B,\phi _{{1}} \colon \cH \to \cK\right)$ such that
\begin{gather*}
\phi _{1}\circ s_\cH = s_\cK \circ \phi _{0}, \qquad \phi _{1}\circ t_\cH = t_\cK \circ \phi_{0}, \\ 
\varepsilon_\cK \circ \phi _{1} = \phi _{0} \circ \varepsilon_\cH ,  \qquad \Delta_\cK \circ \phi _{1} =\chi \circ \left( \phi _{1}\otimes _{A}\phi_{1}\right) \circ \Delta_\cH,
\end{gather*}
where $\chi \colon \cK \tensor{A} \cK \rightarrow \cK \tensor{B} \cK$ is the obvious projection induced by $\phi_0$, that is $\chi \left(h\tensor{A}k\right) =h\tensor{B}k$. If $A = B$ and $\phi_0 = \id_A$, then we say that $\phi_1 : \cH \to \cK$ is a morphism of bialgebroids over $A$.
\end{definition}

As a matter of terminology, a bialgebroid $(A,\cH)$ as in definition \ref{def:bialgebroid} is often referred to as \emph{a left bialgebroid $\cH$ over $A$}. Since in the following we will mainly deal with bialgebroids over a fixed base $A$, we may often omit to specify it and simply refer to $(A,\cH)$ as \emph{the left bialgebroid $\cH$}.

\begin{remark}
Let us make explicit some of the relations involved in the definition of a left bialgebroid and some of their consequences. In terms of elements of $A$ and $\cH$, and by resorting to Heyneman-Sweedler Sigma Notation, relations \eqref{eq:coasscoun} become
\[
\sum x_{11} \tensor{A} x_{12} \tensor{A} x_2 = \sum x_{1}\tensor{A} x_{21}\tensor{A} x_{22} \qquad \text{and} \qquad \sum s\big(\varepsilon\left(x_1\right)\big)x_2 = x = \sum t\big(\op{\varepsilon\left(x_2\right)}\big)x_1
\]
for all $x\in\cH$. The $A$-bilinearity of $\Delta$ forces
\begin{equation}\label{eq:DeltaLin}
\Delta\left(s\left(a\right)\right) = a\cdot \Delta\left(1_\cH\right) = s\left(a\right)\tensor{A}1_\cH \qquad \text{and} \qquad \Delta\left(t\left(\op{a}\right)\right) = \Delta\left(1_\cH\right)\cdot a = 1_\cH \tensor{A} t\left(\op{a}\right)
\end{equation}
for all $a\in A$, and its multiplicativity forces, as a consequence,
\begin{gather*}
\Delta\left(xs\left(a\right)\right) = \Delta\left(x\right)\Delta\left(s\left(a\right)\right) = \left(\sum x_1\tensor{A}x_2\right)\left(s\left(a\right)\tensor{A}1_\cH\right) = \sum x_1s\left(a\right)\tensor{A}x_2, \\
\Delta\left(xt\left(\op{a}\right)\right) = \Delta\left(x\right)\Delta\left(t\left(\op{a}\right)\right) = \left(\sum x_1\tensor{A}x_2\right)\left(1_\cH\tensor{A}t\left(\op{a}\right)\right) = \sum x_1\tensor{A}x_2t\left(\op{a}\right),
\end{gather*}
for all $x\in\cH$. In particular, $\Delta:\cH \to \cH \tak{A} \cH$ is a morphism of $\Ae$-rings.
\end{remark}

The following result is extremely well-known (it underlies the monoidality of the category of left $\cH$-modules).

\begin{proposition}\label{prop:itwaseasy}
Let $(A,\cH)$ be a left bialgebroid and let $M,N$ be two left $\cH$-modules. If we consider $M$ and $N$ as $A$-bimodules (left $\Ae$-modules) via the restriction of scalars along $\etaup$, then $M \tensor{A} N$ is a left $\cH\tak{A}\cH$-module with
\begin{equation}\label{eq:itwaseasy}
\begin{gathered}
\xymatrix @R=0pt{
\sMto{\cH} \tak{A} \sMto{\cH} \ar[r] & \End{\K}{{_\etaup M} \tensor{A} {_\etaup N}} \\
\sum_ix_i \tensor{A} y_i \ar@{|->}[r] & \Big[m\tensor{A}n \mapsto \sum_ix_i\cdot m\tensor{A}y_i\cdot n\Big].
}
\end{gathered}
\end{equation}
Moreover, \eqref{eq:itwaseasy} is a morphism of $\Ae $-bimodules if we consider $\cH \tak{A} \cH$ endowed with the actions \eqref{eq:actionstak} from the left and \eqref{eq:actionstak2} from the right and we consider $\End{\K}{M \tensor{A} N}$ endowed with the left $\Ae $-action coming from the regular $A$-bimodule structure on the codomain and the right $\Ae $-action coming from the regular $A$-bimodule structure on the domain. Equivalently,
\[(\cH \tak{A} \cH) \tensor{\Ae } (M \tensor{A} N) \to M \tensor{A} N, \qquad \left(\sum_ix_i \tensor{A} y_i\right) \tensor{\Ae } \left(m \tensor{A} n\right) \mapsto \sum_ix_i\cdot m\tensor{A}y_i\cdot n\]
is a left $\Ae $-linear action.
\end{proposition}

\begin{corollary}\label{cor:itwaseasy}
If $M,N$ are left $\cH$-modules, then $M \tensor{A} N$ is a left $\cH$-module via restrictions of scalars along $\Delta$:
\begin{equation}\label{eq:HDeltact}
\cH \to \End{\K}{M \tensor{A} N}, \qquad h \mapsto \Big[ m \tensor{A} n \mapsto \sum h_1\cdot m \tensor{A} h_2 \cdot n\Big]
\end{equation}
and the latter is of $\Ae $-bimodules. Equivalently,
\[\cH \tensor{\Ae } (M \tensor{A} N) \to M \tensor{A} N, \qquad h \tensor{\Ae } \left(m \tensor{A} n\right) \mapsto \sum h_1\cdot m\tensor{A} h_2\cdot n\]
is a left $\Ae $-linear action. In particular, the category of left $\cH$-modules is monoidal with tensor product $\tensor{A}$ and unit object $A$. The left $\cH$-action on $A$ is given by
\begin{equation}\label{eq:epsiaction}
h \cdot a \coloneqq \varepsilon\big(ht(\op{a})\big) = \varepsilon\big(hs(a)\big)
\end{equation}
for all $a \in A$, $h \in \cH$.
\end{corollary}

\begin{definition}
Given a left bialgebroid $(A,\cH)$, a \emph{left $\cH$-module coring} is a comonoid in the monoidal category $\big(\lmod{\cH},\tensor{A},A\big)$ of left $\cH$-modules.
\end{definition}


\subsection{Module corings and relative Hopf modules}\label{ssec:Hopfmods}

Let $(A,\cH)$ be a left bialgebroid and let $\left(\cl{\cH},\cl{\Delta},\cl{\varepsilon}\right)$ be a left $\cH$-module coring. Being an $\cH$-module, $\cl{\cH}$ has an $A$-bimodule structure $\sMto{\cl{\cH}}$ given by
\[
a\cdot u \cdot b \coloneqq \etaup(a \ot \op{b}) \cdot u = s(a)t(\op{b})\cdot u
\]
for all $a,b \in A$, $u \in \cl{\cH}$, with respect to which it is an $A$-coring.
In this setting, we may consider the category
\[
\LHopf{\overline{\cH}}{\cH} \coloneqq \prescript{\overline{\cH}}{}{\Big(\lmod{\cH}\Big)}
\]
of relative $\Big(\overline{\cH},\cH\Big)$-Hopf modules: these are left comodules over the comonoid $\cl{\cH}$ in the monoidal category $\left(\lmod{\cH},\tensor{A},A\right)$. In details, they are left $\cH$-modules $(M,\mu_M\colon \cH \ot M \to M)$ together with a coassociative and counital left $\cl{\cH}$-coaction
\[
\partial_M\colon  M \to \Mto{\cl{\cH}} \tensor{A} \sM{M}, \qquad m \mapsto \sum m_{-1} \tensor{A} m_0,
\] 
which is also a morphism of left $\cH$-modules, that is, for all $h \in \cH$, $m \in M$ one has
\[\sum h_1 \cdot m_{-1} \tensor{A} h_2 \cdot m_0 = h \cdot \partial_M(m) = \partial_M(h \cdot m) = \sum (h \cdot m)_{-1} \tensor{A} (h \cdot m)_0.\]
In the present subsection we are interested in the category $\LHopf{\overline{\cH}}{\cH}$. 
A Hopf algebroid analogue of Doi's equivalence theorem \cite[Theorem 2.3]{Doi} will be of key importance in proving Theorem \ref{thm:bah} below.

Now, suppose that $\pi\colon  \cH \to \cl{\cH}$ is a morphism of left $\cH$-module corings. Recall that if $C$ is an $A$-coring with distinguished group-like element $g\in C$ and if $(N,\partial_N)$ is a left $C$-comodule, then
\[
\coinv{N}{C} \coloneqq \left\{ n\in N ~\big\vert~ \sum n_{-1}\tensor{A} n_0 = g \tensor{A} n \right\}
\]
is the so-called space of \emph{(left) coinvariant elements} in $N$. 
In the standing hypotheses, $\pi(1_\cH)$ is a distinguished group-like element in $\cl{\cH}$ and we may consider 
\[
\coinv{\cH}{\cl{\cH}} = \Big\{x\in \cH ~\big\vert~ \sum \pi(x_1) \tensor{A} x_2 = \pi(1_\cH) \tensor{A} x \Big\},
\]
(also denoted by $\coinv{\cH}{\pi}$), which is the space of (left) coinvariant elements in $\cH$ under the (left) $\cl{\cH}$-coaction $\cH \to \cl{\cH} \tensor{A} \cH$ given by
\begin{equation}\label{eq:smalldelta}
\partial \coloneqq (\pi\tensor{A}\cH)\circ \Delta.
\end{equation}

\begin{remark}
Observe that, in general, $\coinv{\cH}{\cl{\cH}}$ is just a right $A$-submodule (left $\op{A}$-submodule) with respect to the action $\wra$ from \eqref{eq:triangleactions} and a left $A$-submodule  (right $\op{A}$-submodule) with respect to the action $\bla$, because of \eqref{eq:DeltaLin}:
\begin{gather*}
\sum \pi\big((a\bla x)_1\big)\tensor{A}(a\bla x)_2 = \sum \pi(x_1)\tensor{A}x_2t(\op{a}) = \pi(1_\cH) \tensor{A} (a\bla x), \\
\sum \pi\big((x\wra a)_1\big)\tensor{A}(x\wra a)_2 = \sum \pi(x_1)\tensor{A}t(\op{a})x_2 = \pi(1_\cH) \tensor{A} (x\wra a). \qedhere
\end{gather*}
\end{remark}

\begin{lemma}\label{lem:coinvariants}
The space $\coinv{\cH}{\cl{\cH}}$ is an $\Ao$-subring of $\cH$ via $t$ and $\partial$ from \eqref{eq:smalldelta} is right $\coinv{\cH}{\cl{\cH}}$-linear.
\end{lemma}

\begin{proof}
Since $\coinv{\cH}{\cl{\cH}}\subseteq \cH$ and clearly $1_{\cH} \in \coinv{\cH}{\cl{\cH}}$, to check that it is a $\K$-subalgebra we just need to verify that the induced multiplication is well-defined. To this aim, observe that $\cl{\cH}\tensor{A}\cH$ has a left $\cH$-module structure given by Corollary \ref{cor:itwaseasy} and it has a natural right $\cH$-module structure given by
\begin{equation}\label{eq:Hactright}
\cH \to \op{\End{\K}{\cl{\cH}\tensor{A}\cH}}, \qquad x \mapsto \left[u\tensor{A} y \mapsto u \tensor{A} yx\right].
\end{equation}
Thanks to this, we may compute directly that for $x,y\in\coinv{\cH}{\cl{\cH}}$
\begin{gather*}
\sum \pi((xy)_1) \tensor{A} (xy)_2 = \sum \pi(x_1y_1) \tensor{A} x_2y_2 = \sum x_1\cdot \pi(y_1) \tensor{A} x_2y_2 \stackrel{\eqref{eq:HDeltact}}{=} x \cdot \left(\sum \pi(y_1) \tensor{A} y_2\right) \\
 = x\cdot \left( \pi(1_\cH) \tensor{A} y \right) \stackrel{\eqref{eq:HDeltact}}{=} \sum \pi(x_1) \tensor{A} x_2y \stackrel{\eqref{eq:Hactright}}{=} \left(\sum \pi(x_1) \tensor{A} x_2\right)y = \pi(1_\cH) \tensor{A} xy,
\end{gather*}
so that $\coinv{\cH}{\cl{\cH}}$ is indeed a subalgebra. Moreover, since
\[
\sum \pi(t(\op{a})_1) \tensor{A} t(\op{a})_2 = (\pi \tensor{A}\cH)(\Delta(t(\op{a}))) \stackrel{\eqref{eq:DeltaLin}}{=} (\pi \tensor{A}\cH)(1_\cH\tensor{A}t(\op{a})) = \pi(1_\cH)\tensor{A}t(\op{a})
\]
for all $a\in A$, $t$ takes values in $\coinv{\cH}{\cl{\cH}}$ and so $\coinv{\cH}{\cl{\cH}}$ is $\op{A}$-subring of $\cH$ via $t$.

To conclude, notice that if $x \in \cH$ and $b \in \coinv{\cH}{\cl{\cH}}$ then
\begin{align*}
\partial(xb) & = \sum \pi(x_1b_1) \tensor{A} x_2b_2 \stackrel{\eqref{eq:HDeltact}}{=} x \cdot \left(\sum \pi(b_1) \tensor{A} b_2\right)  = x\cdot \left( \pi(1_\cH) \tensor{A} b \right) \\
& \stackrel{\eqref{eq:HDeltact}}{=} \sum \pi(x_1) \tensor{A} x_2b \stackrel{\eqref{eq:Hactright}}{=} \left(\sum \pi(x_1) \tensor{A} x_2\right)b = \partial(x)b. \qedhere
\end{align*}
\end{proof}

Set $B \coloneqq \coinv{\cH}{\cl{\cH}}$, for the sake of brevity. 
Since $\cH$ becomes a right $B$-module by restriction of scalars along $\iota: B \to \cH$, we can consider the category $\lmod{B}$ and the extension of scalars functor
\[F \coloneqq \cH \tensor{B} (-) \colon \lmod{B} \to \lmod{\cH},\] 
which is left adjoint to the restriction of scalars functor ${_{\iota}\big(-\big)} \colon \lmod{\cH} \to \lmod{B}.$
It turns out that the natural transformation
\[
\cH \tensor{B} (-) \to \cl{\cH} \tensor{A} \big(\cH \tensor{B} -\big), \qquad h \tensor{B} - \mapsto \sum\pi(h_1) \tensor{A} \big(h_2 \tensor{B} -\big)
\]
induces a morphism of comonads
\begin{equation}\label{eq:bigtheta}
\Theta\colon  \cH \tensor{B} {_\iota\big(-\big)} \to \cl{\cH} \tensor{A} -, \qquad h \tensor{B} (-) \mapsto \sum \pi(h_1) \tensor{A} h_2\cdot (-)
\end{equation}
as in \cite[Proposition 2.1]{Pepe} and so, since $\LHopf{\overline{\cH}}{\cH}$ is also the category of Eilenberg-Moore objects for the comonad $\cl{\cH} \tensor{A} -$ over $\lmod{\cH}$, the functor $F$ induces a functor $\cK \colon  \lmod{B} \to \LHopf{\cl{\cH}}{\cH}$. The following theorem is nothing else than \cite[Theorem 2.7]{Pepe}.   

\begin{theorem}\label{thm:super!}
Let $(A,\cH)$ be a left bialgebroid with $\sM{\cH}$ $A$-flat. If $\cl{\cH}$ is a left $\cH$-module coring together with a morphism $\pi\colon \cH \to \cl{\cH}$ of left $\cH$-module corings and if $B = \coinv{\cH}{\cl{\cH}}$, then 
\[
\cK \colon  \lmod{B} \to \LHopf{\cl{\cH}}{\cH}, \qquad M \mapsto \cH \tensor{B} M,
\]
is an equivalence of categories if and only if
\begin{enumerate}[leftmargin=0.7cm]
\item $\cH \tensor{B} - \colon  \lmod{B} \to \lmod{\cH}$ is comonadic,
\item the canonical morphism of comonads $\Theta$ is a natural isomorphism.
\end{enumerate}
\end{theorem}

\begin{lemma}
The natural transformation $\Theta$ of \eqref{eq:bigtheta} is a natural isomorphism if and only if
\begin{equation}\label{eq:xiisback}
\xi\colon \cH_\iota \tensor{B} {_\iota\cH} \to \Mto{\cl{\cH}} \tensor{A} \sM{\cH}, \qquad x \tensor{B} y \mapsto \sum \pi(x_1) \tensor{A} x_2y,
\end{equation}
is an isomorphism.
\end{lemma}

\begin{proof}
If $\Theta$ is a natural isomorphism, then $\Theta_{\cH} = \xi$ is an isomorphism. Conversely, since $\xi$ is right $\cH$-linear with respect to the regular right $\cH$-module structures, we may conclude that if $\xi$ is an isomorphism, then
\[\xi \tensor{\cH} M \colon \Big(\cH_\iota \tensor{B} {_\iota\cH}\Big) \tensor{\cH} M \to \left(\Mto{\cl{\cH}} \tensor{A} \sM{\cH}\right)\tensor{\cH} M, \qquad \left(x \tensor{B} y\right)\tensor{\cH}m \mapsto \sum \left(\pi(x_1) \tensor{A} x_2y \tensor{\cH} m\right),\]
is an isomorphism for every $M$ in $\lmod{\cH}$, natural in $M$. The commutativity of the diagram
\[
\xymatrix{
\Big(\cH_\iota \tensor{B} {_\iota\cH}\Big) \tensor{\cH} M \ar[r]^-{\xi \tensor{\cH} M}  \ar[d]_-{\cong} & \left(\Mto{\cl{\cH}} \tensor{A} \sM{\cH}\right)\tensor{\cH} M \ar[d]^-{\cong} \\
\cH_\iota \tensor{B} {_\iota M} \ar[r]_-{\Theta_M} & \Mto{\cl{\cH}} \tensor{A} \sM{M}
}
\]
allows us to conclude that if $\xi$ is an isomorphism, then $\Theta$ is a natural isomorphism. Notice that, due to the fact that only regular module structures are involved, the vertical isomorphisms are, in fact, the canonical isomorphism
\begin{gather*}
\Big(\cH_\iota \tensor{B} {_\iota\cH}\Big) \tensor{\cH} M \cong \cH_\iota \tensor{B} \Big({_\iota\cH} \tensor{\cH} M\Big) \cong \cH_\iota \tensor{B} {_\iota M} \qquad \text{and} \\
\left(\Mto{\cl{\cH}} \tensor{A} \sM{\cH}\right)\tensor{\cH} M \cong \Mto{\cl{\cH}} \tensor{A} \left(\sM{\cH}\tensor{\cH} M\right) \cong \Mto{\cl{\cH}} \tensor{A} \sM{M}. \qedhere
\end{gather*}
\end{proof}

In general, even when $\cK$ is not an equivalence of categories, $\cK$ still admits a right adjoint functor, as in the case for ordinary Hopf algebras.

\begin{proposition}
The construction $\coinv{M}{\cl{\cH}} \coloneqq \{ m \in M \mid \sum m_{-1} \tensor{A} m_0 = \pi(1_\cH) \tensor{A} m\}$ for every $M$ in $\LHopf{\cl{\cH}}{\cH}$ induces a functor $\coinv{(-)}{\cl{\cH}} \colon  \LHopf{\cl{\cH}}{\cH} \to \lmod{B}$ which is right adjoint to $\cK$. Unit and counit are given by
\begin{align}
\zeta_{V} \colon  & V \to \coinv{(\cH \tensor{B} V)}{\cl{\cH}}, \qquad & & v \mapsto 1 \tensor{B} v, \qquad & & \Big(V \text{ in }\lmod{B}\Big) \notag \\
\theta_M \colon  & \cH \tensor{B} \coinv{M}{\cl{\cH}} \to M, \qquad & & h \tensor{B} m \mapsto h \cdot m, \qquad & & \Big(M \text{ in }\LHopf{\cl{\cH}}{\cH}\Big) \label{eq:counit}
\end{align}
respectively.
\end{proposition}

\begin{proof}
Given a relative Hopf module $(M,\mu_M,\partial_M)$, denote by $\Omega(M)$ its underlying $H$-module structure $(M,\mu_M)$. In view of \cite[Proposition 2.3]{Pepe}, $\cK$ admits a right adjoint which is explicitly realized on objects as the equalizer of the parallel pair
\[
\xymatrix @C=50pt{
{_\iota \Omega(M)} \ar@<+0.5ex>[r]^-{m \mapsto \pi{1_\cH} \tensor{A} m} \ar@<-0.5ex>[r]_-{\partial_M} & {_\iota \Big( \cl{\cH} \tensor{A} \Omega(M)\Big)},
}
\]
which is exactly $\coinv{M}{\cl{\cH}}$.
\end{proof}


\section{The correspondence for arbitrary Hopf Algebroids}\label{sec:N1}

This Section contains our first main result, namely,  Theorem \ref{mainthm:Galois}.


\subsection{The correspondence between left ideals two-sided coideals and right coideal subrings}\label{sec:ITCANBEDONE}

Henceforth, all bialgebroids will be left ones. 

\subsubsection{General facts about left ideals two-sided coideals and right coideal subrings}\label{ssec:GenFacts}
As a matter of notation, for $(A,\cH)$ a bialgebroid and $B\subseteq \cH$ we set
\[
\cH^+\coloneqq \ker(\varepsilon) \qquad \text{and} \qquad B^+\coloneqq B \cap \cH^+.
\]
Denote by $\iota\colon B \to \cH$ the inclusion. Assume that $B$ is an $\op{A}$-subring of $\cH$ via $t$ (as we will have later in the paper), that is, $B$ is an $\op{A}$-ring via a $\K$-algebra extension $t'\colon \op{A}\to B$ such that $\iota \circ t'= t$.
If we consider the restriction $\varepsilon' \coloneqq \varepsilon \circ \iota$ of $\varepsilon$ to $B$, then the latter is a right $A$-linear morphism (left $\op{A}$-linear, for the sake of precision) and $B^+ = \cH^+ \cap B = \ker(\varepsilon)\cap B = \ker(\varepsilon')$.

\begin{remark}\label{rem:sesright}
If $B$ is an $\op{A}$-subring of $\cH$ via $t$, then we have that
\begin{equation}\label{eq:ssesright}
\xymatrix{
0 \ar[r] & B^+ \ar[r]^{i} & B \ar[r]^{\varepsilon'} & A \ar[r] & 0
}
\end{equation}
is a short exact sequence of left $B$-modules, where $A$ has the left $B$-module structure coming from the restriction of scalars along the inclusion $\iota$ of $B$ into $\cH$. Namely,
\begin{equation}\label{eq:BonA}
b\cdot a \coloneqq \varepsilon\Big(\iota(b)t(\op{a})\Big) = \varepsilon\Big(\iota\big(bt'(\op{a})\big)\Big) = \varepsilon'\Big(bt'(\op{a})\Big)
\end{equation}
 for all $a\in A,b\in B$ (see \eqref{eq:epsiaction}). Moreover, \eqref{eq:ssesright} is a split short exact sequence of right $A$-modules (in fact, left $\op{A}$-modules via $t'$), where $t'$ provides an $A$-linear section for $\varepsilon'$ in view of \ref{item:B6}:
\[\varepsilon'\big(t'(\op{a})\big) = \varepsilon\Big(\iota\big(t'(\op{a})\big)\Big) = \varepsilon\big(t(\op{a})\big) = a. \qedhere\]
\end{remark}

Recall that an $\Ao$-subring $B$ of $\cH$ via $t$ is a \emph{right $\cH$-comodule $\Ao$-subring} if it admits a right $A$-linear coaction $\delta \colon  \Mto{B} \to \Mto{B} \tensor{A} \sMto{\cH}$ such that the following diagram commutes
\begin{equation}\label{eq:comodulesubring}
\begin{gathered}
\xymatrix @R=12pt{B \ar[r]^-{\delta} \ar[d]_-{\iota} & B \tensor{A} \cH \ar[d]^-{\iota \tensor{A} \cH} \\ \cH \ar[r]_-{\Delta} & \cH \tensor{A} \cH. }
\end{gathered}
\end{equation}

\begin{remark}\label{rem:monoidality}
It is a well-known fact that if $(A,\cH)$ is a left bialgebroid, then the category $\lcomod{\cH}$ of left $\cH$-comodules is a monoidal category with tensor product $\tensor{A}$ and unit object $A$, where the $\cH$-coaction on $A$ is provided by the source map (see, e.g., \cite[Theorem 3.18]{Bohm:handbook}, where the property is stated for the right-hand scenario) and where every left $\cH$-comodule $M$ is a right $A$-module with action
\begin{equation}\label{eq:addAction}
m \cdot a \coloneqq \sum\varepsilon\big(m_{-1}s(a)\big) \cdot m_0
\end{equation}
for all $m \in M$ and $a \in A$. By applying this result to the co-opposite left bialgebroid $(\Ao,\cH^{\mathrm{co}})$, we conclude that the category $\rcomod{\cH}$ of right comodules over the left bialgebroid $(A,\cH)$ is monoidal too, with tensor product $\tensor{\Ao}$ and unit $\Ao$, where the $\cH$-coaction on $\Ao$ is provided by the target map. In this setting, if $\sM{\cH} = {}_{{A\tensor{}\op{1}}}\cH$ is $A$-flat, then a right $\cH$-comodule $\Ao$-subring is nothing other than a monoid in $\big(\rcomod{\cH},\tensor{\Ao},\Ao\big)$. 
\end{remark}

Recall that a \emph{$2$-sided coideal} $N$ in a bialgebroid $(A,\cH)$ is an $A$-subbimodule of $\sMto{\cH}$ such that
\[
\Delta(N) \subseteq \img \Big( N \tensor{A} \cH + \cH \tensor{A} N\Big)  \qquad \text{and} \qquad \varepsilon(N)=0_A,
\] 
where $\img(-)$ denotes the canonical image in the tensor product $\sMto{\cH} \tensor{A} \sMto{\cH}$.


\subsubsection{Correspondence between left ideals two-sided coideals and quotient module corings}

\begin{proposition}\label{prop:uffachefatica}
We have a well-defined bijective correspondence
\[
\xymatrix @R=0pt{
{\left\{ \begin{array}{c} \text{left ideals} \\ 2\text{-sided coideals in } \cH \end{array} \right\}} \ar[r]^-{\Xi} & {\left\{ \begin{array}{c} \text{quotient left } \cH\text{-module} \\ \text{corings of } \cH \end{array} \right\}} \\
I \ar@{|->}[r] & \left[\cH \xrightarrow{\pi} {\cH}/{I}\right] \\
\ker(\pi) & \ar@{|->}[l] \left[\cH \xrightarrow{\pi} \cl{\cH}\right].
}
\]
\end{proposition}

\begin{proof}
Let us begin by picking a left ideal $2$-sided coideal $I\subseteq \cH$ and consider the canonical projection $\pi\colon \cH \to \cH/I$. Since $I$ is a left $\cH$-submodule of $\cH$, $\cH/I$ becomes a left $\cH$-module and, in particular, an $\Ae$-module with $\sMto{\cH/I}$ and $\pi$ is left $\cH$-linear and, in particular, $A$-bilinear. Since $I$ is a $2$-sided coideal, $\sMto{\cH/I}$ inherits a structure $\left(\overline{\Delta},\overline{\varepsilon}\right)$ of $A$-coring in such a way that
\[
\begin{gathered}
\xymatrix @R=15pt{
{\sMto{\cH}} \ar[r]^-{\pi} \ar[d]_-{\Delta} & {\displaystyle \sMto{\frac{\cH}{I}}} \ar[d]^-{\overline{\Delta}} \\
{\sMto{\cH} \tensor{A} {\sMto{\cH}}} \ar[r]_-{\pi\tensor{A}\pi} & { \displaystyle \sMto{\frac{\cH}{I}}  \tensor{A} \sMto{\frac{\cH}{I}} }
}
\end{gathered}
\qquad \text{and} \qquad
\begin{gathered}
\xymatrix{
{\sMto{\cH}} \ar[rr]^-{\pi} \ar[dr]_-{\varepsilon} & & {\displaystyle \sMto{\frac{\cH}{I}}} \ar[dl]^-{\overline{\varepsilon}} \\
 & A & 
}
\end{gathered}
\]
are commutative diagrams in $\Bimod{A}{A} = \lmod{\Ae}$. Let us show that both $\overline{\Delta}$ and $\overline{\varepsilon}$ are morphisms of left $\cH$-modules. Set $\cl{\cH} \coloneqq \cH/I$, for the sake of brevity. The left $\cH$-linearity of $\overline{\Delta}$ is provided by the commutativity of the following diagram
\[
\xymatrix @!0 @R=45pt @C=90pt{
\cH \tensor{\Ae} \cH \ar[rrrr]^-{\mathrm{multiplication}} \ar[dr]^-{\cH \tensor{\Ae} \pi} \ar[ddd]_-{\cH \tensor{\Ae} \Delta} & & & & \cH \ar[dl]^-{\pi} \ar[ddd]^-{\Delta} \\
 & \cH \tensor{\Ae} \overline{\cH}  \ar[rr]^-{\mathrm{multiplication}} \ar[d]_-{\cH \tensor{\Ae} \overline{\Delta}} & & \overline{\cH} \ar[d]^-{\overline{\Delta}} & \\
 & \cH \tensor{\Ae} \left(\overline{\cH} \tensor{A} \overline{\cH}\right) \ar[dr]^-{\Delta \tensor{\Ae} \left(\overline{\cH} \tensor{A} \overline{\cH}\right) } & & \overline{\cH} \tensor{A} \overline{\cH} & \\
\cH \tensor{\Ae} \left(\cH \tensor{A} \cH\right) \ar[ur]^-{\cH \tensor{\Ae} (\pi \tensor{A} \pi)} \ar[drr]_-{\Delta \tensor{\Ae} \left(\cH \tensor{A}\cH\right)} & & \left(\cH \tak{A} \cH\right) \tensor{\Ae} \left(\overline{\cH} \tensor{A} \overline{\cH} \right) \ar[ur]^-{\eqref{eq:itwaseasy}} & & \cH \tensor{A} \cH \ar[ul]_-{\pi \tensor{A} \pi} \\
 & & \left(\cH \tak{A} \cH\right) \tensor{\Ae} \left(\cH \tensor{A} \cH \right). \ar[urr]_-{\eqref{eq:itwaseasy}} \ar[u]|-{\left(\cH \tak{A} \cH\right) \tensor{\Ae} \left(\pi \tensor{A} \pi \right)} & & 
}
\]
The left $\cH$-linearity of $\overline{\varepsilon}$ instead is provided by the commutativity of the following diagram
\[
\xymatrix@R=15pt{
\cH \tensor{\Ae} \cH \ar[rrrr]^-{\cH \tensor{\Ae} \pi} \ar@/_5ex/[dddrr]_-{\cH \tensor{\Ae} \varepsilon} \ar[dr]^(0.6){\mathrm{multiplication}} & & & & \cH \tensor{\Ae} \overline{\cH} \ar[dl]_(0.6){\mathrm{multiplication}} \ar@/^5ex/[dddll]^-{\cH \tensor{\Ae} \overline{\varepsilon}} \\
 & \cH \ar[dr]_-{\varepsilon} \ar[rr]^-{\pi} & & \overline{\cH} \ar[dl]^-{\overline{\varepsilon}} & \\
 & & A & & \\
 & & \cH \tensor{\Ae} A. \ar[u]_-{\eqref{eq:epsiaction}} & & 
}
\]
Summing up, $\cH/I$ is a left $\cH$-module coring.

In the opposite direction, assume that we have a surjective morphism of vector spaces $\pi \colon  \cH \to \cl{\cH}$ which is of left $\cH$-module corings and consider $\ker(\pi) \subseteq \cH$. Since $\pi$ is of left $\cH$-modules, $\ker(\pi)$ is a left ideal in $\cH$. Moreover, being $\pi$ of $A$-corings, we have that the following diagram with exact rows commutes
\[
\xymatrix{
0 \ar[r] & \ker(\pi) \ar[r] \ar@{.>}[d] & \cH \ar[r]^{\pi} \ar[d]_-{\Delta} & \cl{\cH} \ar[r] \ar[d]^-{\cl{\Delta}} & 0 \\
0 \ar[r] & \ker(\pi \tensor{A} \pi) \ar[r] & \cH \tensor{A} \cH \ar[r]_-{\pi \tensor{A} \pi} & \cl{\cH} \tensor{A} \cl{\cH} \ar[r] & 0.
}
\]
Since $\ker(\pi \tensor{A} \pi) = \img\big(\ker(\pi) \tensor{A} \cH + \cH \tensor{A} \ker(\pi)\big)$ and since $\varepsilon\big(\ker(\pi)\big) \subseteq \cl{\varepsilon}\big(\pi\big(\ker(\pi)\big)\big) = 0_A$, it follows that $\ker(\pi)$ is also a $2$-sided coideal.
\end{proof}


\subsubsection{From left ideal two-sided coideals to right coideal subrings}

\begin{proposition}\label{prop:OneIsGone}
Under the assumption that $\sM{\cH} = {}_{{A\tensor{}\op{1}}}\cH$ is $A$-flat, we have a well-defined inclusion-preserving correspondence
\[
\xymatrix @R=0pt{
{\left\{ \begin{array}{c} \text{left ideals} \\ 2\text{-sided coideals in } \cH \end{array} \right\}} \ar[r]^-{\Psi} & {\left\{ \begin{array}{c} \text{right } \cH\text{-comodule} \\ \op{A}\text{-subrings via }t \text{ in } \cH \end{array} \right\}} \\
I \ar@{|->}[r] & \coinv{\cH}{\frac{\cH}{I}}
}
\]
\end{proposition}

\begin{proof}
Let us begin by picking a left ideal $2$-sided coideal $I\subseteq \cH$ and consider the canonical projection $\pi\colon \cH \to \cH/I$.  
We already know from Proposition \ref{prop:uffachefatica} that $\cH/I$ is an $A$-coring and that the canonical projection $\pi\colon  \cH \to \cH/I$ is a morphism of left $\cH$-module corings. Therefore, in view of Lemma \ref{lem:coinvariants}, $\coinv{\cH}{\frac{\cH}{I}}$ is an $\op{A}$-subring of $\cH$ via $t$.
To finish checking the validity of the first correspondence, we are left to check that it is an $\cH$-comodule with respect to the obvious coaction induced by $\Delta$. To this aim, recall that $\coinv{\cH}{\frac{\cH}{I}}$ can be realized as the following equalizer in $\rmod{A}$ (in fact, in $\lmod{\op{A}}$)
\begin{equation}\label{eq:eqright}
\xymatrix{
\Mto{\coinv{\cH}{\frac{\cH}{I}}} \ar[r] & \Mto{\cH} \ar@<+0.5ex>[rr]^-{\pi t ~\tensor{A} \cH} \ar@<-0.5ex>[rr]_-{(\pi~\tensor{A}\cH)\Delta} && {\displaystyle \frac{\cH}{I} \tensor{A} \Mto{\cH}}.
}
\end{equation}
Since $\sM{\cH}$ is $A$-flat, it is enough for us to check that, for every $x\in\coinv{\cH}{\frac{\cH}{I}}$, $\sum x_1 \tensor{A} x_2$ equalizes the pair $\Big((\pi\circ t) \tensor{A} \cH \tensor{A} \cH, (\pi\tensor{A}\cH\tensor{A}\cH)(\Delta\tensor{A}\cH)\Big)$. However,
\begin{align*}
(\pi ~\tensor{A}~ & \cH \tensor{A} \cH)(\Delta\tensor{A}\cH)\left(\sum x_1 \tensor{A} x_2\right) = (\pi\tensor{A}\cH\tensor{A}\cH)(\Delta\tensor{A}\cH)\Delta(x) \\ 
& \stackrel{\eqref{eq:coasscoun}}{=} (\pi\tensor{A}\cH\tensor{A}\cH)(\cH\tensor{A}\Delta)\Delta(x) = \left(\frac{\cH}{I}\tensor{A}\Delta\right)(\pi\tensor{A}\cH)\Delta(x) \stackrel{\eqref{eq:eqright}}{=} \left(\frac{\cH}{I}\tensor{A}\Delta\right)(\pi(1_\cH)\tensor{A}x) \\
& = \sum \pi(1_\cH) \tensor{A} x_1 \tensor{A} x_2 = \big((\pi\circ t)\tensor{A}\cH\tensor{A}\cH\big)\left(\sum x_1\tensor{A}x_2\right)
\end{align*}
and hence $\coinv{\cH}{\frac{\cH}{I}}$ is indeed a right $\cH$-comodule $\op{A}$-subring via $t$.

It is clear that the correspondence is inclusion-preserving because if $I \subseteq J$, then we have a left $\cH$-linear surjective morphism of $A$-corings $\cH/I \twoheadrightarrow \cH/J$ and hence $\coinv{\cH}{\frac{\cH}{I}} \subseteq \coinv{\cH}{\frac{\cH}{J}}$.
\end{proof}


\subsubsection{From right coideal subrings to left ideals two-sided coideals}

In this section we consider a right $\cH$-comodule $\Ao$-subring $B$ of $\cH$ and utilise the notation introduced in Section~\ref{ssec:GenFacts}.
\begin{lemma}
Let $B$ be a right $\cH$-comodule $\Ao$-subring of $\cH$ via $t$ and let $\iota\colon  B \to \cH$ be the inclusion. Denote by $\pi\colon  \cH \to \cH/\cH B^+$ the canonical projection of left $\cH$-modules. Then
\begin{equation}\label{eq:piiotaB}
\pi \circ \iota = \pi \circ \iota \circ t'\circ \varepsilon' = \pi \circ t \circ \varepsilon \circ \iota.
\end{equation}
In particular,
\begin{equation}\label{eq:piiotaDelta}
\left(\pi \tensor{A} \cH\right) \circ \Delta \circ \iota = (\pi \tensor{A} \cH) \circ (t \tensor{A} \cH) \circ \iota.
\end{equation}
\end{lemma}

\begin{proof}
Notice that for all $b \in B$ we have that $b - t'\big(\op{\varepsilon'(b)}\big) \in B^+$ and hence $\iota\Big(b - t'\big(\op{\varepsilon'(b)}\big)\Big) \in \cH B^+$. Therefore,
\[
\pi\iota(b) = \iota(b) + \cH B^+ = \iota\left(t'\left(\op{\varepsilon'(b)}\right)\right) + \iota\big(b - t'\left(\op{\varepsilon'(b)}\right)\big) + \cH B^+ = \pi\iota\left(t'\left(\op{\varepsilon'(b)}\right)\right).
\]
The right-most equality follows by definition of $t'$ and $\varepsilon'$.
To conclude, observe that
\begin{align*}
\left(\pi \tensor{A} \cH\right) \circ \Delta \circ \iota & \stackrel{\eqref{eq:comodulesubring}}{=} \left(\pi \tensor{A} \cH\right) \circ \left(\iota \tensor{A} \cH\right) \circ \delta \\
& \stackrel{\eqref{eq:piiotaB}}{=} \left(\pi \tensor{A} \cH\right) \circ (t \tensor{A} \cH) \circ (\varepsilon \tensor{A} \cH) \circ \left(\iota \tensor{A} \cH\right) \circ \delta \\
& \stackrel{\eqref{eq:comodulesubring}}{=} \left(\pi \tensor{A} \cH\right) \circ (t \tensor{A} \cH) \circ (\varepsilon \tensor{A} \cH) \circ \Delta \circ \iota \\
& \stackrel{\phantom{\eqref{eq:piiotaB}}}{=} (\pi \tensor{A} \cH) \circ (t \tensor{A} \cH) \circ \iota. \qedhere
\end{align*}
\end{proof}

\begin{proposition}\label{prop:Phiright}
We have a well-defined inclusion-preserving correspondence
\[
\xymatrix @R=0pt{
{\left\{ \begin{array}{c} \text{right } \cH\text{-comodule} \\ \op{A}\text{-subrings via }t \text{ in } \cH \end{array} \right\}} \ar[r]^-{\Phi} & {\left\{ \begin{array}{c} \text{left ideals} \\ 2\text{-sided coideals in } \cH \end{array} \right\}} \\
B \ar@{|->}[r] & \cH B^+
}
\]
\end{proposition}

\begin{proof}
Assume that we have a right $\cH$-comodule $\op{A}$-subring $B$ of $\cH$, where the $\op{A}$-ring structure comes from a $\K$-algebra extension $t'\colon \op{A}\to B$ as in Remark \ref{rem:sesright}. We may tensor \eqref{eq:ssesright} by $\cH$ over $B$ on the left and eventually apply the Snake Lemma to find the following commutative diagram of left $\cH$-modules, with exact rows,
\[
\xymatrix @C=40pt{
 & \cH\tensor{B}B^+ \ar[r]^{\cH \tensor{B}i} \ar@{->>}[d]_-{\nu'} & \cH\tensor{B}B \ar[r]^{\cH\tensor{B}\varepsilon'}\ar[d]_-{\nu}^-{\cong} & \cH\tensor{B}A \ar[r] \ar[d]^-{\tilde{\nu}}_-{\cong} & 0 \\
0 \ar[r] & \cH B^+ \ar[r]^{j} & \cH \ar[r]^{\pi} & {\frac{\cH}{\cH B^+}} \ar[r] & 0
}
\]
If we consider $I \coloneqq \cH B^+ = \nu'\left( \cH\tensor{B}B^+\right)$, then this is a left $\cH$-ideal by construction. Let us prove that it is also a two-sided coideal.

First of all, let us collect some observations that will be needed afterwards. 
Recall from Corollary \ref{cor:itwaseasy} that we have a natural $\cH$-action on $\cH \tensor{A} \cH$ given by
\begin{equation}\label{eq:Hactionright}
\cH \ot (\cH \tensor{A} \cH) \xrightarrow{\cdot} \cH \tensor{A} \cH, \qquad x \ot \left(\sum_j h_j \tensor{A} k_j\right) \mapsto \sum_{j} x_1h_j \tensor{A} x_2k_j. 
\end{equation} 
and a natural action of $\cH$ on $\cH/\cH B^+ \tensor{A} \cH$ given by
\begin{equation}\label{eq:Hactionquoright}
\cH \ot \left(\frac{\cH}{\cH B^+} \tensor{A} \cH\right) \xrightarrow{\cdot} \frac{\cH}{\cH B^+} \tensor{A} \cH, \qquad x \ot \left(\sum_j \pi(h_j) \tensor{A} k_j\right) \mapsto \sum_{j} \pi(x_1h_j) \tensor{A} x_2k_j,
\end{equation}
where $\pi \colon  \cH \to \cH/\cH B^+$ is the canonical projection. Moreover, $B\subseteq \cH$ is a right $\cH$-comodule with a right $A$-linear coaction $\delta \colon  \Mto{B} \to \Mto{B} \tensor{A} \sMto{\cH}$ such that \eqref{eq:comodulesubring} commutes, that is,
for every $b \in B$ we have
\begin{equation}\label{eq:iotarhoright}
\sum \iota(b_{0}) \tensor{A} b_1 = \sum \iota(b)_1 \tensor{A} \iota(b)_2.
\end{equation}
Therefore, \eqref{eq:Hactionright} induces an action of $B$ on $\cH \tensor{A} \cH$ via restriction of scalars along $\iota$,
\[
B \otimes \left(\cH \tensor{A} \cH\right) \to \cH \tensor{A} \cH, \qquad b \ot \left(\sum_j h_j \tensor{A} k_j\right) \mapsto \sum_{j} \iota(b)_1h_j \tensor{A} \iota(b)_2k_j = \sum_{j} \iota(b_0)h_j \tensor{A} b_{1}k_j,
\]
with respect to which $\Delta\colon  \cH \to \cH \tensor{A} \cH$ becomes left $B$-linear:
\[\Delta\big(\iota(b)h\big) = \sum \iota(b)_1h_1 \tensor{A} \iota(b)_2h_2 = b \cdot \left(\sum h_1 \tensor{A} h_2\right).\]
Similarly, \eqref{eq:Hactionquoright} induces an action of $B$ on $\cH/\cH B^+ \tensor{A} \cH$. In particular, \eqref{eq:Hactionright} and \eqref{eq:Hactionquoright} themselves factor through the tensor product over $B$ (by associativity of the multiplication in $\cH$):
\[\cH \tensor{B} (\cH \tensor{A} \cH) \xrightarrow{\mu} \cH \tensor{A} \cH \qquad \text{and} \qquad \cH \tensor{B} \left(\frac{\cH}{\cH B^+} \tensor{A} \cH\right) \xrightarrow{\mu'} \frac{\cH}{\cH B^+} \tensor{A}\cH .\]
Furthermore, we have a well-defined action of $B$ on $B \tensor{A} \cH$:
\[
B \otimes \left(B \tensor{A} \cH\right) \to B \tensor{A} \cH, \qquad b \ot \left(\sum_j b'_j \tensor{A} h_j\right) \mapsto \sum_{j} b_0b'_j \tensor{A} b_{1}h_j
\]
and $\delta\colon  B \to B \tensor{A} \cH$ becomes left $B$-linear with respect to this $B$-module structure and the regular $B$-action on the domain. By summing up all the informations we collected, we have a commutative diagram of left $\cH$-modules and left $\cH$-module homomorphisms
\begin{equation}\label{eq:superdiagrright}
\begin{gathered}
\xymatrix @C=45pt{
 & \cH \tensor{B} \cH \ar[dr]^-{\cH \tensor{B} \Delta} & & \\
\cH \tensor{B} B \ar[ur]^-{\cH \tensor{B} \iota} \ar[r]_-{\cH \tensor{B} \delta} \ar[d]_-{\nu} & \cH \tensor{B} \left(B \tensor{A} \cH\right) \ar[r]_-{\cH \tensor{B} \left(\iota \tensor{A} \cH\right)} & \cH \tensor{B} \left(\cH \tensor{A} \cH\right) \ar[d]^-{\mu} \ar[r]^-{\cH \tensor{B} \left(\pi \tensor{A} \cH\right)} & {\displaystyle \cH \tensor{B} \left(\frac{\cH}{\cH B^+} \tensor{A} \cH\right)} \ar[d]^-{\mu'} \\
\cH \ar[rr]_-{\Delta} & & \cH \tensor{A} \cH \ar[r]_-{\pi \tensor{A} \cH} & {\displaystyle \frac{\cH}{\cH B^+} \tensor{A} \cH}.
}
\end{gathered}
\end{equation}
In view of this, for every $\sum_i h_i \iota(b_i) \in \cH B$ we have
\[
\begin{aligned}
(\pi \tensor{A} \cH)\left(\Delta\left(\sum_i h_i \iota(b_i)\right)\right) & \stackrel{\eqref{eq:superdiagrright}}{=} \mu' \Big(\cH \tensor{A} \big(\left(\pi \tensor{A} \cH\right)\circ \Delta \circ \iota\big)\Big)\left(\sum_i h_i \tensor{B} b_i\right) \\
 & \stackrel{\eqref{eq:piiotaDelta}}{=} \sum_i h_i \cdot \Big(\pi(1_\cH)\tensor{A} \iota(b_i)\Big) \stackrel{\eqref{eq:Hactionquoright}}{=} \sum_i \pi\left((h_i)_1\right)  \tensor{A} (h_i)_2\iota(b_i)
\end{aligned}
\]
and hence for all $\sum_i h_i \iota(b_i) \in \cH B^+$ we have
\[
(\pi \tensor{A} \pi)\left(\Delta\left(\sum_i h_i \iota(b_i) \right)\right) = \sum_i \pi\left((h_i)_1\right) \tensor{A} \pi\Big((h_i)_2 \iota(b_i)\Big) \stackrel{\eqref{eq:piiotaB}}{=} 0.
\]
Since, in addition,
\[
\varepsilon\left(\sum_i h_i \iota(b_i)\right) = \sum_i \varepsilon(h_i \iota(b_i)) \stackrel{\ref{item:B5}}{=} \sum_i \varepsilon\Big(h_i \iota \big(t'\left(\op{\varepsilon'(b_i)}\right)\big)\Big) = 0,
\]
we have that $\cH B^+$ is a two-sided coideal (in view of \cite[17.14]{BrWi}, for instance).
Again it is evident that if $B \subseteq \tilde{B}$, then $\cH B^+ \subseteq \cH \tilde{B}^+$. Thus, the correspondence is inclusion-preserving.
\end{proof}


\subsubsection{The canonical inclusions}

\begin{proposition}\label{prop:unit}
Let $(A,\cH)$ be a left bialgebroid over $A$.
Let $B$ be a right $\cH$-comodule $\op{A}$-subring via $t$ of $\cH$. In view of Proposition \ref{prop:Phiright} we know that $\cH B^+$ is a left ideal $2$-sided coideal in $\cH$. Denote by $\pi\colon \cH \to \cH/\cH B^+$ the canonical projection. Then the inclusion $\iota \colon  B \to \cH$ factors through $\coinv{\cH}{\frac{\cH}{\cH B^+}}$. That is to say, we have an inclusion $B \subseteq \coinv{\cH}{\frac{\cH}{\cH B^+}}$, which we denote by $\eta_B$, such that $\iota = j'\circ \eta_B$ where $j'\colon \coinv{\cH}{\frac{\cH}{\cH B^+}} \to \cH$ is the canonical inclusion.
\end{proposition}

\begin{proof}
Recall that $\coinv{\cH}{\frac{\cH}{\cH B^+}}$ can be realized as the following equalizer in $\rmod{A}$
\[
\xymatrix{
\coinv{\cH}{\frac{\cH}{\cH B^+}} \ar[r]^{j'} & \cH \ar@<+0.5ex>[rr]^-{(\pi\circ t) \tensor{A} \cH} \ar@<-0.5ex>[rr]_-{(\pi\tensor{A}\cH)\Delta} && {\displaystyle \frac{\cH}{\cH B^+} \tensor{A} \cH},
}
\]
see \eqref{eq:eqright}. By \eqref{eq:piiotaDelta}, there exists $\eta_B\colon  B \to \coinv{\cH}{\frac{\cH}{\cH B^+}}$ such that $\iota\colon  B \to \cH$ factors as $j'\circ \eta_B$.
\end{proof}

\begin{proposition}\label{prop:counit}
Let $(A,\cH)$ be a left bialgebroid over $A$ such that $\sM{\cH} = {}_{{A\tensor{}\op{1}}}\cH$ is $A$-flat. Let $I$ be a left ideal $2$-sided coideal in $\cH$. In view of Proposition \ref{prop:OneIsGone} we know that $B \coloneqq \coinv{\cH}{\frac{\cH}{I}}$ is a right $\cH$-comodule $\Ao$-subring of $\cH$ via $t$. Then we have an inclusion $\cH B^+ \subseteq I$ that we denote by $\epsilon_I$.
\end{proposition}

\begin{proof}
Notice that $\cl{\cH} = \cH/I$ is clearly an object in $\LHopf{\cl{\cH}}{\cH}$. Set $\pi\colon \cH \to \cH/I$. If we consider $\coinv{\cl{\cH}}{\cl{\cH}}$, then it is not hard to show that the correspondence defined by
\[
\xymatrix@R=0pt{
\coinv{\cl{\cH}}{\cl{\cH}} \ar@{<->}[r] & A \\
\pi(h) \ar@{|->}[r] & \cl{\varepsilon}\big(\pi(h)\big) = \varepsilon(h) \\
\pi\big(t(\op{a})\big) & \ar@{|->}[l] a
}
\]
is a left $B$-linear isomorphism, where the $B$-module structures are both induced by the $\cH$-module structures via restriction of scalars along $\iota\colon  B \to \cH$. In fact, for every $\pi(h) \in \coinv{\cl{\cH}}{\cl{\cH}}$ we have
\[\sum \pi(h)_1 \tensor{A} \pi(h)_2 = \pi(1) \tensor{A} \pi(h)\]
and so
\[\pi(h) = \left(\cl{\cH} \tensor{A} \cl{\varepsilon}\right)\left(\sum \pi(h)_1 \tensor{A} \pi(h)_2\right) = \left(\cl{\cH} \tensor{A} \cl{\varepsilon}\right)\left(\pi(1) \tensor{A} \pi(h)\right) = \pi\Big(t\big(\op{\cl{\varepsilon}(\pi(h))}\big)\Big) = \pi\Big(t\big(\op{\varepsilon(h)}\big)\Big).\]
This proves that the correspondence is bijective. To prove that it is left $B$-linear, notice that
\[\cl{\varepsilon}\Big(\pi\big(\iota(b)h\big)\Big) = \varepsilon(\iota(b)h) \stackrel{\ref{item:B5}}{=} \varepsilon(\iota(b)s\varepsilon(h)) \stackrel{\eqref{eq:BonA}}{=} b \cdot \varepsilon(h).\]
Since, in addition, we know from a short exact sequence like \eqref{eq:ssesright} that $A \cong B/B^+$ as left $B$-module, we may consider the composition
\begin{equation}\label{eq:smallpsi}
\frac{\cH}{\cH B^+} \cong \cH \tensor{B} \frac{B}{B^+} \cong \cH \tensor{B} A\cong \cH \tensor{B} \coinv{\cl{\cH}}{\cl{\cH}}  \xrightarrow{\theta_{\cl{\cH}}}  \cl{\cH} = \frac{\cH}{I}
\end{equation}
explicitly given by $h + \cH B^+ \mapsto \pi(h) = h + I$ for all $h \in \cH$, where $\theta_{\cl{\cH}}$ is the $\cl{\cH}$-component of the counit \eqref{eq:counit}.
Denote it by $\psi$ and consider the commutative diagram
\begin{equation}\label{eq:diagramcounit}
\begin{gathered}
\xymatrix{
0 \ar[r] & \cH B^+ \ar[r]^-{\subseteq} & \cH \ar[r] \ar@{=}[d] & \displaystyle \frac{\cH}{\cH B^+} \ar[r] \ar[d]^-{\psi} & 0 \\
0 \ar[r] & I \ar[r]_-{\subseteq} & \cH \ar[r]_-{\pi} & \displaystyle \frac{\cH}{I} \ar[r] & 0.
}
\end{gathered}
\end{equation}
It follows that $\cH B^+ \subseteq I$ as claimed.
\end{proof}

\begin{theorem}\label{thm:adjunction}
Let $(A,\cH)$ be a left bialgebroid such that $\sM{\cH}$ is $A$-flat.
\begin{enumerate}[leftmargin=0.7cm]
\item If $I$ is a left ideal $2$-sided coideal in $\cH$ and if $B \coloneqq \coinv{\cH}{\frac{\cH}{I}}$, then $B = \coinv{\cH}{\frac{\cH}{\cH B^+}}$. That is, $\Psi\Phi\Psi(I) = \Psi(I)$.
\item If $B$ is a right $\cH$-comodule $\Ao$-subring of $\cH$ via $t$ and if $I \coloneqq \cH B^+$, then $I = \cH \left(\coinv{\cH}{\frac{\cH}{I}}\right)^+$. That is, $\Phi\Psi\Phi(B) = \Phi(B)$.
\end{enumerate}
In other words, $\Phi$ and $\Psi$ form a \emph{monotone Galois connection} (or, equivalently, an adjunction) between the two lattices.
\end{theorem}

\begin{proof}
Flatness is needed to apply Proposition \ref{prop:OneIsGone}.
Recall that for any left ideal $2$-sided coideal $I$ and for any right $\cH$-comodule $\Ao$-subring of $\cH$ via $t$ we have that $B \subseteq \coinv{\cH}{\frac{\cH}{\cH B^+}}$ via $\eta_B$ and $\cH \left(\coinv{\cH}{\frac{\cH}{I}}\right)^+ \subseteq I$ via $\epsilon_I$, by Proposition \ref{prop:unit} and Proposition \ref{prop:counit} respectively.
\begin{enumerate}[leftmargin=0.7cm]
\item Let $I$ be a left ideal $2$-sided coideal in $\cH$ and let $B \coloneqq \coinv{\cH}{\frac{\cH}{I}} = \Psi(I)$. Since $\cH B^+ \subseteq I$, we have a surjective morphism of left $\cH$-module corings $\cH/\cH B^+ \twoheadrightarrow \cH/I$ such that
\[
\xymatrix @=15pt{
 & \cH \ar@{->>}[dl] \ar@{->>}[dr] & \\
{\displaystyle \frac{\cH}{\cH B^+}} \ar@{->>}[rr] & & {\displaystyle \frac{\cH}{I}}
}
\]
commutes and hence 
\[B \subseteq \coinv{\cH}{\frac{\cH}{\cH B^+}} \subseteq \coinv{\cH}{\frac{\cH}{I}} = B.\]
\item Let $B$ be a right $\cH$-comodule $\Ao$-subring of $\cH$ via $t$ and set $I \coloneqq \cH B^+ = \Phi(B)$. Since $B \subseteq \coinv{\cH}{\frac{\cH}{\cH B^+}} = \coinv{\cH}{\frac{\cH}{I}}$ we have that
\[I = \cH B^+ \subseteq \cH \left(\coinv{\cH}{\frac{\cH}{I}}\right)^+ \subseteq I. \qedhere\]
\end{enumerate} 
\end{proof}

\begin{corollary}
Let $(A,\cH)$ be a left bialgebroid such that $\sM{\cH}$ is $A$-flat. For $I$ a left ideal $2$-sided coideal in $\cH$ and $B$ a right $\cH$-comodule $\op{A}$-subring via $t$ of $\cH$, we have that 
\[\cH B^+ \subseteq I \quad \iff \quad B \subseteq \coinv{\cH}{\frac{\cH}{I}}.\]
\end{corollary}

\begin{proof}
This is simply a restatement of the fact that $\Phi$ is left adjoint to $\Psi$.
\end{proof}


\subsection{The Hopf algebroid case and the bijective correspondence}\label{ssec:Galois}

In case the bialgebroid $(A,\cH)$ is a left Hopf algebroid, one can obtain finer results on the correspondence between left ideal two-sided coideals and right comodule subrings than those in Proposition \ref{prop:OneIsGone} and Proposition \ref{prop:Phiright}, as we are going to show in the present subsection.

\begin{remark}
Let $(A,\cH)$ be a left bialgebroid over $A$. We may consider the tensor product
\[
\cH \tensor{\op{A}} \cH = \Mt{\cH} \tensor{\op{A}} \tM{\cH} \coloneqq \frac{\cH \otimes \cH}{\Big\langle xt(\op{a}) \otimes y - x \otimes t(\op{a})y ~\big\vert~ a \in A, x,y \in \cH \Big\rangle}.
\]
It becomes an $\op{A}$-bimodule via
\[
\op{a}~\cl{\wla}~\left(x \tensor{\op{A}} y\right) = xs(a) \tensor{\op{A}} y \qquad \text{and} \qquad \left(x \tensor{\op{A}} y\right)~\cl{\wra}~\op{a} = x \tensor{\op{A}} s(a)y
\]
for all $x,y \in \cH$, $a \in A$. Inside $\cH \tensor{\op{A}} \cH$, we isolate the distinguished subspace
\[
\cH \tak{\op{A}} \cH \coloneqq \left\{\left.\sum_ix_i \tensor{\op{A}} y_i ~\right|~ \sum_it(\op{a})x_i \tensor{\op{A}} y_i = \sum_ix_i \tensor{\op{A}} y_it(\op{a})\right\}.
\]
It is an $\op{A}$-subbimodule and a $\K$-algebra with unit $1 \tensor{\op{A}} \op{1}$ and multiplication
\begin{equation}\label{eq:multAop}
\left(\sum_ix_i \tensor{\op{A}} y_i\right)\left(\sum_ju_j \tensor{\op{A}} v_j\right) \coloneqq \sum_{i,j}x_iu_j \tensor{\op{A}} v_jy_i,
\end{equation}
which acts from the right on $\cH \tensor{\op{A}} \cH$ via 
\begin{equation}\label{eq:actAop}
\left(\cH \tensor{\op{A}} \cH\right)\otimes\left(\cH \tak{\op{A}} \cH\right) \to \cH \tensor{\op{A}} \cH, \qquad \left(\sum_i x_i \tensor{\op{A}} y_i\right)\otimes\left(\sum_j u_j \tensor{\op{A}} v_j\right) \to \sum_{i,j} x_iu_j \tensor{\op{A}} v_jy_i.
\end{equation}
If $\cH$ is, in addition, a left Hopf algebroid in the sense of \cite[Theorem and Definition 3.5]{Schauenburg} with canonical map
\begin{equation}\label{Eq:betamap}
\beta\colon  \Mt{\cH} \tensor{\op{A}} \tM{\cH} \to \Mto{\cH} \tensor{A} \sM{\cH}, \qquad x \tensor{\op{A}} y \mapsto \sum x_1 \tensor{A} x_2y,
\end{equation}
then the assignment
\[\cH \to \Mt{\cH} \tensor{\op{A}} \tM{\cH}, \qquad x \mapsto \beta^{-1}(x \tensor{A} 1_\cH) \eqqcolon \sum x_{+} \tensor{\op{A}} x_{-}\]
induces a morphism of $\K$-algebras
\begin{equation}\label{Eq:gammamap}
\gamma\colon \cH \to \cH \tak{\op{A}} \cH, \qquad x \mapsto \beta^{-1}(x \tensor{A} 1_\cH).
\end{equation}
The fact that $\gamma$ lands into $\cH \tak{\op{A}} \cH$ is \cite[(3.3)]{Schauenburg} and the fact that it is multiplicative and unital is \cite[(3.4) and (3.5)]{Schauenburg}. The map $\gamma$ is referred to as the \emph{the (left) translation map}.  In particular,
\begin{equation}\label{eq:hh'+-}
\sum (xy)_{+}\tensor{A} (xy)_{-} = \sum x_{+}y_{+}\tensor{A} y_{-}x_{-}
\end{equation}
for all $x,y \in \cH$.
\end{remark}

\begin{lemma}\label{lem:xiright}
Let $(A,\cH)$ be a left Hopf algebroid over $A$ such that $\sM{\cH} = {}_{{A\tensor{}\op{1}}}\cH$ is $A$-flat. Let $B$ be a right $\cH$-comodule $\op{A}$-subring of $\cH$ via $t$ such that for all $b \in B$, there exists $\sum_i b_i \tensor{\op{A}} h_i \in B \tensor{\op{A}} \cH$ such that
\begin{equation}\label{eq:Bbeta}
\beta^{-1}\left(\iota(b)\tensor{A}1_{\cH}\right) = \sum_i \iota(b_i) \tensor{\op{A}} h_i.
\end{equation}
Then the canonical isomorphism
\[\beta\colon \Mt{\cH} \tensor{\op{A}} \tM{\cH} \to \Mto{\cH} \tensor{A} \sM{\cH}, \qquad x \tensor{\op{A}} y \mapsto \sum x_1 \tensor{A} x_2y\]
induces an isomorphism 
\[
\xi\colon {\cH_{\iota}} \tensor{B} {_\iota\cH} \to \Mto{\frac{\cH}{\cH B^+}} \tensor{A} \sM{\cH}.
\]
\end{lemma}

\begin{proof}
Consider the composition
\[\Mt{\cH} \tensor{\op{A}} \tM{\cH} \xrightarrow{\beta} \Mto{\cH} \tensor{A} \sM{\cH} \xrightarrow{\pi \tensor{A}\cH} \frac{\cH}{\cH B^+} \tensor{A} \cH.\]
It satisfies
\begin{align*}
\left(\pi \tensor{A}\cH\right)\beta\left(x\iota(b) \tensor{A} y \right) & = \sum \pi\left(x_1\iota(b)_1\right) \tensor{A} x_2\iota(b)_2y \stackrel{\eqref{eq:Hactionquoright}}{=} x \cdot \left(\sum \pi\left(\iota(b)_1\right) \tensor{A} \iota(b)_2y\right) \\
& \stackrel{\eqref{eq:iotarhoright}}{=} x \cdot \left(\sum \pi\left(\iota(b_0)\right) \tensor{A} b_1y\right) \stackrel{\eqref{eq:piiotaB}}{=} x \cdot \left(\sum \pi\iota\left(t'\left(\op{\varepsilon'(b_0)}\right)\right) \tensor{A} b_1y\right) \\
& = x \cdot \left(\sum \pi(1_{\cH}) \tensor{A} \bigg(s\Big(\varepsilon\big(\iota(b_0)\big)\Big)\bigg)b_1y\right) \\
& \stackrel{\eqref{eq:iotarhoright}}{=} x \cdot \left(\sum \pi(1_{\cH}) \tensor{A} \bigg(s\Big(\varepsilon\big(\iota(b)_1\big)\Big)\bigg)\iota(b)_2y\right) \\
& \stackrel{\eqref{eq:Hactionquoright}}{=} \sum \pi(x_1) \tensor{A} x_2\iota(b)y = \left(\pi \tensor{A}\cH\right)\beta\left(x \tensor{A} \iota(b)y \right)
\end{align*}
for all $x,y \in \cH$, $b \in B$, thus it factors through
\[\xi\colon {\cH_{\iota}} \tensor{B} {_\iota\cH} \to \Mto{\frac{\cH}{\cH B^+}} \tensor{A} \sM{\cH}.\]
In the other direction, consider the composition
\begin{equation}\label{eq:xiinv}
\Mto{\cH} \tensor{A} \sM{\cH} \xrightarrow{\beta^{-1}} \Mt{\cH} \tensor{\op{A}} \tM{\cH} \xrightarrow{p} {\cH_{\iota}} \tensor{B} {_\iota\cH},
\end{equation}
where $p$ is the canonical projection. Since $\sM{\cH}$ is $A$-flat, $\cH B^+ \tensor{A} \cH \subseteq \cH \tensor{A} \cH$ and, in fact, $\cH B^+ \tensor{A} \cH = \ker\left(\pi \tensor{A} \cH\right)$. For all $\sum_j x_j\iota(b_j) \tensor{A} y_j \in \cH B^+ \tensor{A} \cH$ we have
\begin{align*}
\beta^{-1}\left(\sum_j x_j\iota(b_j) \tensor{A} y_j\right) & = \sum_j \left(x_j\iota(b_j)\right)_{+} \tensor{\op{A}} \left(x_j\iota(b_j)\right)_{-}y_j \\
 & \stackrel{\eqref{eq:hh'+-}}{=} \sum_j \left(x_j\right)_{+}\left(\iota(b_j)\right)_{+} \tensor{\op{A}} \left(\iota(b_j)\right)_{-}\left(x_j\right)_{-}y_j \\
 & \stackrel{\eqref{eq:actAop}}{=} \sum_j\left(1_\cH \tensor{\op{A}} y_j\right)\left(\sum\left(x_j\right)_{+}\left(\iota(b_j)\right)_{+} \tensor{\op{A}} \left(\iota(b_j)\right)_{-}\left(x_j\right)_{-}\right) \\
 & \stackrel{\eqref{eq:multAop}}{=} \sum_j\left(1_\cH \tensor{\op{A}} y_j\right)\left(\sum\left(x_j\right)_{+} \tensor{\op{A}} \left(x_j\right)_{-}\right)\left(\sum\left(\iota(b_j)\right)_{+} \tensor{\op{A}} \left(\iota(b_j)\right)_{-}\right) \\
 & \stackrel{\eqref{eq:Bbeta}}{=} \sum_j\left(1_\cH \tensor{\op{A}} y_j\right)\left(\sum\left(x_j\right)_{+} \tensor{\op{A}} \left(x_j\right)_{-}\right)\left(\sum_i \iota\left(b_{j,i}\right) \tensor{\op{A}} h_{j,i}\right) \\
 & = \sum_{i,j} \left(x_j\right)_{+}\iota\left(b_{j,i}\right) \tensor{\op{A}} h_{j,i}\left(x_j\right)_{-}y_j.
\end{align*}
Therefore,
\[
p\beta^{-1}\left(\sum_j x_j\iota(b_j) \tensor{A} y_j\right) = \sum_{i,j} \left(x_j\right)_{+}\iota\left(b_{j,i}\right) \tensor{B} h_{j,i}\left(x_j\right)_{-}y_j  = \sum_{j} \left(x_j\right)_{+} \tensor{B} \left(\sum_i\iota\left(b_{j,i}\right)h_{j,i}\right)\left(x_j\right)_{-}y_j 
\]
and since
\[\sum_i\iota\left(b_{j,i}\right)h_{j,i} \stackrel{\eqref{eq:Bbeta}}{=} \sum \iota(b_j)_{+}\iota(b_j)_{-} \stackrel{\text{\cite[(3.9)]{Schauenburg}}}{=} s\varepsilon\iota\left(b_j\right) = 0 \]
for all $j$, we conclude that \eqref{eq:xiinv} induces
\[\widetilde{\xi}\colon  \Mto{\frac{\cH}{\cH B^+}} \tensor{A} \sM{\cH} \to {\cH_{\iota}} \tensor{B} {_\iota\cH}\]
which is inverse to $\xi$.
\end{proof}

In view of Lemma \ref{lem:xiright}, one obtains the following improvement to Proposition \ref{prop:unit}.

\begin{theorem}\label{thm:PsiPhiright}
Let $(A,\cH)$ be a left Hopf algebroid over $A$ such that $\sM{\cH} = {}_{{A\tensor{}\op{1}}}\cH$ is $A$-flat. 
Let $B$ be a right $\cH$-comodule $\op{A}$-subring via $t$ of $\cH$ such that $\cH$ is pure over $B$ on the right and such that $\gamma(B)\subseteq B \tensor{\op{A}}\cH$. 
Then $B = \coinv{\cH}{\frac{\cH}{\cH B^+}}$, that is $\Psi\Phi(B) = B$.
\end{theorem}

\begin{proof}
On the one hand, purity of $\cH$ on $B$ entails that
\[
\xymatrix{
B \ar[r]^{\iota} & \cH \ar@<-0.5ex>[rr]_-{x \mapsto x \tensor{B} 1} \ar@<+0.5ex>[rr]^-{x \mapsto 1 \tensor{B} x} && {\cH_\iota} \tensor{B} {_\iota\cH}
}
\]
is an equalizer diagram in $\vectk$ (see Corollary \ref{cor:eqpure}) and, on the other hand, it entails that $\iota \tensor{\op{A}} \cH$ is an injective morphism (see Corollary \ref{cor:pure}), so that the condition $\gamma(B)\subseteq B \tensor{\op{A}}\cH$ makes sense. By definition, $\coinv{\cH}{\frac{\cH}{\cH B^+}}$ can be realized as the following equalizer in $\vectk$
\[
\xymatrix{
\coinv{\cH}{\frac{\cH}{\cH B^+}} \ar[r]^{j'} & \cH \ar@<+0.5ex>[rr]^-{(\pi\circ t) \tensor{A} \cH} \ar@<-0.5ex>[rr]_-{(\pi\tensor{A}\cH)\Delta} && {\displaystyle \Mto{\frac{\cH}{\cH B^+}} \tensor{A} \sM{\cH}},
}
\]
see \eqref{eq:eqright}. Now, commutativity of the following diagram
\[
\xymatrix{
B \ar[d]_-{\eta_B} \ar[r]^{\iota} & \cH \ar@{=}[d] \ar@<-0.5ex>[rr]_-{x \mapsto x \tensor{B} 1} \ar@<+0.5ex>[rr]^-{x \mapsto 1 \tensor{B} x} && {\cH_\iota} \tensor{B} {_\iota\cH} \ar[d]^-{\xi} \\
\coinv{\cH}{\frac{\cH}{\cH B^+}} \ar[r]_{j'} & \cH \ar@<+0.5ex>[rr]^-{\pi t \tensor{A} \cH} \ar@<-0.5ex>[rr]_-{(\pi\tensor{A}\cH)\Delta} && {\displaystyle \Mto{\frac{\cH}{\cH B^+}} \tensor{A} \sM{\cH}}
}
\]
together with bijectivity of $\xi$ (Lemma \ref{lem:xiright}) entails that $\eta_B$ is bijective and hence $B = \coinv{\cH}{\frac{\cH}{\cH B^+}}$. 
\end{proof}

At the same time, one may obtain the following improvement to Proposition \ref{prop:counit} by taking advantage of Theorem \ref{thm:super!}.

\begin{theorem}\label{thm:bah}
Let $(A,\cH)$ be a left Hopf algebroid over $A$ such that $\sM{\cH} = {}_{{A\tensor{}\op{1}}}\cH$ is $A$-flat. Let $I\subseteq \cH$ be a left ideal $2$-sided coideal such that the extension-of-scalars functor $\cH \tensor{B}-$ is comonadic, where $B \coloneqq \coinv{\cH}{\frac{\cH}{I}}$. Then $I = \cH B^+$, that is to say, $\Phi\Psi(I)=I$.
\end{theorem}

\begin{proof}
First of all, observe that the comonadicity of $\cH \tensor{B}-$ entails that $\iota \colon B \to \cH$ is pure as a morphism of right $\op{A}$-modules, in view of \cite[\S5.3, Theorem]{JanelidzeTholen} and Corollary \ref{cor:pure}. Thus, let us start by proving that $\gamma(B)\subseteq B \tensor{\op{A}}\cH$ (which now makes sense, because $\iota \tensor{\op{A}}\cH$ is injective by purity). 
In a nutshell, the following diagram whose rows are equalizers commutes sequentially
\[
\xymatrix @C=40pt{
B \ar[r] & \cH \ar@<+0.5ex>[rr]^-{\pi t \, \tensor{A}\cH} \ar@<-0.5ex>[rr]_-{(\pi \tensor{A} \cH) \, \Delta} \ar[d]_-{\gamma} & & {\displaystyle \Mto{\frac{\cH}{I}} \tensor{A} \sM{\cH} } \ar[d]^-{\frac{\cH}{I} \tensor{A} \gamma } \\
\Mt{B} \tensor{\op{A}} \tM{\cH} \ar[r] & \Mt{\cH} \tensor{\op{A}} \tM{\cH} \ar@<+0.5ex>[rr]^-{\pi t \, \tensor{A}\cH \tensor{\op{A}} \cH} \ar@<-0.5ex>[rr]_-{(\pi \tensor{A} \cH) \, \Delta \tensor{\op{A}} \cH} & & {\displaystyle \Mto{\frac{\cH}{I}} \tensor{A} \sMt{\cH} \tensor{\op{A}} \tM{\cH} }
}
\]
In more detail, we want show that for any $b\in B$, $\gamma (b)= \sum b_{+}\tensor{\op{A}} b_{-}\in B \tensor{\op{A}} \cH$. 
Since $\iota\colon B \to \cH$ is pure, it follows that $ B \tensor{\op{A}}\cH = \coinv{(\cH \tensor{\op{A}} \cH)}{\frac{\cH}{I}}$ where $\cH \tensor{\op{A}} \cH$ is considered as a left $\cH/I$-comodule via $(\pi\tensor{A}\id_{\cH})\Delta \tensor{\op{A}}\id_{\cH}$. On the other hand, we have that $\sum b_{1}\tensor{A} b_{2} \in 1 \tensor{A} B + \tM{I} \tensor{A} \cH$. By equation (3.6) of \cite{Schauenburg} we have that 
\[
\sum b_{+1}\tensor{A} b_{+2}\tensor{\op{A}} b_{-} = \sum b_{1}\tensor{A} b_{2+} \tensor{\op{A}} b_{2-}\in 1 \tensor{A} \sum b_{+}\tensor{\op{A}} b_{-} + \tM{I}\tensor{A}\cH \tensor{\op{A}} \cH.
\]
Hence $\gamma (b)\in \coinv{(\cH \tensor{\op{A}} \cH)}{\frac{\cH}{I}}$. 

Now, the additional condition on $\gamma(B)$ we proved above ensures that $\xi$ of \eqref{eq:xiisback} is an isomorphism in view of Lemma \ref{lem:xiright}. In particular, $\cK$ of Theorem \ref{thm:super!} is an equivalence of categories. Therefore, the morphism $\theta_{\cl{\cH}}$ in \eqref{eq:smallpsi} from the proof of Proposition \ref{prop:counit} is an isomorphism and hence the vertical arrows in \eqref{eq:diagramcounit} are all isomorphisms.
It follows that $I = \cH B^+$ as claimed.
\end{proof}

Summing up, we have the first main result of this work. The functor $\cH \tensor{B} -$ in the statement below denotes the extension of scalars functor ${_B{\sf Mod}}\to{_{\cH}{\sf Mod}}$.

\begin{theorem}\label{mainthm:Galois}
Let $(A,\cH)$ be a left Hopf algebroid over $A$ such that $\sM{\cH} = {}_{{A\tensor{}\op{1}}}\cH$ is $A$-flat. We have a well-defined inclusion-preserving bijective correspondence
\[
\begin{gathered}
\xymatrix @R=0pt{
{\left\{ \begin{array}{c} \text{left ideal } 2\text{-sided coideals } I \text{ in } \cH \\ \text{such that } \cH \tensor{B}- \text{ is comonadic,} \\ \text{where } B \coloneqq \coinv{\cH}{\frac{\cH}{I}}
\end{array} \right\}} \ar@{<->}[r] & {\left\{ \begin{array}{c} \text{right } \cH\text{-comodule } \op{A}\text{-subrings } B \text{ of } \cH \\   \text{via } t \text{ such that } \cH \tensor{B} - \text{ is comonadic} \\ \text{and } \gamma(B)\subseteq B \tensor{\op{A}}\cH \end{array} \right\}} \\
I \ar@{|->}[r] & \coinv{\cH}{\frac{\cH}{I}} \\
\cH B^+ & B \ar@{|->}[l]
}
\end{gathered}
\]
\end{theorem}

\begin{proof}
It follows from Theorem \ref{thm:PsiPhiright} (in view of \cite[\S5.3, Theorem]{JanelidzeTholen}) and Theorem \ref{thm:bah}.
\end{proof}

\begin{corollary}[of Theorem \ref{mainthm:Galois}]\label{maincor:Galois}
Let $(A,\cH)$ be a left Hopf algebroid over $A$ such that $\sM{\cH} = {}_{{A\tensor{}\op{1}}}\cH$ is $A$-flat. We have a well-defined inclusion-preserving bijective correspondence
\[
\begin{gathered}
\xymatrix @R=0pt{
{\left\{ \begin{array}{c} \text{left ideal } 2\text{-sided coideals } I \text{ in } \cH \\ \text{such that } \cH \text{ is faithfully flat} \\ \text{over } \coinv{\cH}{\frac{\cH}{I}} \text{ on the right}
\end{array} \right\}} \ar@{<->}[r] & {\left\{ \begin{array}{c} \text{right } \cH\text{-comodule } \op{A}\text{-subrings } B \text{ of } \cH \\ \text{ via } t  \text{ such that } \cH_B \text{ is faithfully flat} \\ \text{and } \gamma(B)\subseteq B \tensor{\op{A}}\cH \end{array} \right\}} \\
I \ar@{|->}[r] & \coinv{\cH}{\frac{\cH}{I}} \\
\cH B^+ & B \ar@{|->}[l]
}
\end{gathered}
\]
\end{corollary}

\begin{proof}
It follows from Theorem \ref{mainthm:Galois}, once recalled that if $\cH_B$ is faithfully flat, then $\cH \tensor{B} -$ is comonadic.
\end{proof}

\begin{remark}
The technical-looking condition ``$\gamma(B)\subseteq B \tensor{\op{A}}\cH$'' is automatically satisfied in most of the cases of interest. Apart from the many particular cases we consider below, in general it holds in case the Hopf algebroid has an antipode $S:\cH\to \cH$, i.e.\ if it is a Hopf algebroid in the sense of B\"ohm, see \cite[Definition 4.1]{Bohm:handbook}. Indeed in this case we have that the translation map is given explicitly by $\gamma(b)=b_{(1)}\otimes S(b_{(2)})$ so the condition is clearly fulfilled. 
\end{remark}

The purity of $\cH$ over $B$, for example in Theorem \ref{thm:PsiPhiright}, cannot be avoided in general, as the subsequent example shows.

\begin{example}
Let $A$ be any commutative algebra, $B \coloneqq (A \otimes A)[X]$ and $\cH \coloneqq (A \otimes A)[X^{\pm1}]$. Recall that $A \otimes A$ admits a Hopf algebroid structure as in \cite[Example 3.1]{Lu}:
\[
\begin{gathered}
s_A\colon A \to A \otimes A, \quad a \mapsto a \otimes 1; \qquad t_A\colon  A \to A \otimes A, \quad a \mapsto 1 \otimes a; \\
\varepsilon_A\colon  A \otimes A \to A, \quad a \otimes b \mapsto ab; \\
\Delta_A \colon  A \otimes A \to (A \otimes A) \tensor{A} (A \otimes A), \quad a \otimes b \mapsto (a \otimes 1) \tensor{A} (1\otimes b);  \\
\beta_A\colon  \Mt{(A \otimes A)} \tensor{A} \tM{(A \otimes A)} \to \Mto{(A \otimes A)} \tensor{A} \sM{(A \otimes A)}, \qquad (a \otimes b) \tensor{A} (c \otimes d) \mapsto (a \otimes 1) \tensor{A} (c \otimes bd).
\end{gathered}
\]
In fact, being a commutative Hopf algebroid, $A \otimes A$ admits an antipode
\[\cS_A\colon A\otimes A \to A \otimes A, \qquad a \otimes b \mapsto b \otimes a,\]
such that $\beta_A^{-1}(x \tensor{A} y) = \sum x_1 \tensor{A} \cS_A(x_2)y$, that is
\[\beta_A^{-1}\colon \Mto{(A \otimes A)} \tensor{A} \sM{(A \otimes A)} \to \Mt{(A \otimes A)} \tensor{A} \tM{(A \otimes A)}, \qquad (a \otimes b) \tensor{A} (c \otimes d) \mapsto (a \otimes 1) \tensor{A} (bc \otimes d).\]
We equip $B$ (respectively, $\cH$) with the bialgebroid (respectively, Hopf algebroid) structure coming from the base ring extension of the bialgebra $\K[X]$ (respectively, Hopf algebra $\K[X^{\pm1}]$) along the morphism $\K \to A \otimes A$ (see, for instance, \cite[Example 1.4]{Laiachi-RF} and \cite[Example 3.3]{Laiachi-GT}). This is a particular instance of the Connes-Moscovici bialgebroid construction, in which all the actions of the bialgebra (respectively, Hopf algebra) on the base algebra $A$ are trivial (see \cite{ConnesMoscovici,Kad} for the original constructions and \cite{Saracco-CM} for the straightforward adaptation to the bialgebra setting). Namely, we have
\[
\begin{gathered}
s_B\colon A \to B, \quad a \mapsto a \otimes 1; \qquad t_B\colon  A \to B, \quad a \mapsto 1 \otimes a; \\
\Delta_B \colon  B \to B \tensor{A} B, \quad (a \otimes b)X \mapsto (a \otimes 1)X \tensor{A} (1\otimes b)X; \qquad \varepsilon_B\colon  B \to A, \quad (a \otimes b)X \mapsto ab.
\end{gathered}
\]
and
\[
\begin{gathered}
s_\cH\colon A \to \cH, \quad a \mapsto a \otimes 1; \qquad t_\cH\colon  A \to \cH, \quad a \mapsto 1 \otimes a; \\
\Delta_\cH \colon  \cH \to \cH \tensor{A} \cH, \quad (a \otimes b)X \mapsto (a \otimes 1)X \tensor{A} (1\otimes b)X; \qquad \varepsilon_\cH\colon  \cH \to A, \quad (a \otimes b)X \mapsto ab; \\
\beta_\cH\colon  \Mt{\cH} \tensor{A} \tM{\cH} \to \Mto{\cH} \tensor{A} \sM{\cH}, \qquad (a \otimes b)X^{n} \tensor{A} (c \otimes d)X^{m} \mapsto (a \otimes 1) X^{n} \tensor{A} (c \otimes bd)X^{n+m}.
\end{gathered}
\]
As before, being $\cH$ a commutative Hopf algebroid, it admits an antipode as well
\[\cS_\cH\colon  \cH \to \cH, \qquad (a \otimes b)X \mapsto (b \otimes a)X^{-1}\]
such that
\[\beta_\cH^{-1}\colon \Mto{\cH} \tensor{A} \sM{\cH} \to \Mt{\cH} \tensor{A} \tM{\cH}, \qquad (a \otimes b)X^{n} \tensor{A} (c \otimes d)X^{m} \mapsto (a \otimes 1)X^{n} \tensor{A} (bc \otimes d)X^{m-n}.\]
Observe also that $B^+ = B\cdot \left\{ s(a) - t(a) \mid a \in A \right\} \oplus B\cdot (1-X)$: on the one hand, by \cite[Proposition 4.3]{Saracco-CM} we know that $\left\{ s(a) - t(a) \mid a \in A \right\} = \mathsf{Prim}(B) \subseteq \ker(\varepsilon_B)$ and since $1 - X$ is $(1,X)$-skew primitive, it is in $\ker(\varepsilon_B)$ too. On the other hand, the composition
\[\frac{B}{B\cdot \left\{ X-1, s(a) - t(a) \mid a \in A \right\}} \cong \frac{B/B\cdot (X-1)}{B\cdot \left\{ X-1, s(a) - t(a) \mid a \in A \right\}/B\cdot ( X-1)}\cong \frac{A \otimes A}{\ker(\mu_A)} \cong A,\]
where $\mu_A$ is the multiplication of $A$, coincides exactly with the morphism induced by $\varepsilon_B$.

Now, both $\sM{B}$ and $\sM{\cH}$ are (faithfully) flat over $A$ and $B$ is a right $\cH$-comodule $A$-subring of $\cH$ via $t$. Furthermore, for all $a,b \in A$, $n \in \NN$ we have
\[\gamma\left((a \otimes b)X^n\right) = (a \otimes 1)X^{n} \tensor{A} (b \otimes 1)X^{-n} \in (\iota \tensor{A} \cH)(\Mt{B} \tensor{A} \tM{\cH}).\]
However, $\cH$ cannot be pure over $B$ since $\cH = S^{-1}B$ for $S = \left\{X^n \mid n \in \NN\right\}$ To check this, it is enough to tensor the injection $B \hookrightarrow \cH$ by $A \otimes A \cong B/\langle X\rangle$ over $B$ on the right: $B \tensor{B} (A \ot A) \cong A \ot A$ while $\cH \tensor{B} (A \ot A) = 0$.

To show how Theorem \ref{thm:PsiPhiright} fails, consider a generic element
\[h = \sum_{\substack{z \in \ZZ \\ i \in I}} \left(a_{z,i} \otimes b_{z,i}\right)X^z \in (A \otimes A)[X^{\pm 1}]\]
and denote by $\mathsf{supp}(h)$ the set of the $z \in \ZZ$ such that $X^z$ appears with non-zero coefficient.
Then notice that if $z \geq 0$ then $X^z + \cH B^+ = 1 + \cH B^+$ because
\[1 - X^z = (1 + X + \cdots  + X^{z-1})(1 - X)\]
and if $z < 0$ then $1/X^{|z|} + \cH B^+ = 1 + \cH B^+$ too because
\[1 - \frac{1}{X^{|z|}} = -\frac{1}{X^{|z|}}\left(1 - X^{|z|}\right) \in \cH B^+.\]
Therefore
\[
(\pi \tensor{A} \cH)\Delta_\cH\left(h\right) = \sum_{\substack{z \in \ZZ \\ i \in I}} \Big(\left(a_{z,i} \otimes 1\right)X^z + \cH B^+\Big) \tensor{A} \left(1 \otimes b_{z,i}\right)X^z = 1 \tensor{A} h,
\]
or, equivalently, $\cH = \cH^{\mathrm{co}\frac{\cH}{\cH B^+}}$.
\end{example}



\section{Affine groupoids and commutative Hopf algebroids}\label{sec:1} 

We discuss here the specific case of the aforementioned correspondence for commutative Hopf algebroids and we apply it to the study of affine groupoids.


\subsection{Preliminaries on groupoids and normal subgroupoids}


\subsubsection{Groupoids and subgroupoids}

Recall, for instance from \cite[Chapter I, Definition 1.1]{Mackenzie-old}, that a \emph{groupoid} is a small category (i.e., the class of objects is a set) in which every arrow is invertible. We will usually denote a groupoid $\Gg$ by a pair $\big(\Gg_0,\Gg_1\big)$ of sets: the set of objects $\Gg_0$ and the set of arrows $\Gg_1$. A groupoid $\big(\Gg_0,\Gg_1\big)$ always comes with a (understood) family of maps
\[
\xymatrix @C=50pt{ 
\Gg_0 \ar@<1ex>@{<-}|(.4){\,\sigma\,}[r] \ar@<-1ex>@{<-}|(.4){\,\tau\,}[r] & \ar@{<-}|(.4){\,e\,}[l] \ar@(ul,ur)^{\iota} \Gg_1 \ar@{<-}^-{\bullet}[r] & \Gg_1 \tak{\Gg_0}\Gg_1}
\]
satisfying the expected, reasonable, compatibility conditions. 
For the sake of simplicity, one often denotes a composition $\bullet(f,g)$ by simple juxtaposition $fg$.
A \emph{subgroupoid} of a groupoid $\Gg$ is a subcategory which is also a groupoid.

For a given groupoid $\Gg = \big(\Gg_0, \Gg_1\big)$, the associated  \emph{isotropy groupoid} (also called \emph{inner subgroupoid} in \cite[Chapter I, Definition 2.4]{Mackenzie-old}) is the subgroupoid $\Gg^{{(i)}}\coloneqq \left(\Gg_0,\cup_{x\, \in\, \Gg_0}\Gg_x\right)$, where $\Gg_x$ is the \emph{isotropy group} of $x \in \Gg_0$, that is to say:
\[\Gg_x \coloneqq \big\{g \in \Gg_1 \mid \sigma(g) = x = \tau(g)\big\}.\]
In the isotropy groupoid, the source and the target coincide and we denote them simply by $o\colon \Gg^{{(i)}} \to \Gg_0, g \mapsto \sigma(g) = \tau(g)$. Moreover, there is a canonical morphism of groupoids $\Gg^{{(i)}} \to \Gg$ (the inclusion).


\subsubsection{Groupoid actions and normal subgroupoids}\label{sssec:groupoidactions}

Recall that a  \emph{left $\Gg$-action} of a groupoid $\Gg$ on a set $N$ consists of two maps $\rho\colon  N \to \Gg_{0}$ and $\mu \colon \Gg_{1}\operatorname{\due \times {\sigma} {\rho}} N \to N, \ (g,n) \mapsto gn$, satisfying
\[
\rho(gn) = \tau (g),\qquad  e_{\rho(n)} n =  n,\qquad g'(gn) = (g'g)n,
\]
for all $g , g'\in \Gg_1$ and $n \in N$ (see \cite[Definition 4.1]{Mackenzie-old} for the topological analogue and \cite[Definition 1.6.1]{Mackenzie-new} for the differential one). Given a left $\Gg$-action on $N$, one can consider the \emph{left translation groupoid} $\Gg \lJoin N$ with $\Gg_{1}\, \due \times {\sigma} {\rho} N$ as set of arrows and $N$ as set of objects. This is  the so-called \emph{semi-direct product groupoid}. The source and the target of this groupoid are given by 
\[
\sigma,\tau \colon  \Gg_{1}\, \due \times {\sigma} {\rho} N \to N, \qquad \sigma(g,n) = n, \qquad \tau(g,n) = gn.
\]
The identity of $n \in N$ is $e_n = (e_{\rho(n)}, n)$. The composition is given by 
\[
(g,n) (g',n') = (gg', n') \qquad \textrm{whenever} \ g'n'=n.
\]
Obviously the pair $({\rm pr}_1, \rho)\colon  (\Gg_1 \,\due \times {\sigma} {\rho} N\, , N) \to (\Gg_1, \Gg_0)$ defines a morphism of groupoids. 

\begin{example}[Action groupoid]
Any group $G$ can be considered as a groupoid by taking $\Gg_1 = G$ and $\Gg_0=\{*\}$ (the singleton). If $X$ is a set with an action $G \times X \to X$ of $G$, then one can define the so-called \emph{action groupoid}. It is constructed over $\Gg_1 \coloneqq G \times X$ and $\Gg_0 \coloneqq X$. Source and target are provided by $s(g,x) \coloneqq x$ and $t(g,x) \coloneqq g\cdot x$ for all $x \in X$ and $g \in G$. The identity map sends $x$ to $(e, x)$, where $e$ is the neutral element of $G$. The composition law is given by $(h,y) \bullet (g,x) \coloneqq (hg,x)$ whenever $y = g \cdot x$, and the inverse is defined by $(g,x)^{-1}\coloneqq (g^{-1},g \cdot x)$. Clearly, the pair of maps $({\rm pr}_1, *) \colon (\Gg_1, \Gg_0) \to (G, \{*\})$ defines a morphism of groupoids.
\end{example}

As for groups, conjugation in $\Gg$ can be seen as an action of $\Gg$ on the underlying set of arrow of its isotropy groupoid $\Gg^{{(i)}}$. 
In detail, define the map $\ad \colon \Gg_1 \, \due \times{\sigma} {o} \Gg^{(i)}_1 \to  \Gg^{(i)}_1$ acting as $\ad(g,f) = gfg^{-1}$. One can show that this defines an action, called the \emph{adjoint action}, so that  $\Gg^{(i)}_1$ becomes a left $\Gg$-set.  

A \emph{normal subgroupoid} of $\Gg$ is a subgroupoid $\Nn \hookrightarrow \Gg$ such that 
\begin{enumerate}[label=(NG\arabic*), ref=(NG\arabic*)]\addtocounter{enumi}{-1}
\item\label{item:norm1} $\Nn_0 =\Gg_0$ (i.e., $\Nn$ is \emph{wide} in the sense of \cite[Chapter I, Definition 2.4]{Mackenzie-old});
\item\label{item:norm2} $\Nn_1(x,y)= \emptyset$,  for all $x \neq  y $ in $\Nn_0$ (that is, $\Nn$ is a subgroupoid of $\Gg^{(i)}$);
\item\label{item:norm3} for every $g \in \Gg_1$, we have 
\[
\ad(g)\Big(\Nn_1\big(\sigma(g), \sigma(g)\big)\Big) \subseteq \Nn_1(\tau(g),\tau(g))
\]
as subgroups of $\Gg_1\big(\tau(g),\tau(g)\big)$, where $\ad(g)(f)=\ad(g,f)$ for all $f$ with $\sigma(f) = \tau(f) = \sigma(g)$.
\end{enumerate}
Equivalently, a subgroupoid $\Nn \hookrightarrow \Gg$ is normal (in the above sense) if and only if $\Nn$ is a subgroupoid of $\Gg^{(i)}$,
\[
g\Nn g^{-1} \coloneqq \Big\{ghg^{-1} ~\Big\vert~ h \in \Nn_1\big(\sigma(g),\sigma(g)\big)\Big\}
\]
is non empty and $g\Nn g^{-1} \subseteq \Nn$ for all $g \in \Gg_1$ (see \cite[Lemma 3.1]{Paques-Tamusiunas}).
In particular, we denote again by $o \colon \Nn \to \Nn_0, n \mapsto \sigma(n) = \tau(n)$, the common source and target of a normal subgroupoid.

\begin{remark}
Condition \ref{item:norm2}, which explicitly appears in \cite[\S1]{Brown} and \cite[Definition 2.2.1]{Mackenzie-new}, does not appear in \cite[Chapter I, Definition 2.5]{Mackenzie-old} or \cite[\S3]{Paques-Tamusiunas}. However, as observed in \cite[page 8]{Mackenzie-old}, whether or not a subgroupoid is normal depends only on those of its elements which lie in its isotropy groupoid. 
\end{remark}

The fact that normal subgroups can be characterized as the invariant subgroups under the adjoint action can be extended to the groupoids context, as the following lemma states.  

\begin{lemma}\label{lema:normalsugpd}
Let $\Gg$ be a groupoid and $\Nn \hookrightarrow \Gg$ a subcategory with $\Nn_0=\Gg_0$. Then $\Nn$ is a normal subgroupoid if and only if  $\Nn_1$ is a $\Gg$-equivariant subset of $\Gg^{(i)}_1$.  
\end{lemma}

\begin{proof}
Straightforward.
\end{proof}


\subsubsection{The orbit space}\label{ssec:orbits}
Let $\Gg$ be a groupoid and $N$ be a left $\Gg$-set. For $n \in N$ one defines the orbit of $n$ under $\Gg$ to be
\[\Gg n \coloneqq \left\{gn \in N ~\big\vert~ g \in \sigma^{-1}(\rho(n))\right\}\] 
(see, e.g., \cite[Definition 4.1]{Mackenzie-old}). This induces an equivalence relation on $N$ given by
\[n \sim n' \quad \Leftrightarrow \quad \exists\, g\in \sigma^{-1}(\rho(n)) \text{ such that } n'= gn \quad \Big( \Leftrightarrow \quad \Gg n = \Gg n'\Big).\]
The quotient space with respect to this relation, called the \emph{orbit space} of $N$ with respect to the action of $\Gg$, is denoted by $N/\Gg$ and its elements are the orbits $\Gg n$ for $n \in N$. See, for instance, \cite[Definition 1.67]{Goehle-PhD}. 
In particular, given a groupoid $\Gg$ one can consider the orbit equivalence relation on $\Gg_0$ with respect to the action of $\Gg$ given by:
\[\Gg_1 \operatorname{\due \times {\sigma} {\id}} \Gg_0 \to \Gg_0, \qquad (g,x) \mapsto \tau(g).\]
In this case, $x \sim y$ if and only if there exists a $g \in \Gg_1$ such that $\sigma(g) = x$ and $\tau(g) = y$. The orbit of $x \in \Gg_0$ is
\[\Gg x = \tau \big( \sigma ^{-1}(x)\big).\]
The set of orbits, called the \emph{orbit space}, of $\Gg$ is the quotient set $\Gg_{0}/\sim$, which is often denoted by $\Gg_{0}/\Gg$.
For further details, we refer the reader to \cite[\S1.2]{Goehle-PhD}.

\begin{remark}
For a left $\Gg$-set $N$, the orbit space $N/\Gg$ of $N$ with respect to the action of $\Gg$ can also be described as the orbit space of the left translation groupoid $\Gg \lJoin N$. In fact, for $p \in N$ we have
\begin{align*}
\Gg p & = \tau \big( \sigma ^{-1}(p)\big) = \tau\left(\left\{(g,n) \in \Gg_{1} \operatorname{\due \times {\sigma} {\rho}} N \mid p = n\right\}\right) \\
 & = \tau\big(\left\{(g,p) \in \Gg_{1} \times N \mid \rho(p) = \sigma(g)\right\}\big) = \big\{gp \mid g \in \sigma^{-1}\big(\rho(n)\big)\big\} \qedhere
\end{align*}
\end{remark}

Now we will use this notion to describe the quotient groupoid by a normal subgroupoid as in the classical case of groups. Let us consider a normal subgroupoid $\Nn \hookrightarrow \Gg$. It is clear that $(\Gg_1, \tau )$ is a left $\Nn$-set  via the left action 
\begin{equation}\label{Eq:Naction}
\Nn \operatorname{\due \times {o} {\tau}}  \Gg_1 \to \Gg_1, \qquad (g, f) \mapsto gf.
\end{equation}
We  denote by $\Gg_1/\Nn$ the orbit space of this action (or, equivalently, the orbit space of the associated left translation groupoid $\Nn \lJoin \Gg_1$) and call it the \emph{quotient groupoid} (see \cite[page 9-10]{Mackenzie-old}). As a matter of description, the equivalence class of a given arrow $g \in \Gg_1$ is a ``flower with stalk at $g$'':
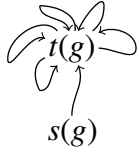
\begin{center}
\begin{tikzpicture}[x=8pt,y=8pt,thick]\pgfsetlinewidth{0.5pt}
\node[inner sep=1pt](1) at (0,0){${s (g)}$};
\node[inner sep=1pt](2) at (0,4) {${t (g)}$};

\draw[->] (1) to [out=90, in=-70] (2);
\draw[-] (2) to [out=0, in=-90] (3,4);
\draw[->] (3,4) to  [out=90, in=30] (2);
\draw[-] (2) to  [out=40, in=-90] (1.5,6);
\draw[->] (1.5,6) to  [out=90, in=70] (2);
\draw[-] (2) to  [out=80, in=0] (-1.5,6);
\draw[->] (-1.5,6) to  [out=-90, in=100] (2);
\draw[-] (2) to  [out=130, in=90] (-3,3.5);
\draw[->] (-3,3.5) to  [out=-90, in=160] (2);
\draw[-] (2) to  [out=200, in=120] (-1.5,2);
\draw[->] (-1.5,2) to  [out=0, in=230] (2);
\end{tikzpicture}
\end{center}
Two arrows are equivalent if they ``support the same flower''.
The action \eqref{Eq:Naction} is in fact a left one. There is also a right $\Nn$-action on $(\Gg_1, \sigma )$, which leads to the groupoid $\Gg_1 \rJoin \Nn$ and, in principle, to another orbit space $\Gg_1/\Nn$. It is easy to verify that the two orbit spaces coincide, since $\Nn$ is normal. The following result is expected (see also \cite[page 9-10]{Mackenzie-old}).

\begin{lemma}[{\cite[Proposition 2.16]{Barbaran/Kaoutit}}]\label{lema:quotient}
Let $\Nn \hookrightarrow \Gg$ be a normal subgroupoid. Then $\Gg/\Nn\coloneqq (\Gg_1/\Nn, \Gg_0)$ have a structure of groupoid such that there is a chain of morphisms
$$
\xymatrix{\Nn\, \ar@{^{(}->}[r] & \Gg \ar@{->}[r]^-{\varpi} & \Gg/\Nn.}
$$
Furthermore, any morphism $\Gg \to \Hh$ whose kernel contains $\Nn$ factors through $\varpi$.
\end{lemma}


\subsection{Commutative Hopf algebroids}\label{ssec:halgd}

In view of the easier expression of commutative Hopf algebroids with respect to general bialgebroids and Hopf algebroids, we begin by recalling explicitly their definition, both for the convenience of the reader and for future reference.

As a matter of notation, if $A$ and $R$ are commutative $\K$-algebras, then we will write $A(R)$ for $\calg_\K(A,R)$.


\subsubsection{Commutative Hopf algebroids and Hopf ideals}

\begin{definition}
A \emph{commutative Hopf algebroid} is a cogroupoid object in the category $\calg_{\K }$ of commutative algebras (see e.g.~\cite{Ravenel:1986}). Thus, a commutative Hopf algebroid consists of a pair of commutative $\K$-algebras $\left( A,\mathcal{H}\right) $ together with a diagram of algebra maps
\[
\xymatrix@C=45pt{ A \ar@<1ex>@{->}|(.4){\,s\,}[r] \ar@<-1ex>@{->}|(.4){\,t\,}[r] & \ar@{->}|(.4){ \,\varepsilon\,}[l] \ar@(ul,ur)^{{\cS}} \cH \ar@{->}^-{{\Delta}}[r] & \cH \tensor{A}\cH}
\]
(where the $A$-bimodule structure on $\cH$ used to perform the tensor product is $\sMt{\cH}$) such that
\begin{enumerate}[label=(CH\arabic*),ref=(CH\arabic*)]
\item The datum $(\cH, \Delta, \varepsilon)$ is an $A$-coring structure on $\sMt{\cH}$. At the level of groupoids, this encodes a unitary and associative composition law between arrows.
\item The antipode $\cS$ satisfies $\cS \circ s=t$, $\cS \circ t = s$ and $\cS^2=\id_{\cH}$, which encode the fact that the inverse of an arrow interchanges source and target and that the inverse of the inverse is the original arrow.
\item The antipode also satisfies $\sum\cS(u_1)u_2=(t\circ \varepsilon)(u)$ and $\sum u_1\cS(u_2)=(s\circ \varepsilon)(u)$, which encode the fact that the composition of a morphism with its inverse on either side gives an identity morphism.
\end{enumerate}
In particular, the inverse of the Hopf-Galois map
\[\beta\colon \Mt{\cH} \tensor{A} \tM{\cH} \to \Mt{\cH} \tensor{A} \sM{\cH}, \qquad u \tensor{A} v \mapsto \sum u_1 \tensor{A} u_2v,\]
is provided by
\[\beta^{-1}\colon  \Mt{\cH} \tensor{A} \sM{\cH} \to \Mt{\cH} \tensor{A} \tM{\cH}, \qquad u \tensor{A}v \mapsto \sum u_1 \tensor{A} \cS(u_2)v.\]
Thus, explicitly, 
\begin{equation}\label{eq:gammacomm}
\gamma\coloneqq \Big(  \cH \xrightarrow{\Delta} \Mt{\cH} \tensor{A} \sM{\cH} \xrightarrow{\cH \tensor{A} \cS} \Mt{\cH} \tensor{A} \tM{\cH} \Big), \qquad  \gamma(u) = \sum u_1 \tensor{A} \cS(u_2).
\end{equation}

A \emph{morphism} of commutative Hopf algebroids $\B{\phi}\colon  (A,\cH) \to (B,\cK)$ is a pair of algebra maps $\left( \phi _{{0}} \colon A \to B,\phi _{{1}} \colon \cH \to \cK\right)$ such that
\begin{gather*}
\phi _{1}\circ s_\cH = s_\cK \circ \phi _{0}, \qquad \phi _{1}\circ t_\cH = t_\cK \circ \phi_{0}, \qquad \varepsilon_\cK \circ \phi _{1} = \phi _{0} \circ \varepsilon_\cH ,  \\
\Delta_\cK \circ \phi _{1} =\chi \circ \left( \phi _{1}\otimes _{A}\phi_{1}\right) \circ \Delta_\cH , \qquad \cS_\cK \circ \phi _{1} =\phi _{1} \circ \cS_\cH,
\end{gather*}
where $\chi \colon \cK \tensor{A} \cK \rightarrow \cK \tensor{B} \cK$ is the obvious projection induced by $\phi_0$, that is $\chi \left(k\tensor{A}k'\right) =k\tensor{B}k'$. If $A = B$ and $\phi_0 = \id_A$, then we say that $\phi_1 : \cH \to \cK$ is a morphism of Hopf algebroids over $A$.
\end{definition}

\begin{remark}\label{rem:flat}
Let $(A, \cH)$ be commutative Hopf algebroid such that $\sM{\cH}$ is flat as an $A$-module (i.e.~the extension $s \colon  A \to \cH$ is flat). In this case $t \colon A \to \cH$ is also flat due to the isomorphism given by the antipode. Therefore both extension $s $ and $t $ are  faithfully flat since both are (left) split morphisms of modules over the base algebra. 
\end{remark}

Similarly to how commutative Hopf algebras over a field $\K$ are the same thing as affine $\K$-group schemes (that is, representable presheaves of groups on $\aff_\K = \calg_{\K }^{\mathrm{op}}$, the opposite of the category of commutative $\K$-algebras), commutative Hopf algebroids are the same thing as affine groupoid schemes (that is, representable presheaves of groupoids on $\aff_\K$).  
Precisely, given a commutative Hopf algebroid $(A, \cH)$ and a commutative $\K$-algebra $R$, we have a groupoid $\xymatrix@C=25pt{A(R) \ar|(0.45){\,e\,}[r] & \cH(R) \ar@<1ex>|(0.45){\,\tau\,}[l] \ar@<-1ex>@{->}|(0.45){\,\sigma\,}[l]} $ by reversing the structures of $(A,\cH)$. We will often adopt this point of view and we will denote by $\Gg_{(A,\cH)} = \big(\Gg_A,\Gg_\cH\big)$ the groupoid scheme defined by $(A,\cH)$, that is to say, 
\[
\Gg_A(R) = A(R) = \calg_\K(A,R) \qquad \text{and} \qquad \Gg_\cH(R) = \cH(R) = \calg_\K(\cH,R),
\]
in order to stress the groupoid structure on the pair $\big(\Gg_A(R),\Gg_\cH(R)\big)$. Observe that, under this perspective,
\begin{equation}\label{eq:comp}
(\psi \bullet \varphi)(h) = \sum \varphi\big(h_1\big)\psi\big(h_2\big)
\end{equation}
for all $(\psi,\varphi) \in \Gg_\cH(R) \operatorname{\due{\times}{s^*}{t^*}} \Gg_\cH(R)$, $h \in \cH$ and $R$ in $\calg_\K$.

\begin{remark}
In particular, $\big(\Gg_A(\K),\Gg_\cH(\K)\big)$ will be called an \emph{affine $\K$-groupoid} and when $A,\cH$ are finitely generated and reduced, then it will be called an \emph{affine algebraic $\K$-groupoid}. Suppose that $\K$ is algebraically closed. Denote by $\aff_1:\calg_\K \to \set$ the affine scheme represented by the polynomial algebra $\K[T]$ and by $\operatorname{Var}_\K$ the category of affine $\K$-varieties (i.e., the category whose objects are $\calg_\K(R,\K)$ for $R$ a finitely generated and reduced commutative algebra and whose morphisms are given by precomposition by morphisms in $\calg_\K$). Then 
\[(A,\cH) \mapsto \big(\Gg_A(\K),\Gg_\cH(\K)\big) \quad \text{and} \quad \big(\Gg_A(\K),\Gg_\cH(\K)\big) \mapsto \Big(\operatorname{Var}_\K\big(\Gg_A(\K),\aff_1(\K)\big),\operatorname{Var}_\K\big(\Gg_\cH(\K),\aff_1(\K)\big)\Big)\]
give an anti-equivalence 
between the category of affine (i.e., finitely generated, commutative, reduced) Hopf algebroids and the category of affine algebraic $\K$-groupoids, where 
\[
\operatorname{Var}_\K\big(\Gg_A(\K),\aff_1(\K)\big) = \calg_\K(\K[T],A).\qedhere
\]
\end{remark}

Let $(\Gg_A,\Gg_\cH)$ be an affine groupoid scheme and let $I\subseteq \cH$ be an ideal. Denote by $\pi \colon \cH \to \cH/I$ the canonical projection and by $\pi^*\colon \calg_\K(\cH/I,-) \to \calg_\K(\cH,-)$ the natural transformation given by precomposition by $\pi$. We say that $(\Gg_A,\Gg_{\cH/I})$ is a \emph{wide closed groupoid subscheme} of $(\Gg_A,\Gg_\cH)$ (or, simply, a \emph{wide closed subgroupoid} or \emph{wide affine subgroupoid} of $(\Gg_A,\Gg_\cH)$) if $(\Gg_A,\Gg_{\cH/I})$ is a groupoid itself and $\pi^*:\Gg_{\cH/I} \hookrightarrow \Gg_\cH$ induces a morphism of groupoids $(\Gg_A,\Gg_{\cH/I}) \to (\Gg_A,\Gg_\cH)$. Since the structure maps of the commutative Hopf algebroid $\cH/I$ should be induced by the structure maps of $\cH$ via $\pi$, $(\Gg_A,\Gg_{\cH/I})$ is a wide subgroupoid of $(\Gg_A,\Gg_\cH)$ if and only if $(A,\cH/I)$ is a commutative Hopf algebroid and $\pi:\cH \to \cH/I$ is a morphism of commutative Hopf algebroids, if and only if the ideal $I$ satisfies
\begin{enumerate}[label=(HI\arabic*),ref=(HI\arabic*)]
\item\label{item:HI1} $\varepsilon(I) = 0$; 
\item\label{item:HI2} $\Delta(I) \subseteq \img(\cH\tensor{A}I + I\tensor{A}\cH)$;
\item\label{item:HI3} $\cS(I) \subseteq I$ (and so $\cS(I) = I$).
\end{enumerate}

An ideal $I \subseteq \cH$ satisfying \ref{item:HI1} and \ref{item:HI2} will be called a \emph{wide bi-ideal} of $(A,\cH)$. When $I$ satisfies \ref{item:HI3} as well, it will be called a \emph{wide Hopf ideal} of $(A,\cH)$. 

\begin{remark}
The terminology ``wide'' introduced above is justified by the fact that if $I \subseteq \cH$ is a wide Hopf ideal, then $\big(\Gg_A(\K),\Gg_{\cH/I}(\K)\big)$ is a wide subgroupoid of $\big(\Gg_A(\K),\Gg_{\cH}(\K)\big)$ in the sense of \cite[Chapter I, Definition 2.4]{Mackenzie-old}, that is to say, the two groupoids share the same set of objects. A more general notion of ideal for a Hopf algebroid would involve considering at the same time an ideal $I_A$ in $A$ and an ideal $I_\cH$ in $\cH$ such that the canonical projections $(\pi_A\colon A \to A/I_A,\pi_\cH\colon \cH \to \cH/I_\cH)$ form a morphism of Hopf algebroids, but we do not discuss this case herein. See Section 3.2 of \cite{ghobadi2021isotopy} for additional details.
\end{remark}

In what follows we are mainly concerned with wide subgroupoids and wide ideals, hence we will often omit the adjective ``wide''.
For the sake of completeness, the following proposition can be proved by finding inspiration from the proof of Proposition \ref{prop:uffachefatica}.

\begin{proposition}
Let $\cH$ and $\cK$ be bialgebroids (resp., Hopf algebroids) over the same base algebra $A$ and let $\phi: \cH \to \cK$ be a morphism of bialgebroids (resp., Hopf algebroids) over $A$. If the two-sided ideal $I \subseteq \cH$ is a bi-ideal (resp., Hopf ideal) of $\cH$ and if we denote by $\pi:\cH \to \cH/I$ the canonical projection, then
\begin{itemize}[leftmargin=0.9cm]
\item $(A,\cH/I)$ has a unique bialgebroid (resp., Hopf algebroid) structure such that $\pi$ becomes a morphism of bialgebroids (resp., Hopf algebroids) over $A$;
\item $\ker(\phi)$ is a bi-ideal (resp., Hopf ideal) of $(A,\cH)$;
\item for any bi-ideal (resp., Hopf ideal) $I \subseteq \ker(\phi)$ there exists a unique morphism of bialgebroids (resp., Hopf algebroids) over $A$, $\tilde{\phi}:\cH/I \to \cK$, such that $\tilde{\phi} \circ \pi = \phi$.
\end{itemize}
\end{proposition}

Now, let $(A,\cH)$ be a commutative Hopf algebroid and consider the ideal $\langle s-t \rangle$ generated by the vector subspace $\{s(a) - t(a) \mid a \in A\}$. Since $s(a) - t(a) \in \cP_{A}(\cH)$ (the space of primitive elements) for every $a \in A$, $\langle s-t \rangle$ is a Hopf ideal in $\cH$. Its associated quotient Hopf algebroid will be denoted by $\iso{\cH} \coloneqq \cH/\langle s-t \rangle$. For $(A, \iso{\cH})$, the source equals the target. That is, it is a commutative Hopf $A$-algebra. Denote by $\eta\colon  A \to \iso{\cH}$ the unit map (that is, $\eta(a) = s(a) + \langle s-t\rangle = t(a) + \langle s-t\rangle$ for all $a \in A$) and by $\iso{\varpi} \colon  \cH \to \iso{\cH}$ the canonical projection sending $x$ to $\bara{x} = x + \langle s-t \rangle$. 

\begin{lemma}\label{lema:Hbar}
Let $(A,\cH)$ be a commutative Hopf algebroid. Then the presheaf of groupoids $\big(\Gg_A, \Gg_\cH^{(i)}\big)$ obtained by considering the isotropy groupoid of $\big(\Gg_A(R),\Gg_\cH(R)\big)$ for every $R$ in $\calg_\K$ is represented by the Hopf $A$-algebra $\iso{\cH}$. In other words, we have an isomorphism 
\[
\Big(A(R), \cH(R)^{{(i)}}\Big) \cong  \Big(A(R),  \iso{\cH}(R)\Big),
\]
natural in $R$, between the isotropy groupoid of $\big(\Gg_A, \Gg_\cH\big)$ and the groupoid $\big(\Gg_A,\Gg_{\iso{\cH}}\big)$.
\end{lemma}

\begin{proof}
Let $R$ be a commutative $\K$-algebra. The fact that the correspondence
\begin{gather*}
\xymatrix @R=0pt{
\iso{\cH}(R) = \calg_\K\left(\displaystyle\frac{\cH}{\langle s-t\rangle},R\right) \ar[r]^-{{\iso{\varpi}}^*} & \cH(R)^{(i)} = \left\{\varphi \in \calg_\K(\cH,R) \mid \varphi\circ s = \varphi\circ t\right\} \\
\phi \ar@{|->}[r] & \phi \circ \iso{\varpi},
}
\end{gather*}
is bijective is almost tautological and it is a morphism of groupoids because it is induced by a morphism of commutative Hopf algebroids. The naturality in $R$ follows from the associativity of the composition.
\end{proof}

Even for Hopf algebroids over noncommutative base algebras, $\cH/\langle s-t \rangle$ forms a natural Hopf algebroid structure and is called the \emph{isotropy quotient} of $\cH$ \cite{ghobadi2021isotopy}. In particular, Lemma~\ref{lema:Hbar} was previously discussed in Example 3.13 of \cite{ghobadi2021isotopy}.


\subsubsection{Comodule algebras and normal Hopf ideals}

With an eye to the definition of normal Hopf ideals, we give the analogue of the adjoint action for groupoids (see \S\ref{sssec:groupoidactions}) in the Hopf algebroid context. To this aim, recall first that a (commutative) \emph{right $\cH$-comodule algebra} for a commutative Hopf algebroid $(A,\cH)$ is a (commutative)  monoid in the symmetric monoidal category $\rcomod{\cH}$ (see \ref{rem:monoidality}). That is, a (commutative) $A$-algebra $\cR$ (with structure map $r\colon  A\to \cR$) which is also a right $\cH$-comodule with coaction $\rcaction{\cR}\colon  \cR \to  \cR\tensor{A}\cH, x \mapsto \sum x_{0} \tensor{A} x_1$, satisfying for all $x, y \in \cR$
\begin{equation}\label{Eq:comodalg}
\rcaction{\cR}(xy) = \sum x_{0}y_{0} \tensor{A} x_{1}y_{1} 
\qquad \mbox{and} \qquad \rcaction{\cR}(1_{\cR})= 1_{\cR}\tensor{A}1_{\cH}.
\end{equation} 
In others words, the coaction $\rcaction{\cR}$ is an $A$-algebra map, where $\cR\tensor{A}\cH$ is seen as an $A$-algebra via $A \to \cR \tensor{A} \cH,a  \mapsto 1_{\cR}\tensor{A} t(a)$. 
A \emph{morphism of right $\cH$-comodule algebras} is an $A$-algebra map which is also a right $\cH$-comodule morphism. 
{\em Left} $\cH$-comodule algebras are analogously defined. 
Henceforth, all comodule algebras will be \emph{commutative} comodule algebras. A trivial example of a comodule algebra is the base algebra $A$ of a Hopf algebroid $(A, \cH)$ itself.

\begin{proposition}\label{prop:groupoidaction}
Let $\Gg_{(A,\cH)} = (\Gg_A,\Gg_\cH)$ be an affine groupoid scheme and let $\Nn_\cB = \calg_\K(\cB,-)$ be an affine scheme, for $\cB$ in $\calg_\K$. Then $\Nn_\cB(R)$ is a left $\Gg_{(A,\cH)}(R)$-set (naturally in $R$ in $\calg_\K$) if and only if $\cB$ is a right $\cH$-comodule algebra.
\end{proposition}

\begin{proof}
If $(\cB,b:A\rightarrow \cB)$ is an $A$-algebra with a right $\cH$-comodule algebra structure given by a coaction $\rcaction{\cB}$, then $\rho_R \coloneqq b^* \colon \calg_\K(\cB,R) \to \calg_\K(A,R)$ and
\[
\Gg_\cH(R) \operatorname{\due{\times}{\sigma_R}{\rho_R}} \Nn_\cB(R) \cong \calg_\K(\cB \tensor{A} \cH,R) \xrightarrow{\rcaction{\cB}^*} \calg_\K(\cB,R) = \Nn_\cB(R)
\]
turn $\Nn_\cB(R)$ into a left $\Gg_{(A,\cH)}(R)$-set. Conversely, if $\Big(\Nn_\cB(R),\rho_R,\mu_R\Big)$ is a left $\Gg_{(A,\cH)}(R)$-set natural in $\calg_\K$, then $b \coloneqq \rho_\cB(\id_\cB)$ and $\rcaction{\cB}\coloneqq \mu_{\cB \tensor{A} \cH}\big(\id_{\cB \tensor{A} \cH}\big)$ make of $\cB$ a right $\cH$-comodule algebra.
\end{proof}

In the setting of Proposition \ref{prop:groupoidaction}, we say that $\Nn_\cB$ is an \emph{affine (left) $\Gg_{(A,\cH)}$-set}.

Note that, for a right $\cH$-comodule algebra $(\cR,r)$, the $\K$-vector subspace 
\[
\coinv{\cR}{\cH} = \{x \in \cR \mid  \rcaction{\cR}(x) = x \tensor{A} 1_{\cH} \}
\]
of $\cH$-coinvariant elements is a $\K$-subalgebra of $\cR$ (this follows easily from \eqref{Eq:comodalg}).

\begin{remark}
In general, $\coinv{\cR}{\cH}$ does not necessarily contain the image $r(A)$ of $A$, unless one makes additional assumptions. For instance, if the source and the target of $\cH$ are equal, then it does.
\end{remark}

Dually to the construction of the left translation groupoid, if $(\cR,r)$ is a right $\cH$-comodule algebra, then $(\cR, \cR\tensor{A}\cH)$ admits a natural structure of commutative Hopf algebroid such that $(r, 1_{{\cR}}\tensor{A}-)\colon (A,\cH) \to (\cR,\cR\tensor{A}\cH)$ is a morphism of Hopf algebroids. Namely,
\begin{gather*}
s\colon \cR \to \cR \tensor{A} \cH, \quad x \mapsto x \tensor{A} 1_\cH, \qquad t\colon \cR \to \cR \tensor{A} \cH, \quad x \mapsto \sum x_0 \tensor{A} x_1 \\
\Delta\colon \cR \tensor{A} \cH \to \big(\cR \tensor{A} \cH\big) \tensor{\cR} \big(\cR \tensor{A} \cH\big), \qquad (x \tensor{A} h) \mapsto \sum (x \tensor{A} h_1) \tensor{\cR} (1_\cR \tensor{A} h_2) \\
\varepsilon\colon \cR \tensor{A} \cH \to \cR, \ x \tensor{A} h \mapsto x r\big(\varepsilon(h)\big), \qquad \text{and} \\
\cS \colon \cR \tensor{A} \cH \to \cR \tensor{A} \cH, \qquad x \tensor{A} h \mapsto \sum x_0 \tensor{A} x_1\cS(h)
\end{gather*}
This is called the \emph{left translation Hopf algebroid} of $(A, \cH)$ along $r$. Of course, if $\sM{\cH}$ is flat over $A$, then $\cR\tensor{A}\cH$ is flat over $\cR$ (see \cite[Lemma 3.4]{Kaoutit/Kowalzig:17} for more properties).

Now, dually to the case of groupoids, we introduce the \emph{adjoint coaction}.

\begin{proposition}\label{prop:barcoaction}
The pair $(\iso{\cH}, \eta)$ is a right $\cH$-comodule algebra with coaction defined by:
\begin{equation}\label{Eq:coaction}
\rcaction{\iso{\cH}} \colon  \iso{\cH} \to \iso{\cH} \tensor{A} \cH, \qquad \bara{h} \mapsto \sum \bara{h_2} \tensor{A} \cS(h_1)h_3 .
\end{equation}
In particular, $(\iso{\cH}, \iso{\cH}\tensor{A}\cH)$ is a commutative Hopf algebroid (it is the left translation Hopf algebroid along $\eta$).
\end{proposition}

\begin{proof}
The fact that $\rcaction{\iso{\cH}}$ is well-defined can be deduced by using the following map:
\[
\Mt{\cH} \tensor{A} \sMt{\cH} \to {\iso{\cH}}_\eta \tensor{A} \sMs{\cH}, \qquad u\tensor{A}v \mapsto \bara{v}\tensor{A} \cS(u),
\] 
which is well-defined and right $A$-linear. The map $\rcaction{\iso{\cH}}$ is clearly right $A$-linear, that is, $\rcaction{\iso{\cH}}\Big( \bara{h}\eta(a)\Big) = \rcaction{\iso{\cH}}(\bara{h})t(a)$ for every $a \in A$ and $\bara{h} \in \iso{\cH}$. Aside from the coaction being well-defined, the rest of the proof showing the coassosiativity and counital property of $\delta$ and the comodule algebra structure on $\iso{\cH}$ follow exactly as in the case of the right adjoint coaction on ordinary Hopf algebras, by simply using Sweedler's notation. Therefore, $(\iso{\cH}, \rcaction{\iso{\cH}})$ is a right $\cH$-comodule. Since $\rcaction{\iso{\cH}}$ is manifestly multiplicative and unital, $(\iso{\cH}, \eta)$ is a right $\cH$-comodule algebra. 
\end{proof}

Let $I$ be an ideal of $\cH$ which contains $\langle s-t \rangle$. We denote by $\iso{I}$ its image in $\iso{\cH}$: $\iso{I} \coloneqq I/\langle s-t\rangle$.

\begin{definition}\label{def:normal}
Let $(A, \cH)$ be a commutative Hopf algebroid. A Hopf ideal $I$ of $\cH$ is said to be \emph{normal} if it satisfies the following conditions:
\begin{enumerate}[label=(NI\arabic*),ref=(NI\arabic*)]
\item\label{item:NI1} $\langle s-t \rangle \subseteq I$;
\item\label{item:NI2} for every $\bara{x} \in \iso{I}$, we have $\rcaction{\iso{\cH}}\big(\bara{x}\big) = \sum\bara{x_2} \tensor{A} \cS(x_1)x_3 \in \img\left(\iso{I} \tensor{A} \cH\right)$.
\end{enumerate}
\end{definition}

\begin{remark}
If $\sM{\cH}$ is $A$-flat, then \ref{item:NI2} is equivalent to requiring that the image of $\iso{I}$ is a right $\cH$-subcomodule of $\iso{\cH}$ with respect to the coaction \eqref{Eq:coaction}. Namely, that there exists $\rcaction{\iso{I}}\colon \iso{I} \to \iso{I} \tensor{A} \cH$ such that the following diagram commutes
\[
\begin{gathered}
\xymatrix{
\iso{I} \ar[r]^-{j} \ar[d]_-{\rcaction{\iso{I}}} & \iso{\cH} \ar[d]^-{\rcaction{\iso{\cH}}} \\
\iso{I} \tensor{A} \cH \ar[r]_-{j \tensor{A} \cH} & \iso{\cH} \tensor{A} \cH.
} 
\end{gathered}
\qedhere
\]
\end{remark}

For commutative Hopf algebroids with source equal to target, that is, for Hopf algebras over a commutative algebra, one recovers the classical definition. Notice further that in this latter case the associated isotropy Hopf algebra coincides with the Hopf algebroid itself.   

\begin{lemma}\label{lema:bnormal}
Let $I$ be a normal Hopf ideal in $(A,\cH)$. Then $\iso{I}$ is a normal Hopf ideal in $(A, \iso{\cH})$.
\end{lemma}

\begin{proof}
Straightforward. 
\end{proof}

The following proposition is the Hopf algebroid analogue of Lemma \ref{lema:normalsugpd}.

\begin{proposition}\label{prop:inducedcoaction}
Let $(A,\cH)$ be a commutative Hopf algebroid and let $I \subseteq \cH$ be a Hopf ideal. Then $I$ is normal if and only if $\langle s-t \rangle \subseteq I$ and $\cH/I \cong \iso{\cH}/\iso{I}$ is a quotient $\cH$-comodule of $\iso{\cH}$ with respect to the coaction \eqref{Eq:coaction}.
\end{proposition}

\begin{proof}
The statement follows from the fact that $I$ is normal if and only if the left-hand side square in the following diagram with exact rows commutes
\[
\begin{gathered}
\xymatrix{
0 \ar[r] & \iso{I} \ar[r]^-{j} \ar@{.>}[d]^-{\rcaction{\iso{I}}} & \iso{\cH} \ar[r]^-{\iso{\pi}} \ar[d]^-{\rcaction{\iso{\cH}}} & {\displaystyle \frac{\iso{\cH}}{\iso{I}}} \ar[r]  \ar@{-->}[d]^-{\tilde{\delta}_{\iso{\cH}}} & 0 \\
0 \ar[r] & \img\big(\iso{I} \tensor{A} \cH\big) \ar[r]_-{\subseteq} & \iso{\cH} \tensor{A} \cH \ar[r]_-{\iso{\pi} \tensor{A} \cH} & {\displaystyle \frac{\iso{\cH}}{\iso{I}} \tensor{A} \cH} \ar[r] & 0,
}
\end{gathered}
\]
if and only if the right-hand side square commutes, because $\img\left(\iso{I} \tensor{A} \cH\right) = \ker\left(\iso{\pi} \tensor{A} \cH\right)$.
\end{proof}

Let $\big(\Gg_A,\Gg_{\cH/I}\big)$ be an affine subgroupoid of $\big(\Gg_A,\Gg_\cH\big)$ (i.e., $I \subseteq \cH$ is a Hopf ideal). We say that $\big(\Gg_A,\Gg_{\cH/I}\big)$ is a \emph{normal affine subgroupoid} of $\big(\Gg_A,\Gg_\cH\big)$ if and only if $\big(\Gg_A(R),\Gg_{\cH/I}(R)\big)$ is a normal subgroupoid of $\big(\Gg_A(R),\Gg_\cH(R)\big)$ in the sense of \S\ref{sssec:groupoidactions}, for all $R$ in $\calg_\K$. The following proposition shows that normal Hopf ideals are in bijective correspondence with normal affine subgroupoids. 

\begin{proposition}\label{prop:cociente}
Let $(A, \cH)$ be a commutative Hopf algebroid. There is a one-to-one correspondence between normal Hopf ideals of $(A,\cH)$ and normal affine subgroupoids of $\big(\Gg_A, \Gg_\cH\big)$.
\end{proposition}

\begin{proof}
Let $I$ be a normal Hopf ideal of $\cH$. Then, in view of \ref{item:NI1} and Lemma \ref{lema:bnormal}, $\cH/I \cong \iso{\cH}/\iso{I}$ is a commutative Hopf $A$-algebra with a surjective morphism of commutative Hopf algebroids $(\id_A,\tilde{\pi}) \colon (A,\cH) \to (A, \iso{\cH}/\iso{I})$, where $\tilde{\pi}\colon \cH \to \iso{\cH}/\iso{I}$ is induced by $\pi\colon \cH \to \cH/I$, up to the isomorphism above. We have then the affine subgroupoid $\Big(\Gg_A,\Gg_{\iso{\cH}/\iso{I}}\Big)$ of $\big(\Gg_A,\Gg_\cH\big)$. For $R$ in $\calg_\K$, consider $(\id,\tilde{\pi}^*) \colon \Big(\Gg_A(R),\Gg_{\iso{\cH}/\iso{I}}(R)\Big) \hookrightarrow \big(\Gg_A(R),\Gg_\cH(R)\big)$. Condition \ref{item:norm1} is always satisfied in our treatment. Condition \ref{item:norm2} is satisfied because for every $\varphi \in \calg_\K\big(\iso{\cH}/\iso{I},R\big)$, we have
\[\sigma(\varphi) = \varphi \circ (\tilde{\pi} \circ s) \stackrel{\ref{item:NI1}}{=} \varphi \circ (\tilde{\pi} \circ t) = \tau(\varphi).\]
Concerning condition \ref{item:norm3}, for every $\psi \in \calg_\K(\cH,R)$ and for every $\varphi \in \calg_\K\big(\iso{\cH}/\iso{I},R\big)$ such that $\varphi\circ(\tilde{\pi}\circ s) = \psi \circ s = \varphi \circ (\tilde{\pi} \circ t)$ (i.e., source and target of $\varphi$ coincide with the source of $\psi$), 
set $\varphi' \coloneqq \psi \bullet \tilde{\pi}^*(\varphi) \bullet \psi^{-1} \in \calg_\K(\cH,R)$. Notice that 
\begin{multline*} 
\varphi'\big(s(a)-t(a)\big) = \big(\psi \bullet \tilde{\pi}^*(\varphi) \bullet \psi^{-1}\big)\big(s(a)-t(a)\big) \stackrel{\eqref{eq:comp}}{=} \psi^{-1}\big(s(a)\big)\varphi\Big(\bara{1}\Big)\psi(1) - \psi^{-1}(1)\varphi\Big(\bara{1}\Big)\psi\big(t(a)\big)
\\=\psi\big(\cS s(a)\big)\varphi\Big(\bara{1}\Big)\psi(1) - \psi(\cS (1))\varphi\Big(\bara{1}\Big)\psi\big(t(a)\big)=\psi\big(t(a)\big)\varphi\Big(\bara{1}\Big)\psi(1) - \psi(1)\varphi\Big(\bara{1}\Big)\psi\big(t(a)\big) = 0
\end{multline*}
for all $a \in A$, whence there exists a unique $\tilde{\varphi}'\colon \iso{\cH} \to R$ such that $\tilde{\varphi}' \circ \iso{\varpi} = \varphi'$. Now, for every $\bara{x} = \iso{\varpi}(x) \in \iso{I}$ we have that
\begin{multline*}
\tilde{\varphi}'\big(\bara{x}\big) = \big(\psi \bullet \tilde{\pi}^*(\varphi) \bullet \psi^{-1}\big)(x) \stackrel{\eqref{eq:comp}}{=} \sum\psi\big(\cS(x_1)\big)\varphi\big(\bara{x_2} + \iso{I}\big)\psi(x_3) \\
 = \sum\varphi\big(\bara{x_2} + \iso{I}\big)\psi\big(\cS(x_1)x_3\big) = \big(\varphi \tensor{A} \psi\big)\big(\iso{\pi} \tensor{A} \cH\big)\left(\sum\bara{x_2} \tensor{A} \cS(x_1)x_3\right) \stackrel{\ref{item:NI2}}{=} 0,
\end{multline*}
whence there exists a unique $\phi \colon \iso{\cH}/\iso{I} \to R$ such that $\phi \circ \iso{\pi} = \tilde{\varphi}'$. Summing up, $\psi \bullet \tilde{\pi}^*(\varphi) \bullet \psi^{-1} = \tilde{\pi}^*(\phi)$ for some $\phi \in \Gg_{\iso{\cH}/\iso{I}}(R)$ and so \ref{item:norm3} is satisfied.

Conversely, take a normal affine subgroupoid $\big(\Gg_A,\Gg_{\cH/I}\big)$ of $\big(\Gg_A,\Gg_\cH\big)$ induced by $\pi\colon \cH \to \cH/I$. By definition, $I$ is a Hopf ideal of $\cH$. Moreover, for any commutative $\K$-algebra $R$ and for any $\varphi \in \calg_\K\big(\cH/I,R\big)$ we have that $\varphi \circ (\pi \circ s)  = \varphi \circ (\pi \circ t)$, because in $\big(\Gg_A(R),\Gg_{\cH/I}(R)\big)$ the only arrows that we have are loops. By taking $R=\cH/I$ and $\varphi = \id_{\cH/I}$, we get $\pi \circ t  = \pi \circ s $, which shows that $\langle s - t \rangle \subseteq I$ and hence \ref{item:NI1} is satisfied and $\cH/I \cong \iso{\cH}/\iso{I}$. We still need  to check condition \ref{item:NI2}. For this, recall from Lemma \ref{lema:normalsugpd} that $\big(\Gg_A,\Gg_{\cH/I}\big)$ is a normal affine subgroupoid of $\big(\Gg_A,\Gg_\cH\big)$ if and only if 
\[
\xymatrix{  
\Gg_\cH(R) \operatorname{\due \times {s^*} {(\pi \circ t)^*}} \Gg_{\frac{\cH}{I}}(R) \ar[r] \ar[d] & \Gg_{\frac{\cH}{I}}(R) \ar[d] \\
\Gg_\cH(R) \operatorname{\due \times {s^*} {(\iso{\varpi}\circ t)^*}} \Gg_{\iso{\cH}}(R) \ar[r] & \Gg_{\iso{\cH}}(R)
}
\] 
commutes for every $R$ in $\calg_\K$, if and only if there exists a right $\cH$-comodule algebra structure $\rcaction{\cH/I}$ on $\cH/I$ such that
\[
\xymatrix{  
\iso{\cH}  \ar[rr]^-{\rcaction{\iso{\cH}}} \ar[d]_-{\iso{\tilde{\pi}}} & & \iso{\cH}\tensor{A}\cH \ar[d]^-{\iso{\tilde{\pi}} \tensor{A} \cH} \\ 
{\displaystyle\frac{\cH}{I}} \ar^-{\rcaction{\cH/I}}[rr]   & & {\displaystyle\frac{\cH}{I}\tensor{A}\cH}
}
\] 
commutes (where $\iso{\tilde{\pi}}$ is the composition $\iso{\cH} \xrightarrow{\iso{\pi}} \iso{\cH}/\iso{I} \cong \cH/I$ with kernel $\iso{I}$), if and only if $\rcaction{\iso{\cH}}\big(\iso{I}\big) \subseteq \ker\big(\iso{\tilde{\pi}} \tensor{A} \cH\big) = \img\big(\iso{I} \tensor{A} \cH\big)$.
\end{proof}

\begin{remark}
The quotient $\iso{\cH}$ admits also a left $\cH$-comodule structure, which is related to the right comodule structure above via the isomorphism between the category of left $\cH$-comodules and the category of right $\cH$-comodules provided by the antipode: if $\left(N,\delta\colon n \mapsto \sum n_0 \tensor{A} n_1 \right)$ is a right $\cH$-comodule, then $N$ with $\partial \colon N \to \cH \tensor{A} N, n \mapsto \sum \cS(n_{1}) \tensor{A} n_0$, is a left $\cH$-comodule (recall that right $\cH$-comodules are, in fact, symmetric $A$-bimodules with the right-left being defined from \eqref{eq:addAction}). The left $\cH$-coaction on $\iso{\cH}$ is then given by $\partial_{\iso{\cH}} \colon  \iso{\cH} \to \cH \tensor{A} \iso{\cH}$ sending $\bara{x}$ to $\sum x_1\cS(x_3) \tensor{A} \bara{x_2}$. This corresponds to the right action by conjugation given at the groupoid level by $(f,g) \mapsto g^{-1} \bullet f \bullet g$. If $\sM{\cH}$ is flat over $A$, then also the quotient $\iso{I}$ of a normal Hopf ideal $I$ admits both a left and a right $\cH$-comodule structures as above.
\end{remark}


\subsubsection{The space of invariants}

Let $(A, \cH)$ be a commutative Hopf algebroid and let $I \subseteq \cH$ be a normal Hopf ideal. 
Denote by $\tilde{\pi} \colon  \cH \to \iso{\cH}/\iso{I}$ the composition of the canonical projection $\pi\colon \cH \to \cH/I$ with the isomorphism $\cH/I \cong \iso{\cH}/\iso{I}$. Since this leads to a morphism of Hopf algebroids, we have an induced $\iso{\cH}/\iso{I}$-comodule structure on $\cH$ as in \S\ref{sec:N1}, given by
\begin{equation}\label{Eq:hbar}
\rcaction{\cH} \colon  \cH \to \cH\tensor{A} {\displaystyle \frac{\iso{\cH}}{\iso{I}}}, \qquad  h \mapsto \sum h_1\tensor{A} \left(\bara{h_2} + \iso{I}\right).
\end{equation} 
In this way $(\cH, s)$ becomes a right $\iso{\cH}/\iso{I}$-comodule algebra. Let us denote its coinvariant subalgebra by
\[
\rcoinv{\cH}{I} \coloneqq \LR{  h \in \cH ~\Big\vert~ \rcaction{\cH}(h) = h\tensor{A}\left(\bara{1}+\iso{I}\right) } 
\]
for the sake of simplicity.

\begin{remark}
Notice that the coaction \eqref{Eq:hbar} corresponds to the action described in \eqref{Eq:Naction}. Namely, denote by $\tilde{\eta}$ the composition of $A \xrightarrow{\eta}  \iso{\cH}/\iso{I} \cong \cH/I$ and, for a given ring $R$, consider the normal subgroupoid $\Gg_{\cH/I}(R)$ of $\Gg_\cH(R)$. In this framework, the action \eqref{Eq:Naction} is given by 
\[
\Gg_{\frac{\cH}{I}}(R) \operatorname{\due\times{\tilde{\eta}^*}{t^*}} \Gg_\cH(R) \to \Gg_\cH(R), \qquad (g,f) \mapsto (g\circ \pi) \bullet f.
\]
By using the natural isomorphism
\[\Gg_{\frac{\cH}{I}}\left(R\right) \operatorname{\due\times{\tilde{\eta}^*}{t^*}} \Gg_\cH\left(R\right) \cong \calg_\K\left(\cH\tensor{A} {\displaystyle \frac{\cH}{I}},R\right) \cong \calg_\K\left(\cH \tensor{A} {\displaystyle \frac{\iso{\cH}}{\iso{I}}}, R\right),\] 
the action \eqref{Eq:Naction} corresponds exactly to the coaction \eqref{Eq:hbar}.
\end{remark}

Analogously, the left $\iso{\cH}/\iso{I}$-coaction on $\cH$, which corresponds to the right action of $\Gg_{\cH/I}(R)$ on $\Gg_\cH(R)$ from the end of \S\ref{ssec:orbits}, is given by
\[
\partial_{\cH} \colon  \cH \to {\displaystyle \frac{\iso{\cH}}{\iso{I}}} \tensor{A} \cH, \qquad  h \mapsto  \sum \left(\bara{h_1} + \iso{I}\right) \tensor{A} h_2.
\]
which provides a structure of left $\iso{\cH}/\iso{I}$-comodule algebra on $(\cH, t)$. Its coinvariant subalgebra is given by 
\begin{equation}\label{Eq:lbaracoin}
\lcoinv{\cH}{I} \coloneqq  \LR{  h \in \cH ~\Big\vert~ \partial_{\cH}(h) = \left(\bara{1}+\iso{I}\right) \tensor{A} h }. 
\end{equation}
The key properties of these subalgebras are collected in the forthcoming lemmata.

\begin{lemma}\label{lema:proHJ} 
Let $I$ be a normal Hopf ideal of $(A, \cH)$. Then, the subalgebras $s (A)$ and $t (A)$ are contained in both $\rcoinv{\cH}{I}$ and  $\lcoinv{\cH}{I}$. Furthermore, the antipode $\cS$ establishes an isomorphism of $A$-bimodules $\cS \colon  \rcoinv{\cH}{I}\to \lcoinv{\cH}{I}$.  
\end{lemma}

\begin{proof}
We only show that $t(A) \subseteq \rcoinv{\cH}{I}$, since the other property is analogous. For every $a \in A$ we have
\[\rcaction{\cH}\big(t(a)\big) = 1\tensor{A} \left(\bara{t(a)} + \iso{I}\right) = 1\tensor{A} \left(\bara{s(a)} + \iso{I}\right) = t(a) \tensor{A} \left(\bara{1} + \iso{I}\right)\]
by \ref{item:NI1}, and hence $t(a) \in \rcoinv{\cH}{I}$. 
Concerning the second claim, if $h \in \rcoinv{\cH}{I}$ then 
\begin{align*}
\partial_\cH\big(\cS(h)\big) & = \sum \left(\bara{\cS(h)_1} + \iso{I}\right) \tensor{A} \cS(h)_2 = \sum \left(\bara{\cS(h_2)} + \iso{I}\right) \tensor{A} \cS(h_1) \\
& \stackrel{\ref{item:HI3}}{=} \sum \bara{\cS}\left(\bara{h_2} + \iso{I}\right) \tensor{A} \cS(h_1) = \bara{\cS}\left(\bara{1} + \iso{I}\right) \tensor{A} \cS(h) \\
& = \left(\bara{1} + \iso{I} \right)\tensor{A} \cS(h),
\end{align*}
whence $\cS(h) \in \lcoinv{\cH}{I}$. Similarly, one proves that if $h \in \lcoinv{\cH}{I}$, then $\cS(h) \in \rcoinv{\cH}{I}$.
\end{proof}

\begin{lemma}\label{lem:IHisHI}
Let $(A, \cH)$ be a commutative Hopf algebroid and let $I$ a normal Hopf ideal in $(A, \cH)$.
Then $\rcoinv{\cH}{I} = \lcoinv{\cH}{I}$ as subalgebras of $\cH$.
\end{lemma}

\begin{proof}
We find inspiration from \cite[Lemma 4.4]{Takeuchi:1972}. First of all, consider the $A$-bilinear morphism
\[\psi \colon \cH \tensor{} {\displaystyle \frac{\iso{\cH}}{\iso{I}}} \to {\displaystyle \frac{\iso{\cH}}{\iso{I}}} \tensor{A} \sM{\cH}, \qquad x \tensor{} \big(\bara{y} + \iso{I}\big) \mapsto \sum \big(\bara{y_2} + \iso{I} \big)\tensor{A} x\cS\big(y_1\big)y_3,\]
where the $A$-bilinearity is expressed by
\begin{align*}
\psi\Big(a \cdot \Big(x \tensor{} \big(\bara{y} + \iso{I}\big)\Big)\cdot b\Big) & = \psi\Big(s(a)x \tensor{} \big(\bara{y} + \iso{I}\big)\eta(b)\Big) \\
& = \sum \eta(a)\big(\bara{y_2} + \iso{I} \big)\tensor{A} x\cS\big(y_1\big)y_3t(b) \\
& = a\cdot\bigg( \sum \big(\bara{y_2} + \iso{I} \big)\tensor{A} x\cS\big(y_1\big)y_3\bigg)\cdot b \\
& = a \cdot \psi\Big(x \tensor{} \big(\bara{y} + \iso{I}\big)\Big)\cdot b.
\end{align*}
for all $a,b \in A$, $x,y \in \cH$. It satisfies
\begin{align*}
\psi\Big(xt(a) \tensor{} \big(\bara{y} + \iso{I}\big)\Big) & = \sum \big(\bara{y_2} + \iso{I} \big)\tensor{A} xt(a)\cS\big(y_1\big)y_3 = \sum \big(\bara{y_2} + \iso{I} \big)\tensor{A} x\cS\big(s(a)y_1\big)y_3 \\
& = \psi\Big(x \tensor{} \big(\bara{s(a)y} + \iso{I}\big)\Big) = \psi\Big(x \tensor{} \eta(a)\big(\bara{y} + \iso{I}\big)\Big)
\end{align*}
for all $a \in A$, $x,y \in \cH$. Therefore, $\psi$ factors uniquely through the tensor product over $A$, inducing an $A$-bilinear morphism
\[\widetilde{\psi} \colon \sMt{\cH} \tensor{A} {\displaystyle \frac{\iso{\cH}}{\iso{I}}} \to {\displaystyle \frac{\iso{\cH}}{\iso{I}}} \tensor{A} \sMt{\cH}, \qquad x \tensor{A} \big(\bara{y} + \iso{I}\big) \mapsto \sum \big(\bara{y_2} + \iso{I} \big)\tensor{A} x\cS\big(y_1\big)y_3.\]
Take an element $h \in \rcoinv{\cH}{I}$, so that $\sum h_1 \tensor{A} \left( \bara{h_2} + \iso{I} \right) = h\tensor{A} \left(\bara{1} + \iso{I} \right)$ in $\cH\tensor{A}\iso{\cH}/\iso{I}$. By applying $\widetilde{\psi}$ to both sides of the latter identity, we get
\[
\sum \left( \bara{h_3} + \iso{I} \right) \tensor{A} h_1\cS\big(h_2\big)h_4 = \left(\bara{1} + \iso{I} \right) \tensor{A} h \quad \text{in} \quad {\displaystyle \frac{\iso{\cH}}{\iso{I}}} \tensor{A} \cH,
\]
that is to say, $\sum \left( \bara{h_1} + \iso{I} \right) \tensor{A} h_2 = \left(\bara{1} + \iso{I} \right) \tensor{A} h$ and hence $h \in \lcoinv{\cH}{I}$.
The other inclusion is proved similarly by using the morphism
\[{\displaystyle \frac{\iso{\cH}}{\iso{I}}} \tensor{A} \sMt{\cH} \to \sMt{\cH} \tensor{A} {\displaystyle \frac{\iso{\cH}}{\iso{I}}}, \qquad \big(\bara{x} + \iso{I}\big) \tensor{A} y \mapsto \sum x_1\cS\big(x_3\big)y \tensor{A} \big(\bara{x_2} + \iso{I} \big). \qedhere\]
\end{proof}

As we have seen in Proposition \ref{prop:cociente}, for every commutative $\K$-algebra $R$ we have that $\big(\Gg_A(R),\Gg_{\cH/I}(R)\big)$ is a normal subgroupoid of $\big(\Gg_A(R),\Gg_{\cH}(R)\big)$ and hence, by Lemma \ref{lema:quotient}, we can consider the extension $\Gg_{\cH/I}(R) \hookrightarrow \Gg_\cH(R) \twoheadrightarrow \Gg_\cH(R)/\Gg_{\cH/I}(R)$. As a matter of notation, set $\phi\colon \cH^I \to \cH$ for the canonical inclusion and $\Phi_R \colon \calg_\K(\cH, R) \to \calg_\K(\cH^I, R)$ for its dual $\phi^*$.

\begin{lemma}\label{lema:JHJ}
Let $(A, \cH)$ be a commutative Hopf algebroid and let $I\subseteq \cH$ be a normal Hopf ideal.
Then there is a canonical morphism
\begin{equation}\label{eq:PsiR}
\Psi_R \colon {\displaystyle\frac{\Gg_\cH(R)}{\Gg_{\frac{\cH}{I}}(R)}} \to \calg_\K(\rcoinv{\cH}{I}, R), \qquad \Gg_{\frac{\cH}{I}}(R)\bullet g \mapsto \Phi_R(g),
\end{equation} 
which is natural in $R \in \calg_\K$.
\end{lemma}

\begin{proof}
Consider $\Psi_R$ as in the statement and observe that if $\Gg_{\cH/I}(R)\bullet g = \Gg_{\cH/I}(R)\bullet g'$, then there exists $f \in \Gg_{\cH/I}(R)$ such that $(f \circ \pi) \bullet g = g'$. Therefore, for every $u \in \cH^I$
\[\Phi_R(g')(u) = \Phi_R \big((f \circ \pi) \bullet g\big)(u) = \sum g\big(\phi(u)_1\big)f\Big(\bara{\phi(u)_2} + \iso{I}\Big) = g\big(\phi(u)\big) = \Phi_R(g)(u),\]
since $\phi$ is just the inclusion. Thus $\Psi_R$ is well-defined and it is clearly natural.
\end{proof}


\subsection{The correspondence for commutative Hopf algebroids}

We fix a commutative Hopf algebroid $(A, \cH)$ which we assume to be flat as left $A$-module (that is, the extension $s \colon  A \to \cH$ is flat). In this case $t \colon A \to \cH$ is also flat, thanks to the antipode, and both extension $s $ and $t $ are faithfully flat (because both are split monomorphisms of modules over the base algebra). 
Also, since $A$ and $\cH$ are commutative algebras, then $\Mt{\cH}$ being $A$-flat is equivalent to $\tM{\cH}$ being $A$-flat. 

\begin{proposition}\label{prop:HopfNonGalois}
If $(A,\cH)$ is a commutative Hopf algebroid and $\sM{\cH}$ is $A$-flat, then we have a well-defined inclusion-preserving bijective correspondence
\[
\begin{gathered}
\xymatrix @R=0pt{
{\left\{ \begin{array}{c}  \text{bi-ideals } I \text{ in } \cH \text{ such that} \\ \cH \text{ is pure over } \coinv{\cH}{\frac{\cH}{I}} \end{array} \right\}} \ar@{<->}[r] & {\left\{ \begin{array}{c} \text{right } \cH\text{-comodule } A\text{-subalgebras} \\ B \subseteq \cH \text{ via } t  \text{ such that } \cH \text{ is pure over } B  \end{array} \right\}} \\
I \ar@{|->}[r] & \coinv{\cH}{\frac{\cH}{I}} \\
\cH B^+ & B \ar@{|->}[l]
}
\end{gathered}
\]
\end{proposition}

\begin{proof}
Since all the algebras involved are commutative, it follows from \cite[Corollary 5.4]{JanelidzeTholen} that a $\K$-algebra morphism $B \to \cH$ is pure as a morphism of $B$-modules if and only if the extension-of-scalars functor $\cH \tensor{B} -$ is comonadic. Thus, it is enough to show that the additional condition in Theorem \ref{mainthm:Galois} which involves the translation map $\gamma$, is always satisfied in this setting.
Recall from \eqref{eq:gammacomm} that $\gamma$ is explicit given by $\big(\cH \tensor{A} \cS\big) \circ \Delta$. 
If $B\subseteq \cH$ is a right $\cH$-coideal $A$-subalgebra via $t$, then the left-hand side square of the following diagram is commutative
\[
\xymatrix{
\Mt{B} \ar@{ >->}[d]_-{\iota} \ar[r]^-{\partial} & \Mt{B} \tensor{A} \sMt{\cH} \ar@{ >->}[d]^-{\iota \tensor{A} \cH} \ar[r]^-{B \tensor{A} \cS} & \Mt{B} \tensor{A} \tMs{\cH} \ar@{ >->}[d]^-{\iota \tensor{A} \cH} \\
\Mt{\cH} \ar[r]^-{\Delta} \ar@/_4ex/[rr]_-{\gamma} & \Mt{\cH} \tensor{A} \sMt{\cH} \ar[r]^-{\cH \tensor{A} \cS} & \Mt{\cH} \tensor{A} \tMs{\cH}.
}
\]
Since the right-hand side square is clearly commutative, we have that the condition $\gamma(B) \subseteq B \tensor{A} \cH$ is satisfied.
Therefore, the conclusion follows from Theorem \ref{mainthm:Galois}.
\end{proof}

\begin{proposition}\label{prop:nHideal}
Let $(A,\cH)$ be a commutative Hopf algebroid such that $\sM{\cH}$ is $A$-flat. Suppose that $\cK\subseteq\cH$ is a sub-Hopf algebroid, over the same base $A$, such that $\cH_\cK$ is pure. Then $\cH\cK^+$ is a normal Hopf ideal in $\cH$ such that $\cH$ is pure over $\coinv{\cH}{\frac{\cH}{\cH\cK^+}}$.
\end{proposition}

\begin{proof}
Since it is clear that $\cK$ is a right $\cH$-comodule $A$-subalgebra of $\cH$ via $t$, we already know from Proposition \ref{prop:HopfNonGalois} that $\cH\cK^+$ is a (two-sided) ideal of $\cH$ which satisfies \ref{item:HI1} and \ref{item:HI2}, and that $\cH$ is pure over $\coinv{\cH}{\frac{\cH}{\cH\cK^+}}$. Concerning \ref{item:HI3}, it is easily checked that
\[\cS\big(\cH\cK^+\big) \subseteq \cS\big(\cH\big)\cS\big(\cK^+\big) \subseteq \cH\cK^+\]
because $\cK$ is a sub-Hopf algebroid (whence $\cS(\cK)\subseteq \cK$). Therefore, $\cH\cK^+$ is a Hopf ideal. In addition, since for every $a \in A$ we have that $s(a) - t(a) \in \cK^+$, we also have that $\langle s-t \rangle \subseteq \cH\cK^+$, which is \ref{item:NI1}. To check that \ref{item:NI2} also holds, let us begin by observing that $A \cong \cK/\cK^+$ has a $(\cK,A)$-bimodule structure via the regular right $A$-action $x\cdot a = \varepsilon(x)a$ for all $a \in A$, $x \in \cK$. Therefore,
\[\Mt{\frac{\cH}{\cH\cK^+}} \tensor{A} \sM{\cH} \cong \cH_\cK \tensor{\cK} {{}_\varepsilon A_A} \tensor{A} \sM{\cH} \cong \cH_\cK \tensor{\cK} {{}_{s\varepsilon}\cH}\]
via $\big(x + \cH\cK^+\big) \tensor{A} y \mapsto x \tensor{\cK} y$. Denote this latter isomorphism by $\varphi \colon \cH/\cH\cK^+ \tensor{A} \cH \to \cH \tensor{\cK} {{}_{s\varepsilon}\cH}$. In view of the flatness of $\cH$ over $A$, an element $\sum_i \bara{x_i} \tensor{A} y_i$ in $\iso{\cH} \tensor{A} \cH$ belongs to $\iso{(\cH\cK^+)} \tensor{A} \cH$ if and only if it belongs to the kernel of the canonical projection $\iso{\tilde{\pi}} \tensor{A} \cH \colon \iso{\cH} \tensor{A} \cH \to \cH/\cH\cK^+ \tensor{A} \cH$. If $\sum_i \bara{h_ix_i} \in \iso{(\cH\cK^+)}$, then
\begin{align*}
\varphi & \left(\left(\iso{\tilde{\pi}} \tensor{A} \cH\right)\left(\sum_i \bara{\big(h_ix_i\big)_2} \tensor{A} \cS\Big(\big(h_ix_i\big)_1\Big)\big(h_ix_i\big)_3\right)\right) \\
& = \varphi\left(\sum_i \Big(\big(h_i\big)_2\big(x_i\big)_2 + \cH\cK^+\Big) \tensor{A} \cS\Big(\big(h_i\big)_1\Big)\cS\Big(\big(x_i\big)_1\Big)\big(h_i\big)_3\big(x_i\big)_3\right) \\
& = \sum_i \big(h_i\big)_2\big(x_i\big)_2 \tensor{\cK} \cS\Big(\big(h_i\big)_1\Big)\cS\Big(\big(x_i\big)_1\Big)\big(h_i\big)_3\big(x_i\big)_3 \\
& = \sum_i \big(h_i\big)_2 \tensor{\cK} s\varepsilon\Big(\big(x_i\big)_2\Big)\cS\Big(\big(x_i\big)_1\Big)\big(x_i\big)_3\cS\Big(\big(h_i\big)_1\Big)\big(h_i\big)_3 \\
& = \sum_i \big(h_i\big)_2 \tensor{\cK} t\varepsilon\big(x_i\big)\cS\Big(\big(h_i\big)_1\Big)\big(h_i\big)_3 = 0,
\end{align*}
which entails that $\sum_i \bara{\big(h_ix_i\big)_2} \tensor{A} \cS\Big(\big(h_ix_i\big)_1\Big)\big(h_ix_i\big)_3 \in \iso{(\cH\cK^+)} \tensor{A} \cH$, as prescribed by \ref{item:NI2}.
\end{proof}

\begin{proposition}\label{prop:subHalgd}
Let $I$ be a normal Hopf ideal of a commutative Hopf algebroid $(A,\cH)$ such that $\sM{\cH}$ is $A$-flat and $\cH$ is pure over $\coinv{\cH}{\frac{\cH}{I}}$. Then the pair $\Big(A, \coinv{\cH}{\frac{\cH}{I}}\Big)$ is a sub-Hopf algebroid of $(A,\cH)$.
\end{proposition}

\begin{proof}
Set $\cK \coloneqq \coinv{\cH}{\frac{\cH}{I}}$, which in \eqref{Eq:lbaracoin} we denoted by $\lcoinv{\cH}{I}$, for the sake of brevity. By Proposition \ref{prop:HopfNonGalois} we already know that $\cK$ is a right $\cH$-comodule subalgebra of $\cH$ via $t$. By Lemma \ref{lema:proHJ} and Lemma \ref{lem:IHisHI}, we also know that $\cK = \rcoinv{\cH}{I}$, that $s$ takes values in $\cK$, and that $\cS\big(\cK\big) \subseteq \cK$. Therefore, by adapting Proposition \ref{prop:OneIsGone}, $\cK$ is also a left $\cH$-comodule subalgebra of $\cH$ via $s$.

Since $\cH$ is pure over $\cK$ and flat over $A$, $\cK$ is flat over $A$ (see Corollary \ref{cor:frompuretoflat}) and so 
\[\cK \tensor{A} \cK \subseteq \big(\cK \tensor{A} \cH\big) \cap \big(\cH \tensor{A} \cK\big) \subseteq \cH \tensor{A} \cH.\]
Hence, in view of the foregoing observations, we only need to check that $\Delta(\cK)$, which we know is contained in $\big(\cK \tensor{A} \cH\big) \cap \big(\cH \tensor{A} \cK\big)$, is contained in $\cK \tensor{A} \cK$.  
To this aim, denote by $\iota\colon \cK \to \cH$ the canonical inclusion, which we know is pure as morphism of $A$-modules by Corollary \ref{cor:pure}. 
Now, in view of the fact that both $\cH$ and $\cK$ are flat over $A$, the columns of the following diagram are monomorphisms and the rows are equalizers
\[
\xymatrix{
\cK \tensor{A} \cK \ar[r]^-{\cK \tensor{A} \iota} \ar@{ >->}[d]_-{\iota \tensor{A} \cK} & \cK \tensor{A} \cH \ar@<-0.5ex>[rr]_-{\cK \tensor{A} (\tilde{\pi} \tensor{A} \cH)\Delta} \ar@<+0.5ex>[rr]^-{\cK \tensor{A} \tilde{\pi}t \tensor{A} \cH} \ar@{ >->}[d]_-{\iota \tensor{A} \cH} & & {\displaystyle \cK \tensor{A} \frac{\iso{\cH}}{\iso{I}} \tensor{A} \cH } \ar@{ >->}[d]^-{\iota \tensor{A} \frac{\iso{\cH}}{\iso{I}} \tensor{A} \cH} \\
\cH \tensor{A} \cK \ar[r]^-{\cH \tensor{A} \iota} & \cH \tensor{A} \cH \ar@<-0.5ex>[rr]_-{\cH \tensor{A} (\tilde{\pi} \tensor{A} \cH)\Delta} \ar@<+0.5ex>[rr]^-{\cH \tensor{A} \tilde{\pi}t \tensor{A} \cH} & & {\displaystyle \cH \tensor{A} \frac{\iso{\cH}}{\iso{I}} \tensor{A} \cH }
}
\]
and since the diagram commutes sequentially, we can conclude that the left-hand square is a pullback square and hence that $\cK \tensor{A} \cK = \big(\cK \tensor{A} \cH\big) \cap \big(\cH \tensor{A} \cK\big)$.
\end{proof}

\begin{theorem}\label{thm:forsecisiamo}
If $(A,\cH)$ is a commutative Hopf algebroid and $\sM{\cH}$ is $A$-flat, then we have a well-defined inclusion-preserving bijective correspondence
\[
\begin{gathered}
\xymatrix @R=0pt{
{\left\{ \begin{array}{c}  \text{normal Hopf ideals } I \text{ in } \cH \text{ such that} \\ \cH \text{ is pure over } \coinv{\cH}{\frac{\cH}{I}} \end{array} \right\}} \ar@{<->}[r] & {\left\{ \begin{array}{c} \text{sub-Hopf algebroids } \cK \subseteq \cH \text{ such that } \\ \cH \text{ is pure over } \cK  \end{array} \right\}} \\
I \ar@{|->}[r] & \coinv{\cH}{\frac{\cH}{I}} \\
\cH \cK^+ & \cK \ar@{|->}[l]
}
\end{gathered}
\]
\end{theorem}

\begin{proof}
It follows from Proposition \ref{prop:HopfNonGalois}, in view of Proposition \ref{prop:nHideal} and Proposition \ref{prop:subHalgd}.
\end{proof}

Recall that a Hopf algebroid $\cH$ for which the base algebra $A$ is a separable-Frobenius algebra is exactly a weak Hopf algebra. As in the above correspondence theorems the base algebra is the same for all considered sub-Hopf algebroids, these sub-Hopf algebroids are also weak sub-Hopf algebras. Examples of weak Hopf algebras arise naturally from finite groupoids.
As a reality check, we now reflect on what this bijection reduces to for a simple family of commutative Hopf algebroids coming from finite groupoids in an elementary way. In this finite setting, we recover the fact that the bijection reduces to the correspondence between normal subgroupoids and quotient groupoids. 

\begin{example}\label{ex:finitegroupoidexample} 
Let $\K$ be an algebraically closed field. Throughout this example we will assume that $\Gg = \big(\Gg_{0}, \Gg_{1}\big) $ is a \emph{finite} groupoid, meaning $\Gg_{0}$ and $\Gg_{1}$ are both finite sets. In this case we can define the commutative Hopf algebroid $\big(\K (\Gg_{0}), \K (\Gg_{1})\big)$ of functions on the groupoid $\Gg$ as follows. Let $A \coloneqq \K (\Gg_{0})$ denote the commutative algebra of functions on the set $\Gg_{0}$ which is spanned by elements of the form $f_{p}$ corresponding to $p\in  \Gg_{0}$ satisfying $f_{p}.f_{p'}= \delta_{p,p'}f_{p}$, while $\cH \coloneqq \K(\Gg_{1})$ denotes the algebra of functions on the set of morphisms $\Gg_{1}$ in a similar way. Hence, $\K(\Gg_{1})$ is generated by elements $f_{ {g}}$ corresponding to morphisms $ {g}\in \Gg_{1}$ and $f_{ {g_{1}}}. f_{ {g_{2}}}=\delta_{ {g_{1}}, {g_{2}}} f_{ {g_{1}}}$ holds for arbitrary $g_{1},g_{2}\in \Gg_{1}$. The Hopf algebroid structure is defined as follows: 
\begin{align*}
s (f_{p}) = \sum_{ {g\in \Gg_{1}}: \ \sigma( {g})=p} &f_{ {g}},\quad t (f_{p}) = \sum_{ {g\in \Gg_{1}}:\  \tau ( {g})=p} f_{ {g}},\quad \Delta ( f_{ {g}}) = \sum_{ {g_{1}}, {g_{2}}\in \Gg_{1}:\  {g_{2}}. {g_{1}}=  {g}} f_{ {g_{1}}}\tensor{A} f_{ {g_{2}}}
\\ \varepsilon ( f_{ {g}})& =\sum_{p\in \Gg_{0}} \delta_{ {g}, e({p})}f_{p}, \quad S\left( f_{ {g}}\right)= f_{ {g}^{-1}} ,\quad  1= \sum_{ {g}\in \Gg_{1}} f_{ {g}}
\end{align*}
Notice that, $A$ being semisimple, every $A$-module is projective and hence flat.

Let us first characterise the structures in the left hand side of the bijection in Theorem~\ref{thm:forsecisiamo}. In Example~3.6 of \cite{ghobadi2021isotopy} it was already shown that any Hopf ideal $I$ of $\K(\Gg_{1})$ must be of the form $\K(S_1)$ for a subset $S_1$ of $\Gg_{1}$ for which $(\Gg_{0},\Gg_{1} \setminus S_{1})$ forms a subgroupoid of $\Gg$. If $I$ is also normal, \ref{item:NI1} implies that 
$$\left\langle \sum_{ {g\in \Gg_{1}}: \sigma( {g})=p} f_{ {g}}-\sum_{ {h\in \Gg_{1}}: \tau( {h})=p} f_{ {h}}  ~\bigg\vert~ p\in\Gg_{0} \right\rangle= \left\langle f_{g} ~\big\vert~ g\notin \Gg_{1}^{(i)}\right\rangle \subseteq I=\K(S_{1})$$
and therefore $\Gg_{1} \setminus S_{1}\subseteq \Gg^{(i)}_{1}$. It is also easy to check that
\[\rcaction{\iso{\cH}}\big(\bara{f_{g}}\big)= \sum_{h\in \Gg_{1}: \tau(h)=\sigma(g)}\bara{f_{h^{-1}gh}}\tensor{A} f_{h}\] 
for any $\bara{f_{g}}\in I_{(i)}$. Hence, \ref{item:NI2} holds if and only if $\Gg_{1}^{(i)}\cap  S_{1}$ is closed under conjugation if and only if $ \Gg_{1}^{(i)}\setminus  S_{1}$ is closed under conjugation i.e. condition \ref{item:norm3} holds. Hence, normal Hopf ideals $I=\K (S_{1})$ of $\big(\K(\Gg_{0}), \K(\Gg_{1})\big)$ correspond to normal subgroupoids $\big(\Gg_{0}, \Gg_{1}\setminus S_{1}\big)$, as we know from Proposition \ref{prop:cociente}. 

In this situation, $\cH /I \cong \K(\Gg_{1}\setminus S_{1})$ and we denote $\Nn_{1} \coloneqq \Gg_{1}\setminus S_{1}$. Now consider an arbitrary term $\sum_{i=1}^{n}\alpha_{i}f_{g_{i}}$ in $\coinv{\cH}{\frac{\cH}{I}}$, where $n$ is a positive integer, $g_{i}\in \Gg_{1}$ and $\alpha_{g_{i}}\in \K$. Since $ \sum_{i=1}^{n}\alpha_{i}f_{g_{i}}\in \coinv{\cH}{\frac{\cH}{I}}$, we have an equality
\begin{align*}
\sum_{i=1}^{n}\sum_{h\in\Gg_{1}\setminus S_{1}:\ \tau (h)=\tau (g_{i})}\alpha_{i}f_{h^{-1}g_{i}}\tensor{A} \bara{f_{h}}=\sum_{i=1}^{n}\alpha_{i}f_{g_{i}}\tensor{A} \bara{1}= \sum_{i=1}^{n}\alpha_{i}f_{g_{i}}\tensor{A} \left(\sum_{g\in\Gg_{1}\setminus S_{1}:\ \sigma (g)=\tau (g_{i})}\bara{ f_{g}}\right)
\end{align*}
Noting that $\cH\tensor{A}\frac{\cH}{I} \cong \bigoplus_{(g,h)\in \Gg_{1} {}_\tau \times_{\sigma} \Nn_{1} } \K (f_{g}\tensor{A} \bara{f_{h}})$ as a vector space and that $h\in\Gg_{1} \setminus S_{1}\subseteq \Gg^{(i)}_{1}$ have the same source and target, we obtain a family of equivalent equalities
\[\sum_{i:\  \tau (g_{i})= \tau (h)}\alpha_{i} f_{g_{i}}= \sum_{j:\ \tau (g_{j})= \tau (h)} \alpha_{j} f_{h^{-1}g_{j}}\] 
for any $h\in \Gg_{1}\setminus S_{1}$. Recall the natural action of $\Gg_{1}\setminus S_{1}$ on $\Gg_{1}$ from \eqref{Eq:Naction}. If we denote the orbits of this action by $\Oo_{1}, \ldots, \Oo_{m}$, then these equalities holding, means any arbitrary element in $\sum_{i=1}^{n}\alpha_{i}f_{g_{i}}$ in $\coinv{\cH}{\frac{\cH}{I}}$ can be re-written as $\sum_{j=1}^{m} \alpha_{i_{j}}(\sum_{g\in \Oo_{j}} f_{g})$ for some subset of indices $\lbrace i_{j}\rbrace_{j=1}^{m}$ of $\lbrace i\rbrace _{i=1}^{n}$. 

Consequently, any normal Hopf ideal $I=\K (S_{1})$ of $\big(\K (\Gg_{0}), \K (\Gg_{1})\big)$ corresponds to the choice of a normal subgroupoid $\big(\Gg_{0},\Nn_{1} \coloneqq \Gg_{1}\setminus S_{1}\big)$ and $\coinv{\cH}{\frac{\cH}{I}}\cong \K \big( ({\Gg}/{\Nn})_{1}\big)$. In particular, since $\K \big( ({\Gg}/{\Nn})_{1}\big)$ is semisimple, any module (such as $\cH$) over it is flat and thereby any ring extension over it is faithfully flat, hence pure.

Having identified the objects on the left hand side of the bijection in Theorem~\ref{thm:forsecisiamo}, we observe that the theorem tells us that any Hopf subalgerboid $\cK$ of $\cH = \K(\Gg_{1})$ for which $\cH$ is pure over $\cK$ is of the form $\K \big( ({\Gg}/{\Nn})_{1}\big)$ for some quotient groupoid $\Gg / \Nn$. We can observe this directly by looking at the set of characters $\calg_\K(\cK,\K)$ on $\cK$. Since $\calg_\K({\cH},{\K}) = \Gg_{1} $, we have a natural map $\pi \colon \Gg_{1}\rightarrow \calg_\K(\cK,\K) $ coming from the inclusion $\iota \colon \cK\rightarrow \cH$. Since $\iota$ is pure and since elements from $\calg_\K(\cK,\K)$ are in bijection with the maximal ideals in $\cK$, Proposition \ref{prop:equivpure} entails that $\pi$ is surjective. Note that $\calg_\K(\cK,\K) $ obtains a natural groupoid structure over $\Gg_{0}=\calg_\K(A,\K) $ via the structure maps and thereby $\pi$ becomes a morphism of groupoids. It is then easy to see that we are in the situation of Lemma~\ref{lema:quotient}.
\end{example}

\begin{proposition}\label{prop:serve?}
Let $(A, \cH)$ be a commutative Hopf algebroid and let $(A,\cK)$ be a sub-Hopf algebroid of $(A,\cH)$. Denote by $\phi\colon \cK \to \cH$ the inclusion. Then the canonical morphism
\begin{equation}\label{eq:PsiRker}
\Psi_R \colon {\displaystyle\frac{\Gg_\cH(R)}{\Gg_{\cH/\cH\cK^+}(R)}} \to \Gg_\cK(R), \qquad \Gg_{\cH/\cH\cK^+}(R)\bullet g \mapsto \Phi_R(g),
\end{equation}
induced by $\phi$ is injective for every $R \in \calg_\K$. That is to say, the kernel of the $R$-component $\Phi_R \colon \calg_\K(\cH,R) \to \calg_\K(\cK,R)$ of the canonical morphism induced by $\phi$ is $\calg_\K(\cH/\cH \cK^+,R)$.
\end{proposition}

\begin{proof}
Recall from \cite[Definition 2.6]{Mackenzie-old} that the kernel of $\Phi_R$ is the set
\[\Big\{f \in \calg_\K(\cH,R) \mid f \circ \phi = x\circ \varepsilon_\cK \text{ for some } x \in \calg_\K(A,R)\Big\}\]
and that it is a normal subgroupoid of $\big(\Gg_A(R),\Gg_\cH(R)\big)$. Now, it is clear that if $h \colon \cH/\cH\cK^+ \to R$ is an algebra morphism, then $h \circ \pi \in \ker(\Phi_R)$: for every $x \in \cK$ we have
\[\big((h \circ \pi) \circ \phi\big)(x) = (h \circ \pi \circ \phi \circ s_\cK \circ \varepsilon_\cK)(x) = \big((h \circ \pi \circ s_\cH) \circ \varepsilon_\cK\big)(x).\]
Conversely, if $f \in \Gg_\cH(R)$ is such that $f \circ \phi = x \circ \varepsilon_\cK$, then $f(\cH\cK^+) \subseteq f(\cH)f(\cK^+) = 0$ and hence there exists a unique $h \in \calg_\K(\cH/\cH\cK^+,R)$ such that $h \circ \pi = f$.
\end{proof}

\begin{corollary}\label{cor:injectivity}
Let $(A, \cH)$ be a commutative Hopf algebroid such that $\sM{\cH}$ is $A$-flat and let $I\subseteq \cH$ be a normal Hopf ideal such that $\cH$ is pure over $\rcoinv{\cH}{I}$. Denote by $\phi \colon \rcoinv{\cH}{I} \to \cH$ the inclusion.
Then the canonical morphism
\[
\Psi_R \colon {\displaystyle\frac{\Gg_\cH(R)}{\Gg_{\cH/I}(R)}} \to \calg_\K(\rcoinv{\cH}{I}, R), \qquad \Gg_{\cH/I}(R)\bullet g \mapsto \Phi_R(g),
\]
of Lemma \ref{lema:JHJ} is injective for every $R \in \calg_\K$. That is to say, the kernel of the morphism $\Phi_R \colon \calg_\K(\cH,R) \to \calg_\K(\rcoinv{\cH}{I}, R)$ induced by $\phi$ is exactly $\calg_\K(\cH/I,R)$.
\end{corollary}

\begin{proof}
It follows from Proposition \ref{prop:serve?} in view of Theorem \ref{thm:forsecisiamo}, since if $\cK \coloneqq \rcoinv{\cH}{I}$ then $I = \cH\cK^+$.
\end{proof}

\begin{remark}
Recall the following well-known fact. Suppose that $(A,\cK)$ is a sub-Hopf algebroid of a commutative Hopf algebroid $(A,\cH)$ such that $\cH_\cK$ is pure (resp. faithfully flat) and denote by $\phi\colon \cK \to \cH$ the inclusion. For every commutative $\K$-algebra $R$ and for every $g \in \Gg_\cK(R)$, the pushout diagram
\[
\xymatrix{
\cK \ar[r]^-{\phi} \ar[d]_-{g} & \cH \ar[d]^-{\tilde{g}} \\
R \ar[r]_-{\alpha} & \cH \tensor{\cK} R
}
\]
in the category of commutative algebras provides for us a pure (resp. faithfully flat) extension $S \coloneqq \cH \tensor{\cK} R$ of $R$ and a point $\tilde{g} \in \Gg_\cH(S)$ such that $g_S = \Phi_S(\tilde{g})$, where $g_S \coloneqq \alpha \circ g$ is the image of $g$ in $\Gg_\cK(S)$.
This means that even if the $\Psi \colon \Gg_\cH/\Gg_{\cH/\cH\cK^+} \to \Gg_\cK$ whose $R$-component is retrieved in \eqref{eq:PsiRker}, may not be surjective for every $R$, it is still ``locally surjective'' in the sense that any element of $\Gg_\cK(R)$ arises by patching, in some covering, elements which come from $\Gg_\cH$.
This also raises the important question of determining if and when a commutative Hopf algebroid is pure (resp. faithfully flat) over any sub-Hopf algebroid.
\end{remark}

Lemma \ref{lem:ffpure} can also be used to prove the subsequent Proposition \ref{prop:Takeuchi3.12}, which is the natural extension of \cite[Corollary 3.12]{Takeuchi:1972} to the present setting.

\begin{proposition}\label{prop:Takeuchi3.12}
Suppose that $\K$ is an algebraically closed field. Let $(A,\cK)$ be a sub-Hopf algebroid of a commutative Hopf algebroid $(A,\cH)$ such that $\cH$ is pure over $\cK$. Denote by $\phi\colon \cK \to \cH$ the inclusion. Then the $\K$-component $\Phi_\K \colon \calg_\K(\cH,\K) \to \calg_\K(\cK,\K)$ of the canonical morphism induced by $\phi$ is surjective.
\end{proposition}

\begin{proof}
If $\cH$ is finitely generated as a $\K$-algebra, then the statement follows from Hilbert's Nullstellensatz and Proposition \ref{prop:equivpure}, point \ref{item:equivpure7}.

In the general case let us proceed as follows. Observe that $\cH$, as an algebra, is the colimit of its subalgebras of the form $\cL\cK$, where $\cL$ is a finitely generated subalgebra.

Now, take $f \in \calg_\K(\cK,\K)$. Being an algebra map, $f$ allows us to see $\K$ as a $\cK$-algebra and we may consider the finitely generated $\K$-algebra $\cL\cK \tensor{\cK} \K$.

Since the chain of inclusions $\cK \subseteq \cL\cK \subseteq \cH$ is composed by $\cK$-ring maps and $\cH$ is pure over $\cK$, we may apply Corollary \ref{cor:pure} to claim that $\cK \subseteq \cL\cK$ is pure and hence the inclusion 
\[\K \cong \cK \tensor{\cK} \K \subseteq \cL\cK \tensor{\cK} \K\] 
entails that $\cL\cK \tensor{\cK} \K$ is non-zero.
By the Nullstellensatz again, $\calg_\K\left(\cL\cK \tensor{\cK} \K,\K\right)$ is non-empty and hence we may take a $g$ therein. The composition
\[g'\colon \left(\cL\cK \xrightarrow{x \mapsto x \tensor{\cK} 1_\K} \cL\cK \tensor{\cK} \K \xrightarrow{g} \K\right)\]
satisfies $g'\left(x\right) = g\left(x \tensor{\cK} 1 \right) = g \left(1 \tensor{\cK} f(x)\right) = f(x)$ for all $x \in \cK$ and hence any $f \in \calg_\K(\cK,\K)$ can be extended to a $g' \in \calg_\K(\cL\cK,\K)$ for every finitely generated subalgebra $\cL$ of $\cH$. Therefore, $\Phi_\K \colon \calg_\K(\cH,\K) \to \calg_\K(\cK,\K)$ is surjective.
\end{proof}

\begin{corollary}\label{cor:Takeuchi3.12}
Suppose that $\K$ is an algebraically closed field. Let $I$ be a normal Hopf ideal of the commutative Hopf algebroid $(A,\cH)$ such that $\sM{\cH}$ is $A$-flat and $\cH$ is pure over $\coinv{\cH}{\frac{\cH}{I}}$. Then the $\K$-component $\Psi_\K$ of the canonical morphism \eqref{eq:PsiR} from Lemma \ref{lema:JHJ} is an isomorphism. That is to say,
\[\Gg_\cH(\K)/\Gg_{\cH/I}(\K) \ \cong \ \Gg_{\rcoinv{\cH}{I}}(\K).\]
\end{corollary}

\begin{proof}
Surjectivity follows by Proposition \ref{prop:Takeuchi3.12}, in view of the fact that $\coinv{\cH}{\frac{\cH}{I}} = \rcoinv{\cH}{I}$ is a sub-Hopf algebroid of $\cH$ by Proposition \ref{prop:subHalgd}. Injectivity follows from Corollary \ref{cor:injectivity}.
\end{proof}

\begin{corollary}\label{cor:serve?}
Suppose that $\K$ is an algebraically closed field. Let $(A,\cK)$ be a sub-Hopf algebroid of the commutative Hopf algebroid $(A,\cH)$ such that $\cH$ is pure over $\cK$. Denote by $\phi\colon \cK \to \cH$ the inclusion. Then the kernel of the $\K$-component $\Phi_\K \colon \calg_\K(\cH,\K) \to \calg_\K(\cK,\K)$ of the canonical morphism induced by $\phi$ is $\calg_\K(\cH/\cH \cK^+,\K)$. That is to say
\[\Gg_\cH(\K)/\Gg_{\cH/\cH\cK^+}(\K) \ \cong \ \Gg_{\cK}(\K).\]
\end{corollary}

\begin{proof}
Follows directly from Proposition \ref{prop:serve?} and Proposition \ref{prop:Takeuchi3.12}.
\end{proof}

\bigskip

\textbf{Acknowledgements}  
Paolo Saracco is a Charg\'e de Recherches of the Fonds de la Recherche Scientifique - FNRS and a member of the ``National Group for Algebraic and Geometric Structures and their Applications'' (GNSAGA-INdAM). Aryan Ghobadi is a postdoctoral researcher under the EPSRC grant EP/W522508/1 and would also like to thank the LMS for the travel grant ECR-1920-42 which allowed the author to be included in this project, as well as the University of Granada for their hospitality during his visit. L.~EL Kaoutit  would like to thank Leovigilio  Alonso Tarrío and Ana Jeremías López for helpful discussions.

\end{document}